\pdfoutput=1
\documentclass[11pt]{amsart}
\title[Homological mirror symmetry for open Riemann surfaces]{Homological mirror symmetry for open Riemann surfaces from pair-of-pants
decompositions}
\author{Heather Lee}
\date{}
\email{heatherlee@berkeley.edu}
\usepackage{pdfsync}
\usepackage{graphicx}
\usepackage{amsmath, amscd}
\usepackage{amssymb} 
\usepackage{amsthm}
\usepackage{cancel}
\usepackage{geometry}
\usepackage[all]{xy}
\usepackage{mathrsfs}

\numberwithin{equation}{section}

\newtheorem{theorem}{Theorem}[section]
\newtheorem{lemma}[theorem]{Lemma} 

\newtheorem{corollary}[theorem]{Corollary}
\theoremstyle{definition}

\newtheorem{remark}[theorem]{Remark}

\DeclareMathOperator{\Coh}{Coh}
\DeclareMathOperator{\Coker}{coker}
\DeclareMathOperator{\Conv}{Conv}
\DeclareMathOperator{\Crit}{Crit}

\DeclareMathOperator{\Ext}{Ext}
\DeclareMathOperator{\Hom}{Hom}
\newcommand{\sHom}{\mathcal Hom}

\DeclareMathOperator{\interior}{int}

\DeclareMathOperator{\ind}{ind}
\DeclareMathOperator{\Log}{Log}

\newcommand{\bC}{\mathbb C}

\newcommand{\Os}{\mathcal O}

\newcommand{\cA}{\mathcal A}
\newcommand{\cB}{\mathcal B}
\newcommand{\cC}{\mathcal C}
\newcommand{\cD}{\mathcal D}
\newcommand{\cE}{\mathcal E}
\newcommand{\cF}{\mathcal F}
\newcommand{\cG}{\mathcal G}
\newcommand{\cL}{\mathcal L}
\newcommand{\cJ}{\mathcal J}
\newcommand{\cK}{\mathcal K}
\newcommand{\cQ}{\mathcal Q}
\newcommand{\cP}{\mathcal P}
\newcommand{\cT}{\mathcal T}

\newcommand{\nufun}{fun}

\newcommand{\tcE}{\tilde{\mathcal E}}
\newcommand{\tcC}{\tilde{\mathcal C}}

\begin{document}

\begin{abstract} 
Given a punctured Riemann surface with a pair-of-pants decomposition, we compute its wrapped Fukaya category in a suitable model by reconstructing it from those of various pairs of pants.  The pieces are glued together in the sense that the restrictions of the wrapped Floer complexes from two adjacent pairs of pants to their adjoining cylindrical piece agree.  The $A_\infty$-structures are given by those in the pairs of pants.  The category of singularities of the mirror Landau-Ginzburg model can also be constructed in the same way from local affine pieces that are mirrors of the pairs of pants.
\end{abstract}

\maketitle
\tableofcontents

\section{Introduction}

Mirror symmetry is a duality between symplectic and complex geometries, and the homological mirror symmetry (HMS) conjecture was formulated by Kontsevich \cite{Ko94} to capture the phenomenon by relating two triangulated categories.   This first formulation of HMS is for pairs of Calabi-Yau manifolds $(X,X^{\vee})$ and it predicts two equivalences: the derived Fukaya category of $X$ (which depends only on its symplectic structure) is equivalent to the bounded derived category of coherent sheaves of $X^{\vee}$ (which depends only on its complex structure), and the bounded derived category of coherent sheaves of $X$ is equivalent to the derived Fukaya category of $X^{\vee}$.

A non-Calabi-Yau manifold $X$ can also belong to a mirror pair $(X,(X^\vee, W))$, where $(X^{\vee},W)$ is a Landau-Ginzburg model  consisting of a non-compact manifold $X^{\vee}$ and a holomorphic function $W:X^\vee \to \mathbb C$ called the superpotential.  HMS has been extended to cover Fano manifolds by Kontsevich \cite{Ko98} based on works by Batyrev \cite{Ba94}, Givental \cite{Gi96}, Hori-Vafa \cite{HV00}, and others, and more recently to cover general type manifolds \cite{Ka07, KKOY09}. The complex side of $(X^{\vee},W)$ is described by Orlov's triangulated category of singularities of the singular fiber $W^{-1}(0)$, or equivalently the category of matrix factorizations $MF(X^\vee, W)$ \cite{Or04}.  The symplectic side of $(X^{\vee},W)$ is described by the derived Fukaya-Seidel category \cite{Se08} of Lagrangian vanishing cycles associated with $W$.

Another recent discovery is that, in the case where $X$ is an open manifold, the symplectic side of $X$ needs to be described by its wrapped Fukaya category.  (Similarly, when the fibers of $W:X^\vee \to \mathbb C$ are open, the symplectic side of $(X^\vee, W)$ is determined by its fiberwise wrapped Fukaya category \cite{AA}.)  The wrapped Fukaya $A_\infty$-category is an extension of the Fukaya category constructed by Abouzaid and Seidel \cite{AS10} for a large class of non-compact symplectic manifolds known as Liouville manifolds. Examples of such manifolds include cotangent bundles, complex affine algebraic manifolds, and many more general properly embedded submanifolds of $\mathbb C^n$. A Liouville manifold $X$ is equipped with a Liouville $1$-form $\lambda$, has an exact symplectic form $\omega=d\lambda$ and a complete Liouville vector field $Z$ determined by $\iota_Z\omega=\lambda$, and satisfies a convexity condition at infinity. 

In general, the wrapped Fukaya category of a Liouville manifold, $\mathcal W(X)$, is very hard to compute; therefore, it is of much interest to develop sheaf-theoretic techniques to compute $\mathcal W(X)$ by first decomposing $X$ into simpler standard pieces $X=\bigcup_{i\in I} S_i$ and then reconstructing $\mathcal W(X)$ from the wrapped Fukaya categories of the standard pieces.  

Inspired by Viterbo's restriction idea for symplectic cohomology \cite{Vi99}, Abouzaid and Seidel \cite{AS10} constructed an $A_\infty$-restriction functor from a quasi-isomorphic full-subcategory of $\mathcal W(X)$ to $\mathcal W(N)$ for every Liouville subdomain $N$ of $X$ (A Liouville subdomain is a codimension-0 compact submanifold of $X$ whose boundary is transverse to the flow of the Liouville vector field $Z$.  Its completion is a Liouville manifold.)  Suppose $X=S_1\cup S_2$ can be decomposed into two standard Liouville subdomains, then we can hope to compute $\mathcal W(X)$ from $\mathcal W(S_i), i=1,2$, by gluing them along $S_1\cap S_2$ in the sense of matching the images of the restriction functors  $\rho_i: \mathcal W(S_i)\to \mathcal W(S_1\cap S_2)$.  However this procedure cannot be readily implemented due to two obstacles.  First, many pseudo-holomorphic discs that contribute to the $A_\infty$-structures of $\mathcal W(X)$ are not contained in any single $S_i$.  Second, in general it is not always possible to equip $X=\bigcup_{i\in I} S_i$ with a single Liouville structure such that all $S_i$'s are Liouville subdomains.  In addition, the restriction functors could in general have higher order terms which could make the computation intractable.  

We focus on punctured Riemann surfaces that have decompositions into standard pieces which are pairs of pants.  A pair of pants is a sphere with three punctures; its wrapped Fukaya category is computed in \cite{AAEKO13}.  The intersection between two adjacent pairs of pants is a cylinder.   In Section 3, we provide a suitable model for the wrapped Fukaya category of such a punctured Riemann surface and compute it by providing an explicit way to glue together the wrapped Fukaya categories of the pairs of pants.   Thus, our results achieve something very close to the picture conjectured by Seidel \cite{Se12}. (Other instances of sheaf-theoretic computation methods include calculations for cotangent bundles \cite{FO97, NZ09}, and a program proposed by Kontsevich \cite{Ko09} to compute the Fukaya categories in terms of the topology of a Lagrangian skeleton on which they are conjectured to be local.  See also \cite{Na14}, \cite{Ab14}, \cite{Dy15}, and others for recent developments.)

 In Section 4, we show that the category of matrix factorizations $MF(X^\vee, W)$ of the toric Landau-Ginzburg mirror can also be constructed in the same manner from a \v{C}ech cover of $(X^\vee, W)$ by local affine pieces that are mirrors of the various pairs of pants.  We will demonstrate that the restriction from the category of matrix factorizations on an affine piece to that of the overlap with an adjacent piece is homotopic to the corresponding restriction functor for the wrapped Fukaya categories.  In turn, we prove the HMS conjecture that the wrapped Fukaya category of a punctured Riemann surface is equivalent to $MF(X^\vee, W)$; in fact, HMS serves as our guide in developing this sheaf-theoretic method for computing the wrapped Fukaya category.  The HMS conjecture with the A-model being the the wrapped Fukaya category of a punctured Riemann surface has been proved for punctured spheres
and their multiple covers in \cite{AAEKO13}, and punctured Riemann surfaces in \cite{Bo}.  However, our approach yields a new proof that is in some sense more natural, and the main benefit of this approach is that one can hope to extend it to higher dimensions.

\subsection*{Acknowledgement} I thank my thesis adviser Denis Auroux for suggesting this project,  sharing many ideas, and providing me with lots of guidance throughout this work.  I thank Mohammed Abouzaid for his very important suggestions, in particular his insights regarding the restriction functors.  I thank Sheel Ganatra, Yanki Lekili, David Nadler, Daniel Pomerleano, Anatoly Preygel, Paul Seidel, Vivek Shende, Nick Sheridan, Zack Sylvan, and Umut Varolgunes for their very helpful explanations and suggestions, and for the valuable discussions that helped me understand my result better.   This work was partially supported by the NSF grant DMS-1007177.

\section {Review of hypersurfaces, tropical geometry, and mirror symmetry}  \label{sec: tropical}

Let $H$ be a punctured Riemann surface with a pair-of-pants decomposition.  We will focus on the case where $H$ is a hypersurface in $(\mathbb C^*)^2$ that is near the tropical limit, in which case $H$ always has a preferred pair-of-pants decomposition.  Mikhalkin \cite{Mi04} used ideas from tropical geometry to decompose hypersurfaces in projective toric varieties into higher dimensional pairs of pants.  We decompose $H$ into pairs of pants in his style because it is natural for mirror symmetry \cite{Ab06, AAK12} and we hope to generalize our results to hypersurfaces in $(\mathbb C^*)^n$.  In this section, we summarize this decomposition procedure as explained in \cite{Mi04, Ab06, AAK12}.

Consider a family of hypersurfaces 
\begin{equation} \label{eq:f}
H_t=\left\{f_t:= \sum_{\alpha\in A} c_\alpha t^{-\rho(\alpha)}z^\alpha=0\right\}\subset (\mathbb C^*)^2, \ \ \ \ t\to \infty,
\end{equation}
where $c_\alpha\in \mathbb C^*$, $z=(z_1,z_2)\in (\mathbb C^*)^2$, $A$ is a finite subset of $\mathbb Z^2$, $z^\alpha=z_1^{\alpha_1}z_2^{\alpha_2}$, and $t\in \mathbb R_{>0}$.  The function $\rho:A\to \mathbb R$ is the restriction to $A$ of a convex piecewise linear function $\bar \rho: \Conv(A)\to \mathbb R$. 

The family of hypersurfaces $H_t$ has a maximal degeneration for $t\to \infty$ if the maximal domains of linearity of $\bar\rho: \Conv(A)\to \mathbb R$ are exactly the cells of a lattice polyhedral decomposition $\mathcal P$ of the convex hull $\Conv(A)\subset \mathbb R^2$, such that the set of vertices of $\mathcal P$ is exactly $A$ and every cell of $\mathcal P$ is congruent to a standard simplex under $GL(2,\mathbb Z)$ action.

The logarithm map $\Log_t:(\mathbb C^*)^2\to \mathbb R^2$ is defined as $\Log_t(z)=\frac{1}{|\log t|} (\log|z_1|,\log|z_2|)$.  Due to \cite{Mi04, Ru01}, as $t\to \infty$, Log-amoebas $\mathcal A_{t}:=\Log_{t}(H_t)$
 converge in the Gromov-Hausdorff metric to the tropical amoeba $\Pi$, a $1$-dimensional polyhedral complex which is the singular locus of the Legendre transform
\begin{equation} \label{eq:legendre} 
L_\rho(\xi)=\max\{\langle\alpha,\xi\rangle-\rho(\alpha)|\alpha\in A\}.
\end{equation}
An edge of $\Pi$ is where two linear functions from the collection $\{\langle\alpha,\xi\rangle-\rho(\alpha)|\alpha\in A\}$ agree and a vertex is where three linear functions agree.  In fact, $\Pi$ is combinatorially the $1$-skeleton of the dual cell complex of $\mathcal P$, and we can label the components of $\mathbb R^2-\Pi$ by elements of $A$, which are vertices of $\mathcal P$, i.e. 
\begin{equation}
\mathbb R^2-\Pi=\bigsqcup_{\alpha\in A} C_\alpha-\partial C_\alpha.
\end{equation}

We identify $(\mathbb C^*)^2$ with the cotangent bundle of $\mathbb R^2$ with each cotangent fiber quotiented by $(2\pi/|\log t|)\mathbb Z^2$, via $(\mathbb C^*)^2 \cong T^*\mathbb R^2/(2\pi /|\log t|)\mathbb Z^2\cong \mathbb R^2\times (S^1)^{2}$ given by 
\begin{equation} \label{eq: cotangent identification}
z_j=t^{u_j+i\theta_j}, \ j=1,2 \ \ \ \ \ \mapsto \ \ \ (u_j, \theta_j)=\frac{1}{|\log t|}(\log |z_j|,\arg(z_j)).\end{equation}  
This gives a symplectic form on $(\mathbb C^*)^2$ 
\begin{equation}
\omega_t=\frac{i}{2|\log t|^2}\sum_{j=1}^2 d\log z_j\wedge d\log \bar {z}_j=\sum_{j=1}^2 d u_j\wedge d \theta_j,
\end{equation}  
which is invariant under the $(S^1)^{2}$ action with the moment map $(u_1, u_2)=\frac{1}{|\log t|}(\log|z_1|,\log|z_2|)$.

We study a hypersurface $H=f_t^{-1}(0)$ (fix $t\gg 1$) that is near the tropical limit, meaning that it is a member of a maximally degenerating family of hypersurfaces as above, with the property that the amoeba $\mathcal A = \Log_{t}(H) \subset \mathbb R^2$ is entirely contained inside an $\varepsilon$-neighborhood of the tropical hypersurface $\Pi$ which retracts onto $\Pi$, for a small $\varepsilon$.  Then each open component $C_{\alpha, t}$ of $\mathbb R^2-\Log_{t}(H)$ is approximately $C_\alpha-\partial C_\alpha$ as $\partial C_{\alpha, t}$ is contained in an $\varepsilon$-neighborhood of $\partial C_\alpha$.  The monomial 
$c_\alpha t^{-\rho(\alpha)}z^\alpha$ dominates on $\Log_{t}^{-1}(C_{\alpha, t})$.

For SYZ mirror symmetry \cite{SYZ96}, the mirror to $H$ is shown in  \cite{AAK12} to be a Landau-Ginzburg model $(Y, W)$, where $Y$ is a toric variety with its moment polytope being the noncompact polyhedron 
\begin{equation}
\Delta_Y=\{(\xi, \zeta)\in \mathbb R^2\times \mathbb R|\zeta\geq L_\rho(\xi)\}.
\end{equation}
The superpotential $W:Y\to \mathbb C$ is the toric monomial of weight $(0,0,1)$; it vanishes with multiplicity $1$ exactly on the singular fiber $D=W^{-1}(0)$ that is a disjoint union $D=\coprod_{\alpha\in A} D_{\alpha}$ of irreducible toric divisors of $Y$.  Each irreducible toric divisor $D_{\alpha}$ corresponds to a facet of $\Delta_Y$, which corresponds to a connected component of $\mathbb R^2-\Pi$, and $\Crit(W)$ is a union of $\mathbb P^1$'s and $\mathbb C^1$'s corresponding to bounded and unbounded edges of $\Pi$, respectively.   In fact, $(Y,W)$ is equivalent to the mirror of the blow up of $(\mathbb C^*)^2\times \mathbb C$ along $H\times \{0\}$.  

We would like to demonstrate homological mirror symmetry that is the equivalence between the wrapped Fukaya category of $H$ and the triangulated category of the singularities $D_{sg}^b(D)$, which is defined as the Verdier quotient of the bounded derived category of coherent sheaves $D^b(\Coh(D))$ by the subcategory of perfect complexes $\mathfrak{Perf}(D)$.  The category $D_{sg}^b(D)$ is equivalent to the triangulated category of matrix factorizations, $MF(Y,W)$ \cite{Or11}.

\section{The wrapped Fukaya category of the punctured Riemann surface $H$} \label{sec: wFk}

\subsection {Generating Lagrangians} \label{sec: lagrangian}  In this section, we describe a set of Lagrangians that split-generates the wrapped Fukaya category $\mathcal W(H)$.  

There is a projection $\pi: H\to \Pi$ from the Riemann surface to the tropical amoeba such that $\pi$ is a circle fibration over the complement $\Pi\backslash\{\text{vertices}\}$.  This complement consists of open edges, each of which is of the form $\interior(C_\alpha \cap C_\beta)$, whose preimage in $H$ is an open cylinder denoted by
\begin{equation}
 \tilde e_{\alpha\beta}=\pi^{-1}(\interior C_\alpha\cap C_\beta)\cong \interior(C_\alpha\cap C_\beta)\times S^1.
\end{equation}  
The preimage of $\pi$ over each tripod graph at a vertex of $\Pi$ is a pair of pants.

 Recall the defining equation for the Riemann surface $H$ is 
\[
\begin{array}{ll}
0=f_t (z)&= \sum_{\gamma\in A} c_\gamma t^{-\rho(\gamma)}z^\gamma \\
& = c_{\alpha}t^{-\rho(\alpha)}z^{\alpha}\left(1+\frac{c_\beta}{c_\alpha}t^{\rho(\alpha)-\rho(\beta)}z^{\beta-\alpha}+\sum_{\gamma\in A\backslash\{\alpha,\beta\}}\frac{c_\gamma}{c_\alpha} t^{\rho(\alpha)-\rho(\gamma)}z^{\gamma-\alpha}\right),
\end{array}
\]
 for some particular $\alpha$ and $\beta$.  Near the edge $C_\alpha\cap C_\beta$, 
\[
|t^{\rho(\alpha)-\rho(\gamma)}z^{\gamma-\alpha}|=t^{\rho(\alpha)-\rho(\gamma)}t^{\langle \Log_t z,\gamma-\alpha\rangle}=\frac{t^{\langle \Log_t z,\gamma\rangle-\rho(\gamma)}}{t^{\langle \Log_t z,\alpha\rangle-\rho(\alpha)}}
\]
 is very small when $t$ is very large.  Hence $\tilde e_{\alpha\beta}$ in the Riemann surface lies close to the cylinder with the defining equation 
\begin{equation}\label{eqn: edge}
c_\alpha t^{-\rho(\alpha)}z^\alpha+c_\beta t^{-\rho(\beta)} z^\beta=0.
\end{equation}
Since this equation can be written in the form $\frac{c_\alpha}{c_\beta} z^{\alpha-\beta}=-t^{\rho(\alpha)-\rho(\beta)} <0$,  $\tilde e_{\alpha\beta}$ lies close to the cylinder $e_{\alpha\beta}^\Pi$ given by the complexification of $\interior(C_\alpha\cap C_\beta)$ with the argument defined by 
\begin{equation}\label{eqn: edge def} 
 e_{\alpha\beta}^{\Pi}:=\left\{z\in (\mathbb C^*)^2 \left|\right. \Log_t(z)\in \interior(C_\alpha\cap C_\beta), \arg \left(\frac{c_\alpha}{c_\beta}z^{\alpha-\beta}\right) \equiv\pi\right\}.
\end{equation}
In Section 4 of \cite{Ab06}, it is shown that $f_t^{-1}(0)$ are symplectomorphic for all $t$ and that they are symplectomorphic to a piecewise smooth symplectic hypersurface $H_\Pi$ of $(\mathbb C^*)^2$ which projects to $\Pi$.    Each $e_{\alpha\beta}^\Pi$ is a smooth subset of $H_\Pi$ and symplectomorphic to $\tilde e_{\alpha\beta}$. 

When focusing on each bounded edge of $\Pi$, e.g. $\interior(C_\alpha\cap C_\beta)$, we parametrize it by the interval $\{\tau\in (-\epsilon, 4+\epsilon)\}$ for $\epsilon \ll 1$, so $\tilde e_{\alpha\beta}\cong (-\epsilon, 4+\epsilon)\times S^1$.  From now on, we will leave out a small neighborhood around each vertex of $\Pi$ instead of taking the entire $\interior (C_\alpha\cap C_\beta)\times S^1$, and we define the \textbf{edge} $e_{\alpha\beta}$ of the Riemann surface $H$ to be the subset  of $\tilde e_{\alpha\beta}$ corresponding to $\{(\tau,\psi) \in (0,4)\times S^1\}$, with the symplectic form on it being
\begin{equation} \label{eqn: edge symplectic form}
\omega_{\alpha\beta}={c_{\alpha\beta}}d\tau\wedge d\psi.
\end{equation}
The constant $c_{\alpha\beta}$ is determined by the symplectic area of the edge $e_{\alpha\beta}$.  For each bounded edge $e_{\alpha\beta}$, we denote by 
\begin{equation}
n_{\alpha\beta}=|(C_\alpha\cap C_\beta)\cap \mathbb Z^2|-1
\end{equation}
 the integer that is the lattice length of that edge.  Also, denote  
 \begin{equation}
 d_{\alpha,\beta}=\deg \Os(D_{\alpha})|_{D_\alpha\cap D_\beta}.
 \end{equation}
  Choose integers $\delta_{\alpha,\beta}$ and $\delta_{\beta,\alpha}$ satisfying 
  \begin{equation}
  \delta_{\beta,\alpha}-\delta_{\alpha,\beta}=1+d_{\alpha,\beta};
  \end{equation}
 this is well defined since the Calabi-Yau property implies that $d_{\alpha, \beta}+d_{\beta, \alpha}=-2$.

We proceed similarly for unbounded edges of $\Pi$, using $(\tau,\psi) \in (0,\infty)\times S^1$ instead of $(0,4)\times S^1$ (we still use the notation $(\tau,\psi)$ on unbounded edges for convenience).  Each unbounded edge corresponds to a neighborhood of a puncture.  A punctured Riemann surface can be given a Liouville manifold structure  by modeling the neighborhood of each puncture after a cylindrical end $([1,\infty)\times S^1, \tau d\psi)$.

 The Lagrangian objects under consideration are $L_\alpha(k)$, for $\alpha\in A$ and $k\in \mathbb Z$. Each $L_\alpha(k)$ is an embedded curve in $H$ that follows the contour of $C_{\alpha}$ and runs through all edges $e_{\alpha\beta}$, for each $\beta\in A$ such that $C_\beta$ is adjacent to $C_\alpha$.  We require $L_\alpha(k)$ to be invariant under the Liouville flow everywhere in the cylindrical end.

Let $L_{\alpha}(0)$ be a Lagrangian that runs through each edge ``straight'' in a prescribed manner, e.g. with constant $\psi$, also see Remark \ref{rmk: ambience} for an example.  Each $L_\alpha(k)$ differs from $L_\alpha(0)$  in each bounded cylindrical edge $e_{\alpha\beta}\cong (0,4)\times S^1$ by a rotation of $k_{\alpha\beta}$ times in the ``negative'' direction, where $k_{\alpha\beta}=n_{\alpha\beta}k+\delta_{\alpha,\beta}$.   More precisely, $L_\alpha(k)$ differs from $L_\alpha(0)$ in $e_{\alpha\beta}$ by the time-$(k_{\alpha\beta}/c_{\alpha\beta})$ flow of the vector field generated by a Hamiltonian $h_{\alpha\beta}$ that is constant everywhere except on $(1,3)\times S^1$ where  $h_{\alpha\beta}(\tau, \psi)=-\frac{\pi}{2}c_{\alpha\beta}^2(\tau-1)^2$.  Moreover, we make $h_{\alpha\beta}$ smooth with small modifications near $\tau=1, 3$.  Then the time-$(k_{\alpha\beta}/c_{\alpha\beta})$ flow on $(1,3)\times S^1$  is  given by $\varphi^{k_{\alpha\beta}/c_{\alpha\beta}}_{\alpha\beta}(\tau, \psi)=(\tau, \psi-\pi k_{\alpha\beta}(\tau-1))$.  Lastly, we orient each $L_\alpha(k)$ counterclockwise along $\partial C_\alpha$.  A few examples are shown in Figure \ref{fig: lagrangian}.  Applying Abouzaid's generation criterion \cite{Ab10} to a suitably chosen Hochschild cycle, we shown in Appendix A that:

\begin{lemma}\label{lemma: generation}
  The wrapped Fukaya category $\mathcal W(H)$ is split-generated by objects $L_{\alpha}(k)$, $\alpha\in A, k\in\mathbb Z$. 
\end{lemma}
\noindent In fact, the proof in Appendix A shows that $\{L_\alpha(0), \ L_\alpha(1), \ \alpha\in A\}$ already split-generates $\mathcal W(H)$, but we use all $L_\alpha(k)$ to make the algebraic structures more clear.

\begin{figure}[h]
\centering
	\scalebox{0.5}{\includegraphics{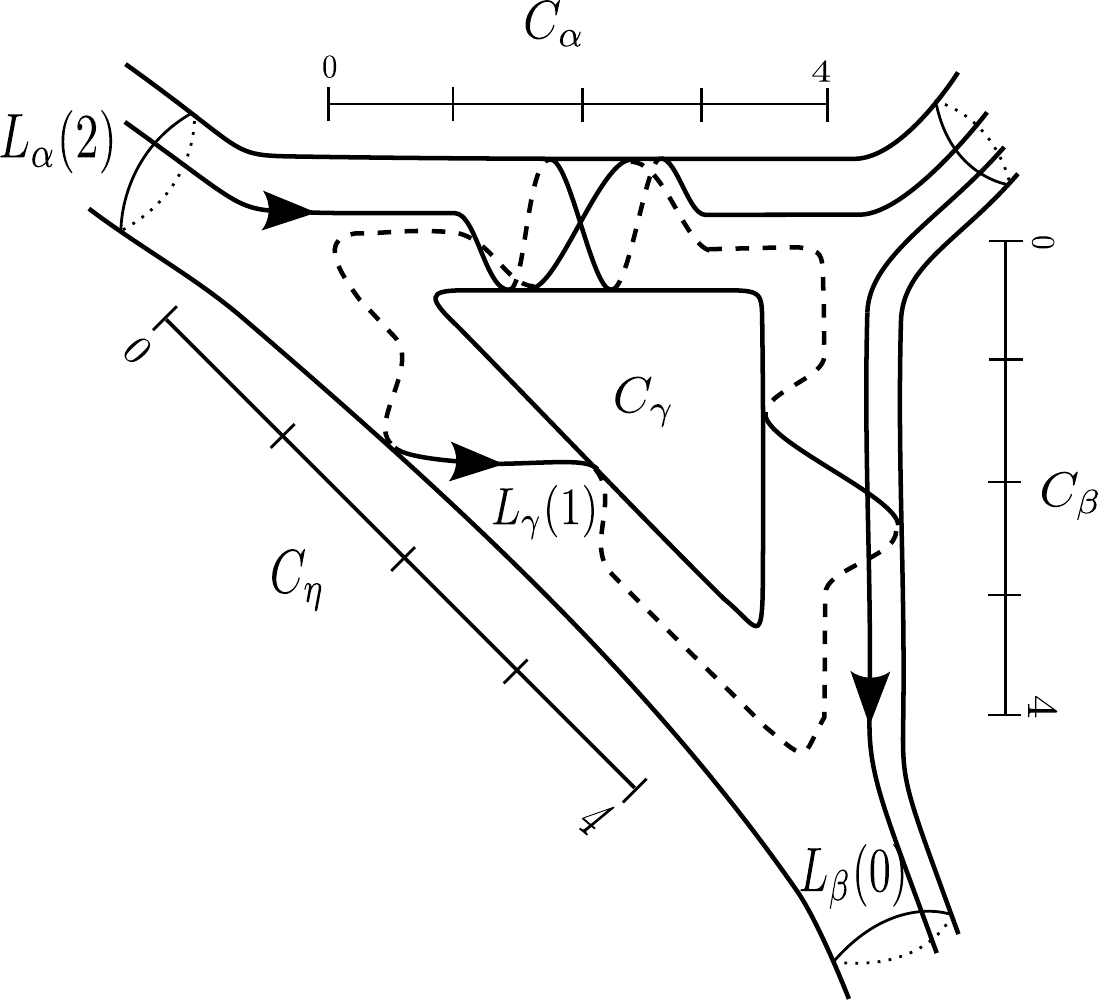}}
	\caption{Lagrangians $L_\alpha(2)$, $L_\beta(0)$,  and $L_\gamma(1)=L_\gamma(-1+2)$.  In this example, assume $n_{\alpha\gamma}=n_{\beta\gamma}=n_{\eta\gamma}=1$ and $d_{\alpha,\gamma}=d_{\beta,\gamma}=d_{\eta, \gamma}=1$, and we pick $\delta_{\alpha,\gamma}=\delta_{\beta,\gamma}=\delta_{\eta, \gamma}=0$ and $\delta_{\gamma,\alpha}=\delta_{\gamma,\beta}=\delta_{\gamma,\eta}=2$. }
	\label{fig: lagrangian}
\end{figure}

\begin{remark} \label{rmk: ambience} We can choose the generating Lagrangians in $H$ as boundaries of Lagrangians in $(\mathbb C^*)^2-H$, assuming that for each $\alpha\in A$,  $c_{\alpha} \in \mathbb R_{>0}$ and there is a point $p_\alpha =(p_{\alpha,1}, p_{\alpha,2})\in \mathbb Z^2\cap C_\alpha$ with the property that $\langle p_\alpha-u, \alpha-\beta\rangle$ is an odd integer for any $u\in  C_\alpha\cap C_\beta$.  Note that $\alpha-\beta$ is an integral normal vector to the edge $ C_\alpha \cap  C_\beta$ pointing into the region $C_\alpha$.  Let $\mathcal L_\alpha$ be the zero section of the cotangent bundle $T^*\mathbb R^2/(2\pi / |\log t|)\mathbb Z^2$ with its base restricted to $C_\alpha$, on which we consider the Hamiltonian function
\[
H_\alpha =-\frac{\pi}{|\log t|} \left((u_1-p_{\alpha,1})^2+(u_2-p_{\alpha,2})^2\right).
\]
 Then for $k\in \mathbb Z$, the time-$(\frac{1}{2}+k)$ Hamiltonian flow is given by (written in the universal cover of $T^*\mathbb R^2/(2\pi/|\log t|) \mathbb Z^2$),
\[
\phi^{k+\frac{1}{2}}_\alpha (u_1,\theta_1, u_2,\theta_2)=\left(u_1, \theta_1-\frac{1}{|\log t|}(\pi+2\pi k)(u_1-p_{\alpha,1}), u_2,\theta_2-\frac{1}{|\log t|}(\pi+2\pi k)(u_2-p_{\alpha,2})\right).
\]
Let $l_\alpha(k):=\partial (\phi_\alpha^{\frac{1}{2}+k}(\mathcal L_\alpha))$.  Lemma \ref{lemma: twist} shows that each $l_{\alpha}(k)$ corresponds to a Lagrangian contained in $H_\Pi$, which in turn corresponds to a Lagrangian in $H$.   We also need to modify $l_\alpha(k)$ by untwisting it so that it is invariant under Liouville flow in each cylindrical end, and further twist it in the bounded edges as needed to account for the $\delta_{\alpha,\beta}$'s.

\begin{lemma} \label{lemma: twist} $l_\alpha(k)\subset H_\Pi$.
\end{lemma}
\begin{proof}  For any $(u, \theta)\in l_\alpha(k)$, let $z=(t^{u_1+i\theta_1}, t^{u_2+i\theta_2})$ be the corresponding point in $(\bC^*)^2$ as in Equation \eqref{eq: cotangent identification}.  Then for each adjacent component $C_\beta$,
\begin{equation} \begin{array}{ll}
\arg\left(\frac{c_\alpha}{c_\beta}z^{\alpha-\beta}\right)=\arg z^{\alpha-\beta} & \equiv \sum_{j=1}^2 |\log t| \theta_j(\alpha_j-\beta_j)\\
 & = \sum_{j=1}^2 -(\pi+2\pi k)(u_j-p_{\alpha,j})(\alpha_j-\beta_j)\\
 & \equiv  -\pi(1+2k)\langle u-p_\alpha, \alpha-\beta \rangle \\
 & \equiv \pi 
\end{array}\end{equation}
where $k\in \mathbb Z$, and the last equivalence comes from $(1+2k) \langle u-p_\alpha, \alpha-\beta\rangle$ being an odd integer.  As defined in Equation (\ref{eqn: edge def}), $z\in  e_{\alpha\beta}^\Pi$. 
\end{proof}

\end{remark}

\subsection{Hamiltonian perturbation} \label{sec: hamiltonian} We now introduce Hamiltonian perturbations that will eventually enable us to define a model for $\mathcal W(H)$ that is suitable for computing $\mathcal W(H)$ from pair-of-pants decompostions. 

For every integer $n\geq 0$, let $H_{n}: H\to \mathbb R$ be a Hamiltonian such that for any unbounded edge $e_{\alpha\beta}$, $H_n$ is linear on its cylindrical end (as in \cite{AS10}), i.e. $H_n(\tau)=n\tau+d_n$, $\tau\in [1,\infty)$ there.  On each bounded edge $e_{\alpha\beta}$, $H_{n}$ agrees with the following function $H_{\alpha\beta,n}$ on $(0,4)\times S^1$ with coordinates $(\tau, \psi)$,   
\begin{equation}\label{eqn: hamiltonian}
H_{\alpha\beta, n}(\tau, \psi) = \left\{ \begin{array}{ll}
                     c_{\alpha\beta}(-n\pi\tau^2+2n\pi)+d_{\alpha\beta,n},              & 0 < \tau\leq 1\\
										c_{\alpha\beta}n\pi(\tau-2)^2+d_{\alpha\beta,n}, &1\leq \tau\leq 3   \\
	c_{\alpha\beta}(-n\pi(\tau-4)^2+2n\pi)+d_{\alpha\beta,n}, & 3\leq \tau < 4 
	                                     \end{array}\right. .
\end{equation} 
  The constants $d_n$ and $d_{\alpha\beta,n}$ above (regardless of whether the edge is bounded or unbounded)are  picked so that $H_n$ is continuous and that it is constant near the vertices in the complement of all the edges.  These can be satisfied by letting 
\[
c_n=\max \left\{2c_{\alpha\beta}n\pi \mid e_{\alpha\beta} \text{ is an bounded edge}\right\},
\]
and then setting $d_{\alpha\beta,n}=c_n-2c_{\alpha\beta}n\pi$ for each bounded edge $e_{\alpha\beta}$ and $d_{n}=c_n-n$ for each unbounded edge $e_{\alpha\beta}$.  This way, $H_n=c_n$ in the complement of all the edges.   We can check that $H_n\geq H_{n-1}$ for all $n$.  Even though $H_n$ is defined for $n\in \mathbb Z_{\geq 0}$, $H_w$ can be defined in the same way for any real value $w \geq 0$, and we will make use of it in Section \ref{sec: continuation}.
  Let $d_{\alpha\beta,0}=1$.  After we have picked $d_{\alpha\beta, n-1}$ for each edge $e_{\alpha\beta}$, we pick $\tilde d_{\alpha\beta,n}$ so that $H_n>H_{n-1}$.  Let

The corresponding time-1 flow generated by the Hamiltonian vector field $X_{H_{\alpha\beta,n}}$ is determined by the definition
\[
\omega_{\alpha\beta}(X_{H_{\alpha\beta,n}}, \cdot)=c_{\alpha\beta}d\tau\wedge d\psi (X_{ H_{\alpha\beta,n}}, \cdot)=dH_{\alpha\beta,n}.
\]  
The time-1 flow of $H_{\alpha\beta,n}$ written in the universal cover of $E_{\alpha\beta}\cong (0,4)\times \mathbb R$ is
\begin{equation}\label{eqn: flow}
\phi^1_{\alpha\beta,n}(\tau, \psi) = 
\left\{ 
\begin{array}{ll}
(\tau,\psi -2n\pi\tau),	& 0< \tau <1\\
 (\tau, \psi+2n\pi(\tau-2)), & 1\leq \tau \leq 3   \\
 (\tau, \psi-2n\pi(\tau-4)), & 3 < \tau < 4
 \end{array}.
 \right. 
\end{equation}
Note that by our construction $\phi^1_{\alpha\beta,n}= \phi^n_{\alpha\beta,1}$.

 Denote by $\phi^t_{n}$ the time $t$ flow of $H_n$.  We can see from Equation \eqref{eqn: flow} that the  Lagrangians  $\phi^1_n(L_{\alpha}(k))$ and $\phi^1_n(L_\beta(k))$ wrap around the bounded cylinder $e_{\alpha\beta}$ in the ``negative'' direction on intervals $(0,1)\cup (3,4)$ for $n$ times each and in the ``positive'' direction on (1,3) for $2n-k_{\alpha\beta}$ times ($+2n$ wraps due to the action of $H_n$ and $-k_{\alpha\beta}$ wraps from the definition of $L_\alpha(k)$, $L_\beta(k)$), as shown in Figure \ref{fig: hamiltonian}.  Note that  the degree of the generators of $CF(L_\alpha(k), L_\alpha(k); H_n)$ is even in the region where $H_n$ wraps positively and odd when it wraps negatively.  

\begin{figure}[h]
\centering
	\scalebox{0.75}{\includegraphics{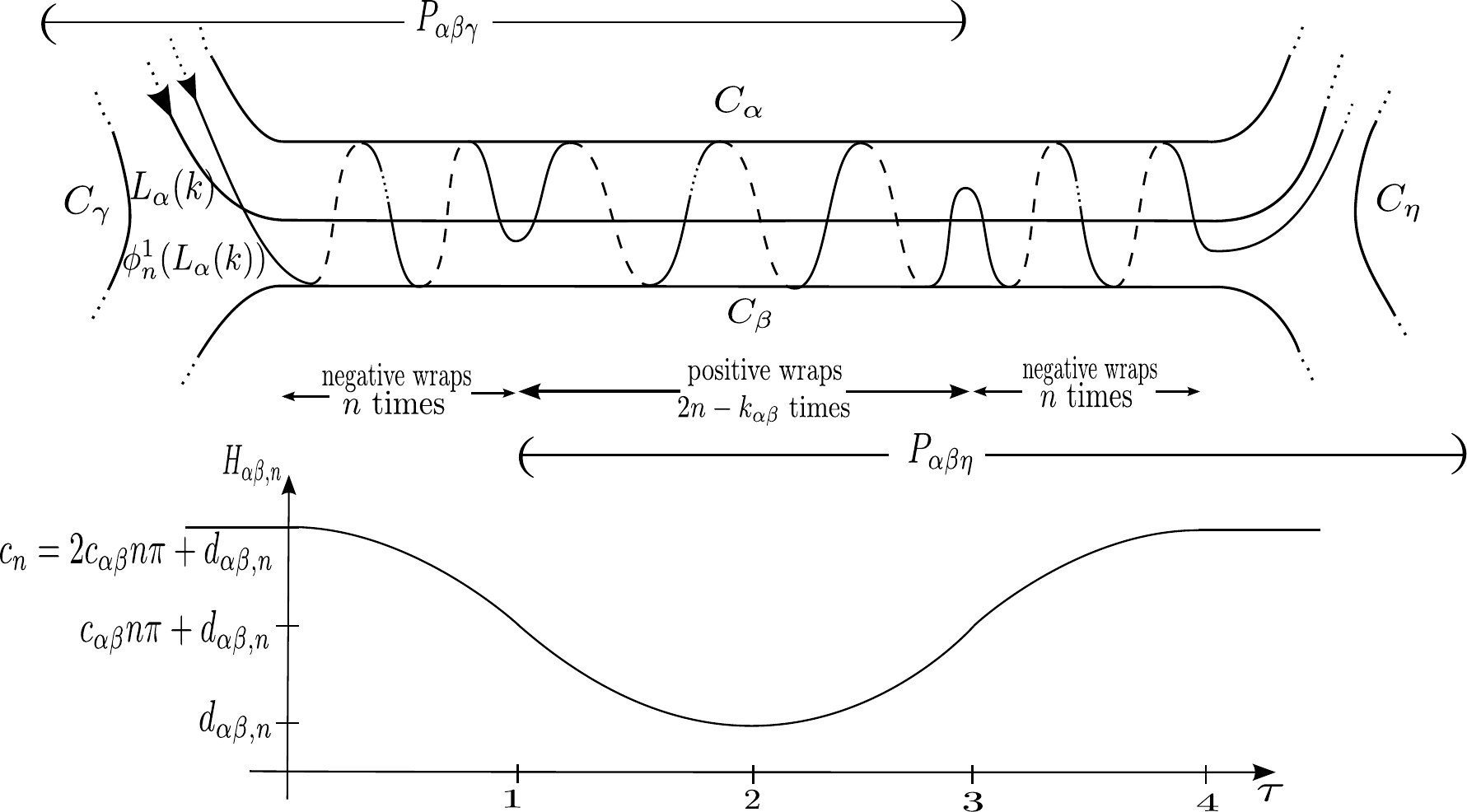}}
	\caption{The Hamiltonian $H_n$ on the cylindrical edge $e_{\alpha\beta}$, and $\phi^1_n(L_{\alpha}(k))$ wraps in the ``negative'' direction on intervals $(0,1)\cup (3,4)$ for $n$ times each and in the ``positive'' direction on (1,3) for $2n-k_{\alpha\beta}$ times.  In this particular drawing, $k_{\alpha\beta}=0$ so $L_{\alpha}(k)$ it self runs through $e_{\alpha\beta}$ straight. }
	\label{fig: hamiltonian}
\end{figure}

We also need to make a small modification of $H_n$ for all $n\in \mathbb Z$ so that $\phi^1_n(L_\alpha(k))$ intersects $L_{\alpha}(k)$ transversely and to make it smooth at $\tau=0,1,3,4$ on bounded edges and at $\tau=1$ on unbounded edges. Furthermore, the perturbations are chosen consistently inside the pair of pants regions (i.e. complements of the edges) so that $\phi^1_n(L_{\alpha}(k))$ intersects $L_\alpha(k)$ transversely, just once inside the pair of pants, and always at the same point which has degree $0$.

\subsection{Backgrounds on Floer complex and product operations} \label{subsec: background}

For any pairs of objects $L_i(k_1)$, $L_j(k_2)$, we get a Floer complex $CF^*(L_i(k_1),L_j(k_2); H_n)$, where $H_n$ is the Hamiltonian defined in Section \ref{sec: hamiltonian}.  This Floer complex is generated by the set $\mathcal X(L_i(k_1), L_j(k_2); H_n)$ consisting of Reeb chords that are time-1 trajectories of the Hamiltonian $H_{n}$ starting in $L_i(k_1)$ and ending in $L_j(k_2)$.  Equivalently, these chords correspond to intersection points of $\phi_{n}^1(L_i(k_1))\cap L_j(k_2)$. 

The differential $d: CF^*(L_i(k_1),L_j(k_2); H_n)\to CF^*(L_i(k_1),L_j(k_2); H_n)$ is given by the count of pseudo-holomorphic strips $u: \mathbb R\times [0,1]\to H$ which are solutions to the inhomogeneous Floer's equation
\begin{equation}\label{eqn: floer}
(du-X_{H_n}\otimes dt)^{0,1}_{J_t}=0, \ \text{equivalently }\ \frac{\partial u}{\partial s}+J_t\left(\frac{\partial u}{\partial t}-X_{H_n}\right)=0
\end{equation}
with boundaries $u(s,0)\in L_i$ and $u(s,1)\in L_j$.  As $s\to \pm\infty$, $u$ converges to Hamiltonian flow lines that are generators the Floer complex involved in the differential. Such a solution $u(s,t)$ can be equivalently seen as an ordinary $\tilde J=(\phi_n^{1-t})_*J$-holomorphic strip $\tilde u(s,t)=\phi_n^{1-t}(u(s,t))$ with boundaries on $\phi^1_n(L_i(k_1))$ and $L_j(k_2)$.   Indeed, 
\[\frac{\partial \tilde u}{\partial s}=(\phi_n^{1-t})_*\left(\frac{\partial u}{\partial s}\right) \ \text{and} \ \frac{\partial \tilde u}{\partial t}=(\phi_n^{1-t})_*\left(\frac{\partial u}{\partial t}-X_{H_n}\right), \]
so the Floer equation (\ref{eqn: floer}) becomes 
$\frac{\partial \tilde u}{\partial s}+\tilde J \frac{\partial \tilde u}{\partial t}=0.$

For Lagrangians $L_{i_0}(k_0),L_{i_1}(k_1), L_{i_2}(k_2)$, $n\gg |k_0|,|k_1|,|k_2|$, the product 
\begin{equation*} 
\begin{array}{l}
\mu^2(H_n): CF^*(L_{i_1}(k_1),L_{i_2}(k_2); H_n)\otimes CF^*(L_{i_0}(k_0), L_{i_1}(k_1); H_n) \hspace{1in} \\
\hspace{3.5in} \to CF^*(L_{i_0}(k_0), L_{i_2}(k_2); 2H_n)
\end{array}
\end{equation*}
is given by the count of solutions $u:D\to H$ of the perturbed Floer equation 
\begin{equation}
(du-X_{H_n}\otimes\beta)^{0,1}_{J_t}=0
\end{equation}
where $D$ is a disc with three strip-like ends, and the images of the three components of $\partial D$ are contained in the respective Lagrangians $L_{i_0}(k_0), L_{i_1}(k_1)$, and $L_{i_2}(k_2)$.  The 1-form $\beta$ on $D$ satisfies $d\beta=0$ and it pulls back to $dt$ on the input strip-like ends and to $2dt$ on the output strip-like end.  Again, by changing to a domain dependent almost-complex structure, this is equivalent to counting standard holomorphic discs with boundaries on $\phi_{n}^2(L_{i_0}(k_0)), \phi_n^1(L_{i_1}(k_1))$, and $L_{i_2}(k_2)$ instead.  

The higher products 
\begin{equation*} \begin{array}{l}
\mu^d(H_n): CF^*(L_{i_{d-1}}(k_{d-1}),L_{i_d}(k_d); H_n)\otimes \cdots \otimes CF^*(L_{i_0}(k_0), L_{i_1}(k_1); H_n) \hspace{0.7in} \\
\hspace{3.8in} \to CF^*(L_{i_0}(k_0), L_{i_d}(k_d); dH_n)
\end{array}
\end{equation*}
are constructed in a similar way.

\subsection {Orientation and grading}
Let's orient each $L_\alpha(k)$ counterclockwise along $\partial C_\alpha$.  This will give each Floer complex a $\mathbb Z_2$-grading.  For each transverse intersection point $x$ of $L_0$ and $L_1$, we can identify $T_xL_0\cong \mathbb R$ and $T_xL_1\cong i\mathbb R$ via a linear symplectic transformation.  Consider the path $l_t$ of Lagrangian lines with $l_0=\mathbb R$ and $l_1=i\mathbb R$, and $l_t=e^{-i\pi t/2}\mathbb R$.  If the path $l_t$ maps the orientation of $T_xL_0$ to the orientation of $T_xL_1$, then $\deg(x)=0$; otherwise, $\deg(x)=1$.  

It is also possible to define a $\mathbb Z$-grading on Floer complexes of the Lagrangians in consideration.  To do that, we need to pick a trivialization of $T^*H^{1,0}$.  There's no canonical way to choose this trivialization.  We can just choose a global nonvanishing section $\Omega$ of $T^*H^{1,0}$ to be any meromorphic form allowed to have zeros or poles at each puncture.   Once we make that choice, for any Lagrangian plane $l\subset T_xH$, $\Omega|_l=\alpha \text{vol}_l$ where $\alpha\in \mathbb C^*$ and $\text{vol}_l$ is a real volume form.  And we can define $\arg(l):=\arg (\Omega|_l)=\arg(\alpha)\in \mathbb R/\pi \mathbb Z$ (or $\mathbb R/2\pi \mathbb Z$ if Lagrangians are oriented already).  The Lagrangian Grassmannian $LGr(TH)$ can then be lifted to a fiberwise universal cover $\widetilde{LGr}(T_pH)=\{(l,\theta)|l\in LGr(T_pH),\theta\in \mathbb R, \theta\equiv \arg(\Omega|_l) \text{ mod } \pi\}$.
 
The tangent lines along a Lagrangian $L$ form a path of Lagrangian planes, which are mapped by the above phase function to $S^1$.   Because of the Lagrangians we consider are simply connected,  the homotopy class of this map is trivial, i.e. the Lagrangians have vanishing Maslov class. Hence for each Lagrangian $L$, we can equip it with a grading, which is a consistent choise of a graded lift to $\widetilde {LGr}(TH)$ of the section $p\mapsto T_pL$ of $LGr(TH)$ over $L$.  Then we can assign a degree to a transverse intersection $p\in L_0\cap L_1$.  In the case of Riemann surface $H$, $\deg (p) = \Big\lceil \frac{\theta_1-\theta_0}{\pi}\Big\rceil$.  (Cf. \cite{Se08, Au13}.)  Once this $\mathbb Z$-grading is available, the index of a rigid holomorphic polygon in $H$ corresponding to a higher product $\mu^k$ is $\deg(\text{output})-\sum\deg(\text{inputs})=2-k$.

We will make use of a $\mathbb Z$-grading in Section \ref{sec: higher continuation} (especially Section \ref{sec: exceptional disc}) to  rule out the possible existence of some exceptional holomorphic discs.   Other than that, everything in this paper is $\mathbb Z_2$-graded.

\subsection {Linear continuation map} \label{sec: continuation} 
For any two Lagrangians $L_1=L_i(k_1),L_2=L_j(k_2)$, we would like to define the wrapped Floer complex $CW^*(L_1, L_2)$ as a direct limit of the perturbed Floer complexes $CF^*(L_1, L_2; H_n)$ as $n\to\infty$.  This definition relies on the existence of continuation maps 
\[
\kappa: CF^*(L_1, L_2; H_n) \to CF^*(L_1, L_2; H_N),
\]
whenever $n\leq N$ (recall, by construction, $H_n\leq H_N$ whenever $n\leq N$).  In general, the direct limit construction would not be compatible with $A_\infty$-structures, hence, $CW^*(L_1,L_2)$ is defined as the homotopy direct limit in \cite{AS10}.  However, we will show in this section that it turns out in our case, each continuation map $\kappa$ above is just an inclusion for $n$ sufficiently large depending on $L_1, L_2$.  Furthermore, we will show in the next section that higher order continuation maps (i.e. those with $d\geq 2$ inputs) are trivial when mapping from sufficiently perturbed Floer complexes.  Consequently, the wrapped Floer complex can be defined as
\begin{equation}\label{eqn: wrapped complex}
CW^*(L_i(k_1), L_j(k_2))=\bigcup_{n=n_0}^\infty CF^*(L_i(k_1), L_j(k_2);H_n)/\sim
\end{equation}
where the equivalence relation is given by continuation maps which are inclusions.  Moreover, we have well defined differential and $A_\infty$-products 
\[
\mu^d: CW^*(L_{d-1},L_d)\otimes \cdots \otimes CW^*(L_0, L_1)\to CW^*(L_0, L_d)
\]
which are given by ordinary Floer differential and products.

In the usual definition of the continuation map, for an input $p\in \mathcal X(L_1, L_2; nH_1)=\mathcal X(L_1, L_2; H_n)$, the coefficient of $q\in \mathcal X(L_1, L_2; NH_1)=\mathcal X(L_1, L_2; H_N)$ in $\kappa(p)$ is given by the count of index zero solutions to the perturbed Floer equation (the continuation equation)
\begin{equation} \label{eqn: continuation}
\frac{\partial u}{\partial s}+ J\left(\frac{\partial u}{\partial t}-\lambda(s) X_{H_1}\right)=0,
\end{equation}
where $\lambda(s)$ is a smooth function which equals $N$ for $s\ll 0$ and $n$ for $s\gg 0$, and such that $\lambda'(s)\leq 0$.  Such a solution also needs to satisfy the boundary conditions $u(s,0)\in L_1$ and $u(s,1)\in L_2$, and it converges to the generator $p$ as $s\to \infty$ and the generator $q$ as $s\to -\infty$.  Instead of considering moduli spaces of perturbed Floer solutions, we will introduce the related moduli spaces of cascades of pseudoholomorphic discs by taking the limit where the derivative of $\lambda$ tends to zero.  Counting cascades of pseudoholomorphic discs with appropriate indices gives equivalent definitions of continuation maps.  We will only discuss the construction of cascades briefly following Appendix A of \cite{Au10}, \cite{AS10}, and Section 10(e) of \cite{Se08}, and we refer to them for details.

 First of all, we would like to go to the universal cover to make visualizations and discussions easier.  For any given input data $(p; L_1, L_2)$, we can obtain a lift $(p'; L_1', L_2')$ in the universal cover $E_H$ of the Riemann surface by picking an arbitrary lift $L_1'$ of $L_1$ and then tracing the lift of $L_2$ through $p'$ to determine $L_2'$.  Rather than counting pseudoholomorphic strips in $H$ with the above conditions, it is equivalent to count pseudoholomorphic strips in the universal cover $E_H$ with boundaries in the lifts $L_1'$ and $L_2'$.  The Lagrangians $L_1'=L_i(k_1)'$ and $L_2'=L_j(k_2)'$ satisfy some nice properties which are necessary for defining cascades:   

\begin{itemize}
	\item [(P1) ]  For all integer values of $w \geq n_0$ large enough, $\phi^1_{w}(L_1)$ is transverse to $L_2$, so $\phi_{w}(L_1')$ is transverse to $L_2'$. 
	\item [(P2) ]  For $w\geq n_0$ large enough, each point $x'\in \phi^1_w(L_1')\cap L_2'$ lies on a unique maximal smooth arc $\gamma: [n_0, \infty) \to E_H$ given by $t\mapsto \gamma(t)$, where $\gamma(t)$ is a transverse intersection point of $\phi_t^1(L_1')$ and $L_2'$ for all $t$.    In other words, as $w$ increases from $n_0$ to $\infty$, no new intersection between $\phi_w(L_1')$ and $L_2'$ are created, and the existing intersections remain transverse.  
\end{itemize}
Both (P1) and (P2) can be achieved by choosing $n_0>\max\{1, |k_1|+|k_2|\}$.  For Lagrangians $L_1, L_2 \subset H$, property (P2) for their corresponding lifts in $E_H$ implies that as $w$ increases from $n_0$ to $\infty$, any newly created intersection of $\phi^1_w(L_1)$ and $L_2$ will not be the output of any $J$-holomorphic disc.  For any $x\in \phi^1_w(L_1)\cap L_2$ and its lift $x'\in \phi^1_w(L_1')\cap L_2'$ which lies on the arc $\gamma$, we can identify $x$ with a unique point $\vartheta_{w}^{w'}(x)$, $w'\geq w$, determined by requiring the lift of $\vartheta_{w}^{w'}(x)$ to be $\gamma(w')$.

 For Lagrangians $L_1=L_i(k_1)$, $L_2=L_j(k_2)$, and $n\geq n_0$, we define the continuation map $\kappa: CF^*(L_1, L_2; H_n) \to CF^*(L_1, L_2; H_N)$, $N\geq n$, as follows.  Given $p\in \mathcal X(L_1, L_2; H_n)$ and  $q\in \phi^1_N(L_1)\cap L_2$, the coefficient of $q$ in $\kappa(p)$ is given by the count of linear cascades from $p$ to $q$ of Maslov index zero.  A \emph{$k$-step linear cascade} from $p$ to $q$ is a sequence of $k$ ordinary pseudo-holomorphic strips $u_1,\ldots u_k: \mathbb R\times [0,1]\to H$ satisfying:
\begin{itemize}
	\item $u_i(\mathbb R\times 0) \subset \phi^1_{w_i}(L_1)$, $u_i(\mathbb R\times 1) \subset L_2$ for some $w_1\leq \cdots \leq w_k$ in the interval $[n,N]$;
	\item $u_i$ has finite energy, and we denote by $p_i^{\pm}\in\phi^1_{w_iH}(L_1)\cap L_2$ the intersection points to which $u_i$ converge at $\pm\infty$;
	\item $p^{+}_{i+1}=\vartheta_{w_i}^{w_{i+1}}(p_i^-)$, $p^+_1=\vartheta_{n}^{w_1}(p)$, and $q=\vartheta_{w_k}^N(p^-_k)$.
\end{itemize}
When $q=\vartheta_n^N (p)$, we allow the special case of $k=0$.

\begin{figure}[h]
\centering
	\scalebox{0.7}{\includegraphics{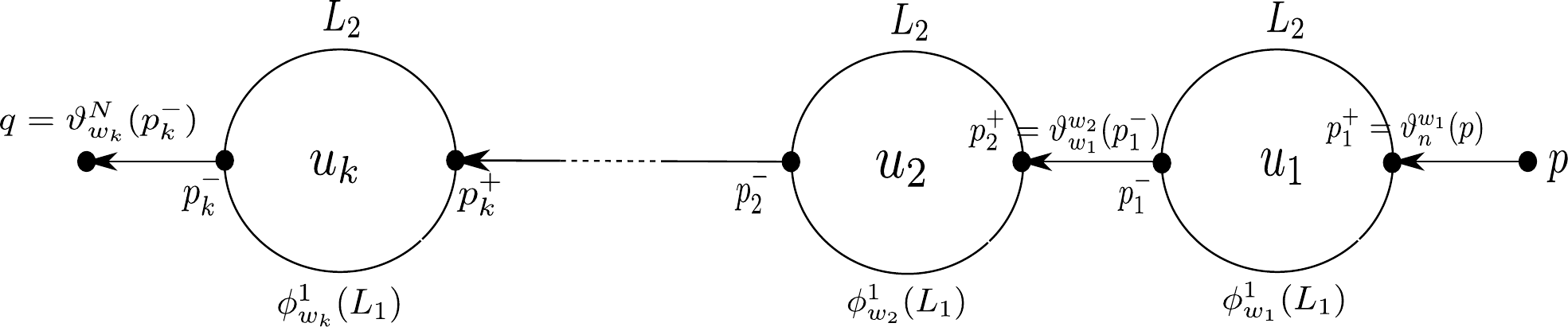}}
	\caption{a $k$-step linear cascade from $p$ to $q$.}
	\label{fig: linear cascade}
\end{figure}

\begin{lemma} \label{lemma: linear continuation}
Given any two Lagrangians, $L_1=L_i(k_1)$ and $L_2=L_j(k_2)$, for a large enough $n_0$ and  $n_0\leq n \leq N$, the linear continuation map $\kappa: CF^*(L_1, L_2; H_n) \to CF^*(L_1, L_2; H_N)$ is just the inclusion.
\end{lemma}

\begin{proof}
The components of an index zero linear cascade are holomoprhic strips whose Maslov indices sum to zero.    The only holomorphic strip of index zero on a Riemann surface (with an arbitrary complex structure) is the constant disc, and there are no holomoprhic strips of negative index.  Hence the continuation map has to be the inclusion induced by identifying intersection points using $\vartheta_n^N$.
\end{proof}

\noindent  As expected, $\kappa$ also has to be an inclusion in the usual definition of the continuation map which is defined by counting inhomogeneous holomorphic strips that are solutions to (\ref{eqn: continuation}).  Otherwise, there are nontrivial index zero inhomogeneous strips from $p$ to a different $q$ ($q\neq \vartheta_n^N(p)$).  Taking the limit where $d \lambda(s)/ds \to 0$, the index zero inhomogeneous holomorphic strips from $p$ to $q$ converge in the sense of Gromov to a nontrivial index zero linear cascade, thus contradicting Lemma \ref{lemma: linear continuation}.  

\subsection{Higher continuation maps}  \label{sec: higher continuation} A given set of boundary data consists of Lagrangians $L_0,\ldots,L_d$ with each $L_i=L_{\alpha_i}(k_i), \alpha_i\in A, k_i\in \mathbb Z$, inputs $p_1\in \mathcal X(L_0, L_1; H_n),\ldots,$ $p_d\in \mathcal X(L_{d-1},L_d; H_n)$, and the output $q\in \mathcal X (L_1, L_2; H_N)$ with $N\geq dn$.   The coefficient of $q$ in the continuation map 
\[ CF^*(L_{d-1}, L_d; nH)\otimes\cdots\otimes CF^*(L_0, L_1; nH)\to CF^*(L_0, L_d; NH)\]
can be defined by counting exceptional solutions, $u:D\to H$ where $D$ is a disc with $d+1$ strip-like ends, to the perturbed Floer equation $(du-X_{\tilde H} \otimes \beta)^{0,1}=0$ with boundaries on Lagrangians $L_0,\ldots,L_d$. Here, $\beta$ is a closed $1$-form, and $\tilde H(s)$ is a domain dependent Hamiltonian. If we represent $D$ as a strip-like disc illustrated in Fig. \ref{fig: strip}, then $\tilde H=\lambda(s)H$ interpolates between $\tilde H(s) =nH$ as $s\to \infty$ and $\tilde H(s)=\frac{N}{d}H$ as $s\to -\infty$. 

\begin{figure}[h]
\centering
	\scalebox{0.9}{\includegraphics{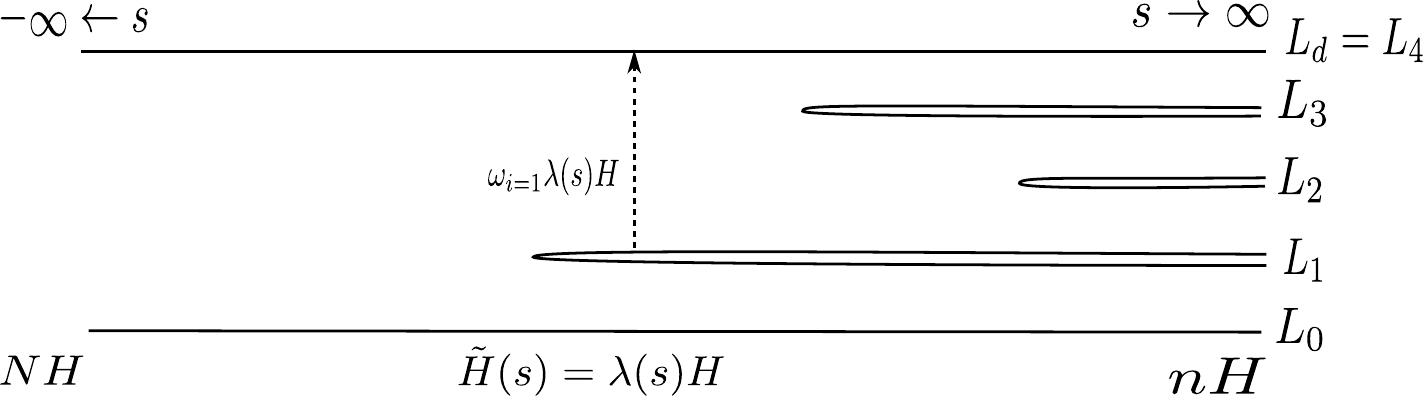}}
	\caption{Higher continuation maps with $d$ inputs can be defined by counting perturbed Floer equations $u:D\to H$ with boundaries on $L_0,\ldots, L_{d}$ ($d=4$ in this picture).  The disc $D$ with $d+1$ strip-like ends is equivalent to the strip $\{(s,t)\in (\mathbb R\times [0,d])\backslash ((d-1) \text{ slits})\}$.}
	\label{fig: strip}
\end{figure}

As with linear continuations, such a definition is hard to use for explicit computations, hence we use the equivalent definition of counting cascades.  Intuitively, perturbed Floer solutions are homotopic to cascades as we take the limit by stretching the strip so that the interpolating Hamiltonian $\tilde H(s)$ changes infinitely slowly, i.e. $\lambda'(s)\to 0$.  Gromov compactness tells us that there are finitely many places where the energy concentrates and we get unperturbed pseudoholomorphic discs as pieces of a cascade.  A cascade is a collection of unperturbed pseudoholomorphic discs with boundary on any $(r+1)$-tuple 
\[
\phi^1_{w_{i_0}\lambda(s) H}(L_{i_0}),\ldots, \phi^1_{w_{i_r}\lambda(s) H}(L_{i_r}) 
\] 
of perturbed Lagrangians with $i_0<\cdots <i_r\in \{0,\ldots, d\}$, $n\leq \lambda \leq N/d$, $w_j=d-j$ for each $j=0,\ldots, d$, and such that at least one of the pseudoholomorphic discs in this collection must be exceptional, i.e. of index less than (2-\#inputs).  Again, refer \cite{Au10},\cite{AS10}, \cite{Se08} for details.  In this section, we will show that there are no such exceptional discs, hence higher continuation maps are all trivial.

\subsubsection{Stability of intersection points and crossing changes.}  

Given a set of boundary data as above, let $L_0', \ldots, L_d'$  be lifts of $L_0,\ldots, L_d$ in the universal cover $E_H$ obtained from choosing an arbitrary lift $L_1'$ of $L_1$ and then determining $L_2',\ldots, L_d'$ by tracing through lifts $p_1',\ldots, p_d'$.  By taking $n_0>\max\{1, |k_1|+\cdots+|k_d|\}$, any pair $L_i', L_j'$ of Lagrangians in this collection of lifts satisfies properties (P1) and (P2) mentioned in Section \ref{sec: continuation}.

Besides pairwise intersections, in order to obtain a vanishing result for cascades, we also want to make sure that intersection points of three or more Lagrangians also stabilize.  That is, for any $(r+1)$-tuple $l_0=\phi^1_{w_{i_0}\lambda}(L'_{i_0}),\ldots, l_r=\phi^1_{w_{i_r}\lambda}(L'_{i_r}) $ with with $i_0<\cdots <i_r\in \{0,\ldots, d\}$, $n\leq \lambda \leq N/d$, $w_j=d-j$ for each $j=0,\ldots, d$ as above, we will show in this section that no intersection points in $l_0\cap \cdots \cap l_r$ are created or canceled as we vary $\lambda(s)$ so long as $n$ is large enough.   This is equivalent to saying that there are no crossing changes (i.e. Reidemeister moves) between multiple Lagrangians in $\{\phi^1_{w_{0}\lambda}(L'_{0}),\ldots,\phi^1_{w_{d}\lambda}(L'_{d})\}$ as $\lambda (s)$ varies so long as $n$ is large enough.

First, suppose all of $l_0,\ldots,l_r$ run through the universal cover $E_{\alpha\beta}\cong (0,4)\times \mathbb R$ of the edge $e_{\alpha\beta}$, we can write down the coordinates for each $l_j=(\tau,\psi_j(\tau))$ using Equation (\ref{eqn: flow}), 
\begin{equation}\label{eqn: lines}
\psi_j(\tau)=\left\{ 
\begin{array}{ll}
c_j(\tau)-2\pi(d-i_j)\lambda(s)\tau,  & 0< \tau <1\\
c_j(\tau)+2\pi(d-i_j)\lambda(s)(\tau-2), & 1\leq \tau \leq 3   \\
c_j(\tau)-2\pi(d-i_j)\lambda(s)(\tau-4), & 3 < \tau < 4 
 \end{array},\right. 
\end{equation}
 where $c_j(\tau)$ is a continuous function that can be arranged to be constant on the $\tau$-interval $(0,1)\cup (5/4,4)$ and $c_j(5/4)-c_j(1)=-2\pi k_{i_j}$.  To arrange $c_j(\tau)$ to be constant in $(0,1)\cup (5/4,4)$, we need to adjust the construction of each $L_{\alpha}(k)$ in Section \ref{sec: lagrangian} so that $L_\alpha(k)$ twists $k_{\alpha\beta}$ times in the interval $(1, 5/4)$ instead of in the interval $(1,3)$.  All the results in this paper are valid (with some obvious changes in the arguments) when we use this modified construction for $L_{\alpha}(k)$, which is homotopic to the construction given in Section \ref{sec: lagrangian}.  We will use this modified construction only in this section. 

As $n$ increases, the $\tau$-coordinate of any intersection point, $p$, between any pair of Lagrangians $l_j$ and $l_{j'}$ in $\{l_0,\ldots, l_r\}$ moves toward $0$, $2$, or $4$.  Indeed, it follows from Equation \eqref{eqn: lines} that one of $\tau(p)$, $\tau(p)-2$, or $\tau(p)-4$ is equal to $\frac{c_j(\tau)-c_{j'}(\tau)}{\lambda(s)2\pi(i_{j'}-i_j)}\to 0$ as $\lambda(s)\to \infty$.  Consequently, given inputs $p_1,\ldots, p_d$, we can always choose $n$ to be large enough so that none of the $\tau$-coordinates of the given $p_i$'s will be in the the region $(1,5/4)$ after replacing $p_1,\ldots, p_d$ by their corresponding generators via the linear continuation map.  Furthermore, by choosing $n$ to be large enough, the appropriate lifts  $l_0,\ldots, l_r$ to the universal cover of the cylinder will no longer have any intersections in the $(1,5/4)$ region.

In each of the intervals $(0,1), (5/4,3)$, and $(3,4)$, $l_j$ is a line segment with the same slope of $2\pi(d-i_j)\lambda(s)$ in absolute value.   In each interval, these line segments belong to $r+1$ lines that may or may not intersect at a single point.  Whether they intersect in a single point or not depends on the $\psi$-intercepts $c_j(\tau)$'s and ratios between the differences in the slopes, which are independent of $\lambda(s)$.  Take the interval $(0,1)$, if the corresponding $(r+1)$ lines do intersect at a single point, then for $n$ large enough, the line segments in $l_0,\ldots, l_r$ will either always intersect in a single point in the interval $(0,1)$ or they will never intersect in that interval.  The same can be said for the other intervals.  We can then apply the same procedure to all edges and for all subsets of multiple Lagrangians in $\{\phi^1_{d\lambda}(L_0'),\ldots, L_d'\}$.  We can obtain a large enough $n$ for each of these cases and take the maximum value.  The argument for intersections inside the unbounded cylinders is the same.  There will be constant multiple intersections inside each pair of pants (in the complement of the edges), but because of the small perturbations we chose at the end of Section \ref{sec: hamiltonian}, they will be always be of degree zero (so, not exceptional), and they will never move or undergo crossing changes.

\subsubsection{Exceptional discs are constant.}  \label{sec: exceptional disc}

A pseudoholomorphic disc $u$ bounded by Lagrangians 
\[
l_0=\phi^1_{w_{i_0}\lambda H}(L_{i_0}'),\ldots, l_r=\phi^1_{w_{i_r}\lambda H}(L_{i_r}')
\] 
is either a nondegenerate polygon, a constant disc at the intersection of all $r+1$ Lagrangians, or a polygon with some of its corners being intersection points of multiple Lagrangians.  We want to investigate when such a pseudoholomorphic disc is exceptional, i.e. of index less than $2-r$.  A nondegenerate polygon has index
\[ 
2-r+2\cdot \#(\text{interior branch points}) + \#(\text{boundary branch points})\geq 2-r,
\] 
hence it's not exceptional.   

Next, we analyze the degenerate cases involving multiple intersection points.  At such a point, note that even though $l_0,\ldots, l_r$ may come from different components and hence have different orientations, we can change the orientation of some of them so that $l_0,\ldots, l_r$ have the same orientation on every edge in trying to compute the index of the disc.  This is because changing the orientation of a Lagrangian $l_j$ will change the degree of intersections between $l_{j-1}\cap l_{j}$ and $l_{j}\cap l_{j+1}$ in opposite ways leaving the overall index of the disc unchanged.  For the rest of this section, let us assume that $l_0,\ldots, l_r$ have the same orientation on every edge.

For any two Lagrangians $l_j, l_{j+1}$,  the degree of an intersection point $p\in l_j\cap l_{j+1}$ inside a cylindrical edge is either $0$ or $1$ depending on the location of $p$. If $p$ is in the positively wrapped region $(1,3)\times \mathbb R$ of some edge, then $\deg(p)=0$ because the slope of  $l_i$ is positive and larger than that of $l_{i+1}$ which is also positive.  On the other hand, if $p$ is in the negatively wrapped region $((0,1)\cup (3,4))\times \mathbb R$ of some edge, then $\deg (p)=1$.  

 \emph{Case 1:} Suppose $u$ is a constant disc with its image in a cylindrical end or in the positively wrapped region of a bounded edge (i.e. in $(1,3)\times S^1$). In this case, all the input and output intersection points have degree zero, hence $\ind u=0$.   When $r\geq 2$, $\ind u\geq 2-r$, so it is not an exceptional disc that contribute to the higher continuation map.

 \emph{Case 2:} Suppose $u$ is a constant disc with its image in the negatively wrapped region of a bounded edge (i.e. in $((0,1)\cup (3,4))\times S^1$).  In this case, all the input and output intersection points have degree 1, hence $\ind u=1-r<(2-r)$ and $u$ is an exceptional disc.  We have $\dim \ker D_{\bar{\partial},u}=r-2$ coming from the freedom to move the $r+1$ marked points, where $D_{\bar{\partial}, u}$ is the linearized Cauchy-Riemann operator.  (Note that $\dim \ker D_{\bar{\partial}, u}=0$ if we fix the marked points.  This is  because $u$ is constant with image $p$ and  $u^*TH=T_pH \cong \mathbb C$, and by the open mapping theorem, the only holomorphic disc with boundary in the union of the lines $T_pl_0, \ldots, T_pl_s$ is the constant map at the origin of the tangent space.) Hence $\dim \Coker D_{\bar{\partial},u}= r-1>0$ and $u$ is not regular.  This analysis shows that we need to perform a deformation to achieve transversality, and this will be our topic of the next subsection. 

\emph{Case 3:}  Suppose $u$ is a polygon with $\tilde r+1$ geometrically distinct vertices $p_0,\ldots, p_{\tilde r}$ ($\tilde r\geq 1$ since we assume $u$ is not constant) with some of its vertices being intersection points of multiple Lagrangians.  We know that the index of a nondegenerate polygon with $\tilde r+1$ vertices (i.e. bounded by $\tilde r+1$ Lagrangians) is at least $2-\tilde r$.   If a vertex of $u$ which is an intersection point of multiple Lagrangians is located in the positively wrapped region, then its contribution to $\ind u$ is zero.  If a vertex of $u$ is an intersection point of $v+2$ Lagrangians $(v>0)$ located in the negatively wrapped region, then the extra degenerate edges and vertices  add to the contribution of this vertex to $\ind u$ by $-v$.  That is, each of the $(r-\tilde r)$ extra Lagrangians contributes at the least a $-1$ to $\ind u$.  Hence, $\ind u\geq (2-\tilde r)-(r-\tilde r)=2-r$, i.e. $u$ is not exceptional.

\subsubsection{Deformation of constant discs.}
At each intersection, $p$, of $r+1$ ($r>1$) Lagrangians
 \[
 l_0=\phi^1_{w_{i_0}\lambda(s) H}(L_{i_0}'),\ldots, l_r=\phi^1_{w_{i_r}\lambda(s) H}(L_{i_r}')
 \] 
 in the negatively wrapped portion of a cylinder, we want to pick Hamiltonian perturbations $ch_0(s), \ldots, ch_r(s)$ that perturb these Lagrangians locally so that as soon as we turn on the Hamiltonian perturbation, the constant disc $u$ at the intersection of these Lagrangians is removed and we don't introduce any new holomorphic discs with boundary on $l_0,\ldots, l_r$, in that order.  Recall that each weight $w_{i_j}=d-i_j$, and the Lagrangians are ordered with $i_0<\cdots <i_r$.  

Let $h(s)$ be a Hamiltonian supported near the moving intersection point (this is the only way in which it depends on $s$) such that locally $h(s)(\tau,\psi)=\psi-\psi_0$, so $X_h\sim -\frac{\partial}{\partial \tau}$ points along the negative $\tau$-axis.  We use $ch_j(s)=cw_{i_j}h(s)$ to perturb $l_j$.  This pushes the Lagrangians into the desired positions as shown in Figure \ref{fig: deformation}.  Indeed, up to rescaling of both axes, each $l_j$ was initially the line $\psi=-w_{i_j}\tau$, and after translation, it becomes $\psi=-w_{i_j}(\tau+cw_{i_j})$.  Lagrangians $l_j$ and $l_k$ intersect where $\psi=-w_{i_j}(\tau+cw_{i_j})=-w_{i_k}(\tau+cw_{i_k})$, so $\tau=-c(w_{i_j}+w_{i_k})$.  For a fixed $j$, these are all distinct and in the correct order.  These local perturbations for the Lagrangians are automatically consistent because they are built from a single $ch(s)$ and the existing weights.

\begin{figure}[h]
\centering
	\scalebox{0.7}{\includegraphics{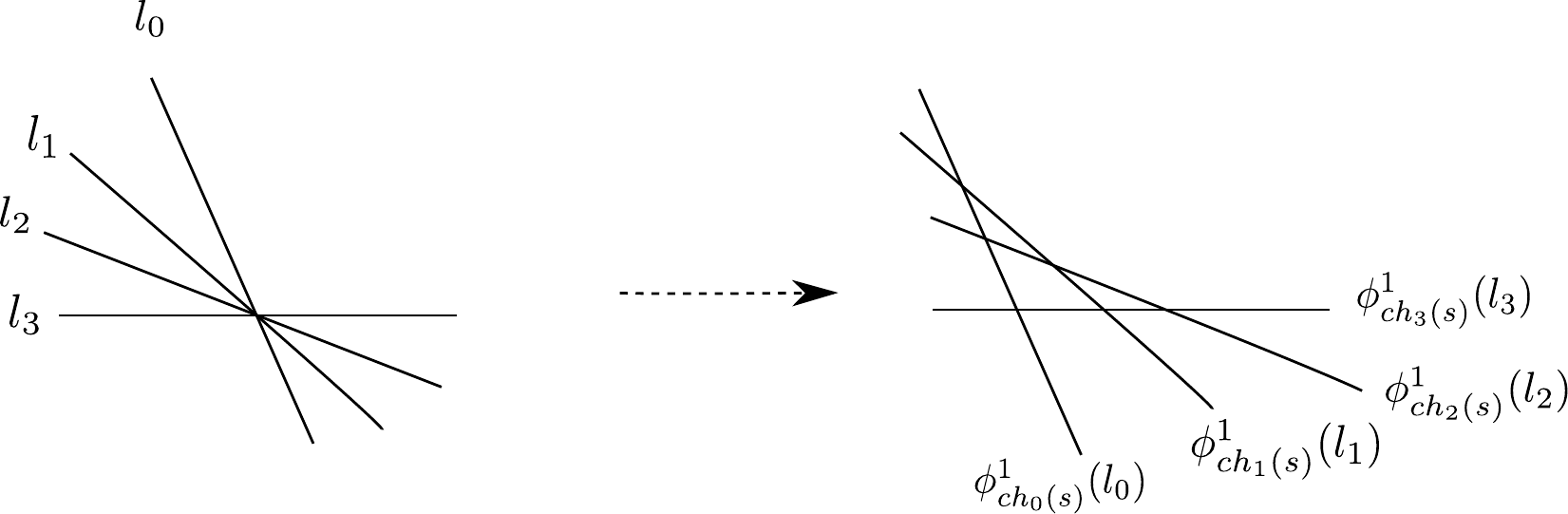}}
	\caption{Left: a constant disc at the intersection of multiple Lagrangians in the negatively wrapped portion of a cylinder.  Right: after Hamiltonian perturbations, the constant disc is removed without introducing new holomorphic discs.}
	\label{fig: deformation}
\end{figure}

The constant map $u$, with its image being the point $p$, is a non-regular solution to the perturbed Floer equation $(du+cX_{h(s)}\otimes \beta)^{0,1}=0$ for $c=0$;  we study its deformations among solutions to this equation for small $c$.  Let us consider the first order variation 
\[\left(d(u+cv)+cX_h(u+cv)\otimes \beta\right)^{0,1}=0,\]
where $v: D\to T_pH=\mathbb C$.   From the above equation, $v$ satisfies the linearized equation $\bar{\partial}v=(X_h(p)\otimes \beta)^{0,1}$ with boundary conditions on the real line $T_pl_j$.  We claim that, no matter the position of the boundary marked points on the domain $D$, this linearized equation has no solutions.  Indeed, we rewrite the linearized equation for $\tilde v=v-tX_h(p)$, where $t$ is a coordinate on $D$ with $dt=\beta$, and $t=-w_{i_j}$ on the $j$-th piece of the boundary of $D$.  Then a solution $\tilde v$ is a holomorphic map with boundary conditions in $T_pl_j+\frac{\partial l_j}{\partial c}=T_pl_j+w_{i_j}X_h(p)$, which are lines in $\mathbb C$ as in Figure \ref{fig: deformation}.   Hence this linearized equation has no solution; that is, the projection to $\Coker_{D_{\bar{\partial}_J}}$ of the perturbation term yields a nonvanishing section of the obstruction bundle over the moduli space of solutions.  This means that there is no way of deforming the constant map $u$ to a solution with the perturbation term added, i.e. cascades with such a constant component do not contribute to the continuation map.

\subsection{Wrapped Fukaya category of a pair-of-pants: some notations.}

We can view $H$ as a union of pairs of pants, $H=\bigcup_{\alpha,\beta,\gamma}P_{\alpha\beta\gamma}$, where each $P_{\alpha\beta\gamma}$ is a pair of pants whose image $\Log_t(P_{\alpha\beta\gamma})$ is adjacent to all three components  $C_{\alpha}, C_{\beta},$ and $C_{\gamma}$.  Also, if $e_{\alpha\beta}\cong (0,4)\times S^1$ is bounded, we require that the leg $P_{\alpha\beta\gamma}\cap e_{\alpha\beta}\cong (0,3)\times S^1$.  This way, if $e_{\alpha\beta}$ is a bounded edge connecting two pairs of pants $P_{\alpha\beta\gamma}$ and $P_{\alpha\beta\eta}$, then these two pairs of pants will overlap on the positive wrapping part of $e_{\alpha\beta}$ (see Fig. \ref{fig: hamiltonian}).  Denote $S_{\alpha\beta}=P_{\alpha\beta\gamma}\cup P_{\alpha\beta\eta}$. When we are considering a Lagrangian restricted to a pair of pants, $L_\alpha(k)\cap P_{\alpha\beta\gamma}$, we still call it $L_{\alpha}(k)$ for convenience.  Note that $L_{\alpha}(k)$ are actually equivalent to $L_{\alpha}(0)$ in the wrapped Fukaya category of the pair of pants. The objects $L_i(0)$, $i\in\{\alpha,\beta,\gamma\}$, generate the wrapped Fukaya category of $P_{\alpha\beta\gamma}$.

 The definitions for Floer complex and product structures introduced in Section \ref{subsec: background} apply to a pair of pants $P_{\alpha\beta\gamma}$ as well. (In fact, if we extend each bounded leg of a pair of pants to an infinite cylinderical end, then the quadratic Hamiltonian $H_n$ on $(1,3)\times S^1$ is equivalent to a linear Hamiltonian on that cylindrical end.)  We can define the Floer complex $CF^*_{P_{\alpha\beta\gamma}}(L_i(k_1), L_j(k_2); H_n)$ and its set of generators $\mathcal X_{P_{\alpha\beta\gamma}} (L_i(k_1), L_j(k_2); H_n)$ for any pairs of objects $L_i(k_1)$, $L_i(k_2)$ in $\mathcal W(P_{\alpha\beta\gamma})$.  Similarly, we can define the differential
\[d: CF_{P_{\alpha\beta\gamma}}^*(L_i(k_1),L_j(k_2); H_n)\to CF_{P_{\alpha\beta\gamma}}^*(L_i(k_1),L_j(k_2); H_n)\] and products
\begin{equation*} \begin{array}{l}
\mu_{P_{\alpha\beta\gamma}}^d(H_n): CF_{P_{\alpha\beta\gamma}}^*(L_{i_{d-1}}(k_{d-1}),L_{i_d}(k_d); H_n)\otimes \cdots \otimes CF_{P_{\alpha\beta\gamma}}^*(L_{i_0}(k_0), L_{i_1}(k_1); H_n) \hspace{0.7in} \\
\hspace{3.8in} \to CF_{P_{\alpha\beta\gamma}}^*(L_{i_0}(k_0), L_{i_d}(k_d); dH_n)
\end{array}
\end{equation*}
by counting perturbed pseudo-holomorphic discs  $u:D \to P_{\alpha\beta\gamma}$ with strip-like ends satisfying appropriate boundary and limiting conditions.  
For the same reason as in the previous section, the continuation maps 
\[\kappa: CF^*_{P_{\alpha\beta\gamma}}(L_i(k_1), L_j(k_2); H_n) \to CF^*_{P_{\alpha\beta\gamma}}(L_i(k_1), L_j(k_2); H_N), \ \ \ n<N,\]
are inclusions and compatible with $A_\infty$ structures if $n>n_0$ for a large $n_0$.  We then have the wrapped Floer complex 
\[CW^*_{P_{\alpha\beta\gamma}}(L_i(k_1), L_j(k_2))=\bigcup_{n=n_0}^\infty CF^*_{P_{\alpha\beta\gamma}}(L_i(k_1), L_j(k_2);H_n)/\sim\] 
 and the product 
\[\mu_{P_{\alpha\beta\gamma}}^d: CW_{P_{\alpha\beta\gamma}}^*(L_{d-1},L_d)\otimes \cdots \otimes CW_{P_{\alpha\beta\gamma}}^*(L_0, L_1)\to CW_{P_{\alpha\beta\gamma}}^*(L_0, L_d).\]

\subsection{Computing $\mathcal W(H)$ from pair-of-pants decompositions} \label{section: PPD}  We can split the generators $\mathcal X_{P_{\alpha\beta\gamma}}(L_i(k_1), L_j(k_2); H_n)$ into two parts 
\begin{equation}
\mathcal X_{P_{\alpha\beta\gamma}}(L_i(k_1), L_j(k_2);H_n)=\mathcal J_{\alpha\beta}^\gamma(L_i(k_1), L_j(k_2);H_n) \cup \mathcal C_{\alpha\beta}(L_i(k_1), L_j(k_2);H_n),
\end{equation}
consisting of generators in $P_{\alpha\beta\gamma}\backslash \left((1,3)\times S^1\right)$ and $(1,3)\times S^1\subset e_{\alpha\beta}$, respectively.  Denote 
\[
\mathcal J_{\alpha\beta}^\gamma(H_n)=\bigcup_{i,j,k_1,k_2} \mathcal J_{\alpha\beta}^\gamma(L_i(k_1), L_j(k_2);H_n),
\] 
\[
\mathcal C_{\alpha\beta}(H_n)=\bigcup_{i,j,k_1,k_2} \mathcal C_{\alpha\beta}(L_i(k_1), L_j(k_2);H_n).
\]
Let $CF^*_{\mathcal J_{\alpha\beta}^\gamma}(L_i(k_1), L_j(k_2);H_n)$ and $CF^*_{\mathcal C_{\alpha\beta}}(L_i(k_1), L_j(k_2);H_n)$ be subspaces generated by\\ $\mathcal J_{\alpha\beta}^\gamma(L_i(k_1), L_j(k_2);H_n)$ and $\mathcal C_{\alpha\beta}(L_i(k_1), L_j(k_2);H_n)$, respectively.

  Let $\mathcal X_{P_{\alpha\beta\gamma}}(L_i(k_1), L_j(k_2))$ be the set of generators of the wrapped Floer complex $CW^*_{P_{\alpha\beta\gamma}}(L_i(k_1), L_j(k_2))$ given in Equation (\ref{eqn: wrapped complex}).   We can similarly define subsets of generators 
 \[
 \cJ_{\alpha\beta}^\gamma(L_i(k_1), L_j(k_2))= \bigcup_{n=n_0}^\infty \cJ_{\alpha\beta}^\gamma(L_i(k_1), L_j(k_2);H_n)/\sim,
 \]
 \[
 \cC_{\alpha\beta}(L_i(k_1), L_j(k_2))= \bigcup_{n=n_0}^\infty \cC_{\alpha\beta}(L_i(k_1), L_j(k_2);H_n)/\sim.  
 \]
 Furthermore, denote
 \[
 \cJ_{\alpha\beta}^\gamma = \bigcup_{i,j, k_1, k_2} \cJ_{\alpha\beta}^\gamma(L_i(k_1), L_j(k_2))=\bigcup_{n=n_0}^\infty\mathcal J_{\alpha\beta}^\gamma(H_n)/\sim,
 \]   
 \[
  \cC_{\alpha\beta} = \bigcup_{i,j, k_1, k_2} \cC_{\alpha\beta}(L_i(k_1), L_j(k_2))=\bigcup_{n=n_0}^\infty\cC_{\alpha\beta}(H_n)/\sim.
  \] 
Define  $CW^*_{\cJ_{\alpha\beta}^\gamma}(L_i(k_1), L_j(k_2))$ and $CW^*_{\cC_{\alpha\beta}}(L_i(k_1), L_j(k_2))$ to be subspaces of $CW^*_{P_{\alpha\beta\gamma}}(L_i(k_1),L_j(k_2))$ generated by $\cJ_{\alpha\beta}^\gamma(L_i(k_1), L_j(k_2))$ and $ \cC_{\alpha\beta}(L_i(k_1), L_j(k_2))$, respectively.

 Note that $\mathcal X_{P_{\alpha\beta\gamma}} (L_i(k_1), L_j(k_2); H_n)=\mathcal X_{P_{\alpha\beta\gamma}} (L_i(k_1), L_j(k_2); nH_1)$ because $nH_1$ only differs from $H_n$ by a constant on each edge and both Hamiltonians only act on the edges (and the small perturbations inside the pants yield the same generators as well).

\begin{lemma} \label{lemma: ideal}

Given $x_1 \in \mathcal X_{P_{\alpha\beta\gamma}}(L_{i_0}(k_0), L_{i_1}(k_1)),\ldots,x_d\in \mathcal X_{P_{\alpha\beta\gamma}}(L_{i_{d-1}}(k_{d-1}), L_{i_d}(k_d)),$ where at least one of $x_j, j=1,...,d$ is in $\mathcal J_{\alpha\beta}^\gamma$, then the output $\mu^d_{P_{\alpha\beta\gamma}}(x_1,\ldots, x_d)$ is in $CW^*_{\cJ_{\alpha\beta}^{\gamma}}(L_{i_0}(k_0), L_{i_d}(k_d))$.

\end{lemma}

\begin{proof} 
From the definition of the wrapped Floer complex, there is a sufficiently large $N$, dependent on $x_1,\ldots, x_d$, such that for all $n\geq N$, $x_j\in \mathcal X_{P_{\alpha\beta\gamma}}(L_{i_{j-1}}(k_{j-1}), L_{i_{j}}(k_{j}))$ has a representative $x_j\in \mathcal X_{P_{\alpha\beta\gamma}}(L_{i_{j-1}}(k_{j-1}), L_{i_{j}}(k_{j}); H_n)$, for all $j=1,\ldots, d$ (we use the same notation for the representative $x_j$ for convenience).   Also at least one of $x_j$ is in $\mathcal J_{\alpha\beta}^{\gamma}(H_n)$. 

 To prove this lemma, we would like to show that there is a sufficiently large integer $N$, dependent on $x_1,\ldots, x_d$, such that for all $n\geq N$, any generator $y$ appearing in the output of the product 
 \[
 \mu_{P_{\alpha\beta\gamma}}^d(x_d,\ldots, x_1; H_n) \in CF_{P_{\alpha\beta\gamma}}^*(L_{i_0}(k_0), L_{i_d}(k_d); dH_n)
 \] 
is in $\cJ_{\alpha\beta}^\gamma(dH_n)$.

 We prove by contradiction.  Suppose a generator $y$ appearing in the output can be in $\cC(dH_n)$ for infinitely many values of $n$.  Due to this assumption, $i_0, i_d \in \{\alpha,\beta\}$.  The output $\mu^d_{P_{\alpha\beta\gamma}}(x_1,\ldots, x_d)$ is given by the count of index $(2-d)$ pseudoholomorphic discs (with a modified almost-complex structure) with boundaries on $\phi_n^d(L_{i_0}(k_0)), \phi_n^{d-1}(L_{i_1}(k_1)),\ldots$, and $L_{i_d}(k_d)$ and with strip-like ends converging to intersection points $\phi_n^{d-1}(x_1)\in \phi_n^d(L_{i_0}(k_0))\cap \phi_n^{d-1}(L_{i_1}(k_1))$, $\ldots$, $x_d \in \phi_n^1(L_{i_{d-1}}(k_{d-1}))\cap L_{i_d}(k_d)$, and $y \in \phi_{n}^d(L_{i_0}(k_0))\cap L_{i_d}(k_d)$.

 From now on we will only consider the universal cover of $e_{\alpha\beta}$ and lifts of all Lagrangians in $e_{\alpha\beta}$ to this universal cover.   We keep the same notation for convenience.

 The boundary of a holomorphic disc satisfying the above traces out two curves on $e_{\alpha\beta}$ starting at $y$, each of which is connected and piecewise smooth.  One curve $C_1$ consists of boundary arcs in $\phi_n^d(L_{i_0}(k_0))$, $\phi_n^{d-1}(L_{i_1}(k_1))$, $\ldots$, $\phi_n^{d-c_1}(L_{i_{c_1}}(k_{c_1}))$.  The other curve $C_2$ consists of boundary arcs in $L_{i_d}(k_d)$, $\phi_{n}^1(L_{i_{d-1}}(k_{d-1}))$, $\ldots$, $\phi_{n}^{c_2}(L_{i_{d-{c_2}}}(k_{d-{c_2}}))$.  We choose the largest possible $c_1$ and $c_2$ so that the four conditions below are satisfied.  We list the first three now and the fourth later:
\begin{enumerate}
	\item all of $i_0,\ldots, i_{c_1}, i_d, i_{d-1},\ldots, i_{d-{c_2}} \in \{\alpha, \beta\}$;
	\item $C_1$ and $C_2$ do not contain any input intersection points that are not in $e_{\alpha\beta}$;
	\item $d-c_2> c_1$.
\end{enumerate}
Note that if the boundary of the holomorphic disc leaves $e_{\alpha\beta}$ and enters back into $e_{\alpha\beta}$ again, then it must go through an intersection point outside $e_{\alpha\beta}$.  This is because if it doesn't go through an intersection point outside $e_{\alpha\beta}$, then it must create a boundary branch point by backtracking along a Lagrangian, but a rigid holomorphic disc does not have any boundary branch points.  Hence, $C_1$ and $C_2$ are connected, and they don't leave $e_{\alpha\beta}$ and then enter back.

 A fourth condition for choosing $c_1$ and $c_2$ becomes necessary when  $C_1$ and $C_2$ satisfying conditions (1)-(3) intersect at an input point, which happens if and only if the entire holomorphic disc is contained in $e_{\alpha\beta}$ with its boundary being the closed loop $C_1\cup C_2$.  In this case, $d-c_2-1=c_1$.   Let $\tau_c=\min\{\tau|(\tau,\psi)\in C_1\cup C_2)\}$.  Every Lagrangian involved has the property that its lift intersects each fiber of the universal cover of $e_{\alpha\beta}$ at only one point.  Hence $\tau_c$ must be the $\tau$ coordinate of an intersection point $\tilde x_c=\phi_n^{d-c}(x_c)\in \phi_n^{d-(c-1)}(L_{i_{c-1}}(k_{c-1}))\cap \phi_n^{d-c}(L_{i_{c}}(k_{c}))$, i.e. $\tau_c=\tau(\tilde x_c)$. To summarize, we require that
\begin{itemize}
	\item [(4) ] if the holomorphic disc is contained in $e_{\alpha\beta}$, then choose $d-c_2-1=c_1=c$ where $\tau(\tilde x_c)=\min\{\tau|(\tau,\psi)\in C_1\cup C_2\}$.
\end{itemize}
Figure \ref{fig: ideal0} illustrates the subset of the holomorphic disc which lies inside $e_{\alpha\beta}$ with boundary $C_1$ and $C_2$. \\
\begin{figure}[h]
\centering
	\scalebox{0.6}{\includegraphics{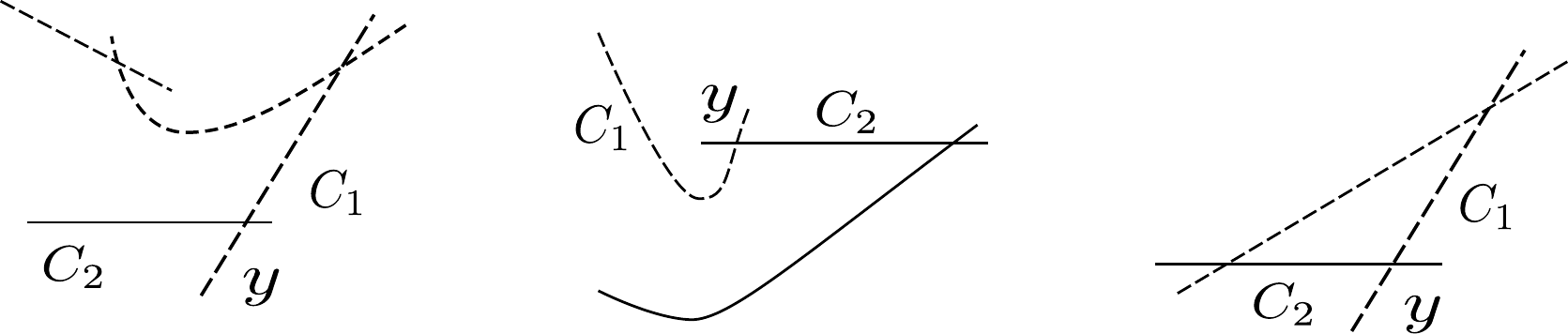}}
	\caption{The subset of the holomorphic disc which lies inside $e_{\alpha\beta}$ has boundary on $C_1\cup C_2$. The dashed curves represent $C_1$. The right-most picture illustrates the situation where the curves $C_1$ and $C_2$ intersect at an input point.}
	\label{fig: ideal0}
\end{figure}
	Observe for each $\tau<\tau(y)$, the fiber over $\tau$ intersects $C_1$ and $C_2$ at no more than one unique point $(\tau, \psi_1(\tau))\in C_1$ and one unique point $(\tau, \psi_2(\tau))\in C_2$.  We show this observation is true by contradiction.  If the fiber of the universal cover over some $\tau<\tau(y)$ intersects $C_1$ at more than one point, then an interior point of $C_1$ must be an intersection point $\tilde x_b=\phi_n^{d-b}(x_b)\in \phi_n^{d-(b-1)}(L_{i_{b-1}}(k_{b-1}))\cap \phi_n^{d-b}(L_{i_{b}}(k_{b}))$ at which the $\tau$ coordinate of $C_1$ backtracks, meaning that $\tau(\tilde x_b)$ is the minimum value of $\tau$ in an open neighborhood of $\tilde x_b$ in $C_1$. The lift of each Lagrangian intersects each fiber (level of $\tau$) of the universal cover of $e_{\alpha\beta}$ at only one point, the rigid holomorphic disc has no branch point, and each corner of the disc is convex.  For these reasons, the $\tau$ coordinate cannot backtrack more than once along $C_1\cup C_2$  , and where it backtracks, $\tau(\tilde x_b)$ is actually the minimum value of $\tau$ achieved by the holomorphic disc, i.e. the holomorphic disc is contained in $\{(\tau,\psi)\in e_{\alpha\beta} | \tau \geq \tau(\tilde x_b)\}.$  From property (4), $b=c$ and $\tilde x_b$ is the end point of $C_1$, contradicting $\tilde x_b$ being an interior point of $C_1$.  The same reasoning can be applied for $C_2$.

 Let's choose $N\gg |k_0|,\ldots, |k_d|$, then for every $\tau\in (0,\tau(y))$, the absolute value of the slope of the tangent line to $C_1$ at $(\tau, \psi_1(\tau))$ is greater than that of $C_2$ at $(\tau, \psi_2(\tau))$.  Also note that from inside the holomorphic disk, the angle between tangent lines $T_y\phi_{n}^d(L_{i_0}(k_0))$ and $T_yL_{i_d}(k_d)$ must be greater than $\pi/2$ due to the ordering of the boundary Lagrangians and the assumption that $n> |k_0|, |k_d|$.

 We want to show for $n>N$ sufficiently large, the holomorphic disc under consideration cannot have any input in $\mathcal J^\gamma_{\alpha\beta}$.  We analyse two cases. 

\emph{Case 1: the holomorphic disc is contained in $e_{\alpha\beta}$.}   Use the same notation that we used when explaining property (4).  The holomorphic disc is contained in $\{\tau\geq \tau(\tilde x_c)\}$ for an input $\tilde x_c$.  The slope of the tangent lines $T_{\tilde x_c} \phi_n^{d-(c-1)}(L_{i_{c-1}}(k_{c-1}))$ and $T_{\tilde x_c} \phi_n^{d-c}(L_{i_{c}}(k_{c}))$ have the same sign, i.e. negative if $\tau(\tilde x_c)\in (0,1)$ and positive if $\tau(\tilde x_c)\in (1,3)$.  Thus the angle between these tangent lines, from inside the holomorphic disc, must be less than $\pi/2$.  The input $\tilde x_c$ must be in $\mathcal C_{\alpha\beta}(H_n)$, i.e. $\tau(\tilde x_c)>1$, because the angle at any input in the positively wrapped region is less than $\pi/2$ and in the negatively wrapped region is greater than $\pi/2$.  This is due to the ordering of the boundary Lagrangians and the assumption that $n\gg |k_0|,\ldots, |k_d|$. So we just have a triangle in $\{\tau>1\}\subset e_{\alpha\beta}$ with no inputs in $\mathcal J^\gamma_{\alpha\beta}$. (See Fig. \ref{fig: ideal}a.)

 \emph{Case 2: the holomorphic disc is not entirely contained in $e_{\alpha\beta}$.}  As explained before, for every $\tau <\tau(y)$, the fiber over $\tau$ intersects $C_1$ once and $C_2$ once, with $\psi_1(\tau)>\psi_2(\tau)$ as illustrated in Figure \ref{fig: ideal}b.  Note that all boundary Lagrangians $\phi_n^{d-j}(L_{i_j}(k_j))$ are dependent on $n$, so are $\psi_1, \psi_2$. The values of $\psi_1(0)$ and $\psi_2(0)$ are dictated by inputs outside of $e_{\alpha\beta}$, so they will stay almost constant as a function of $n$.  Indeed, $\psi_1(0)$ and $\psi_2(0)$ are determined by the remaining portion of the boundary of the disc (other than $C_1\cup C_2$).  As $n$ varies and the inputs $\tilde x_{c_1},\ldots, \tilde x_{d-c_2}$ move by continuation, this boundary curve varies by a homotopy inside the cylindrical ends and remains constant inside the pants, and in particular $\psi_1(0)$ and $\psi_2(0)$ remain constant.   Hence $\psi_1(0)>\psi_2(0)$ always, and their difference remains bounded no matter how large $n$ gets.  However, as $n$ gets large enough, $\psi_1(1)<\psi_2(1)$ because the absolute value of the slope of the tangent line to $C_1$ at each $(\tau, \psi_1(\tau))$ is greater than that of $C_2$ at $(\tau, \psi_2(\tau))$ and the slopes increase with $n$.  Hence  $C_1$ crosses $C_2$ before reaching $\tau=1$, which contradicts an assumption that $y\in \cC_{\alpha\beta}(dH_n)$.  Hence, we can pick a sufficiently large $N$, so that for all $n\geq N$, there does not exist any rigid holomorphic disc with the given inputs $x_1,\ldots, x_d$, at least one of which lies in $\mathcal J_{\alpha\beta}^{\gamma}(H_n)$ and its output in $\mathcal C_{\alpha\beta}(dH_n)$.

\begin{figure}[h]
\centering
	\scalebox{0.6}{\includegraphics{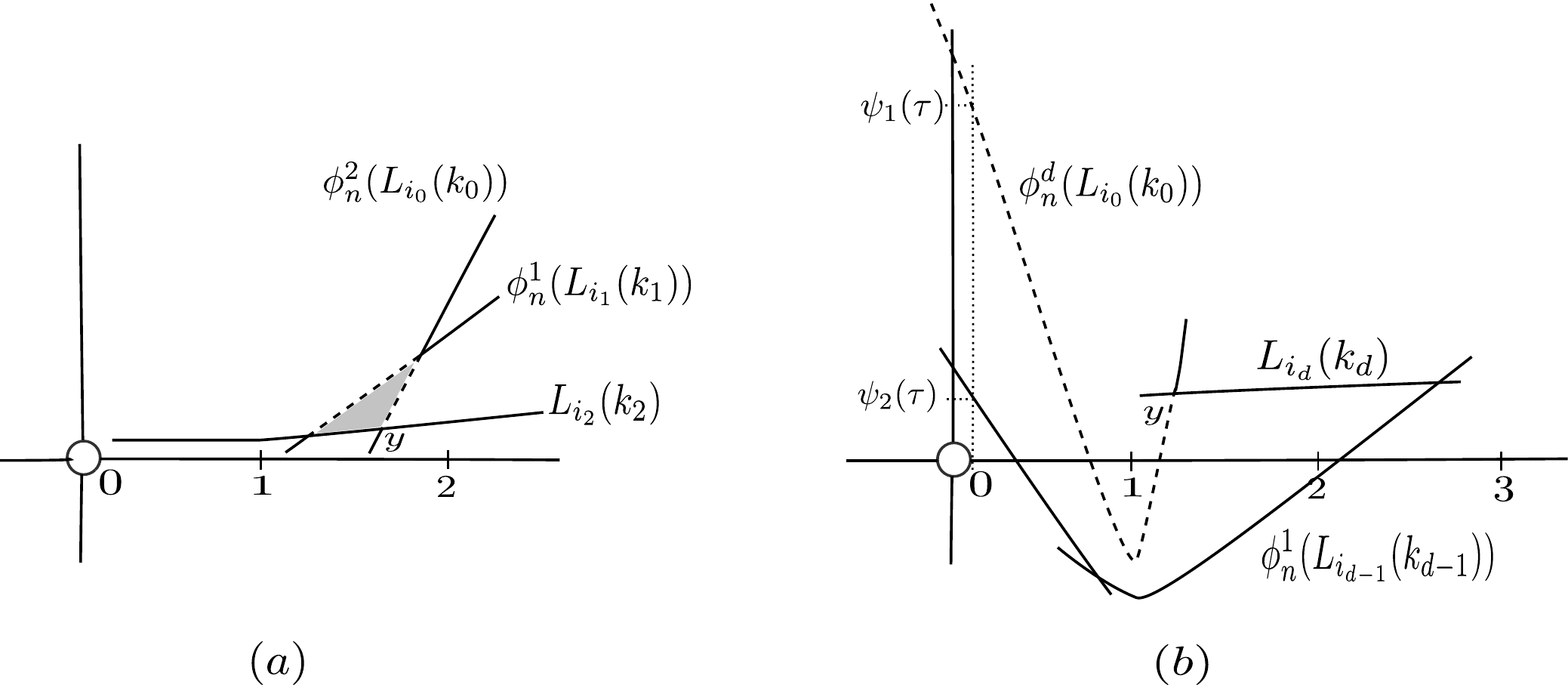}}
	\caption{(a) A generic holomorphic disc contained in $e_{\alpha\beta}$ as described in Case 1. (b) The portion in $e_{\alpha\beta}$ of a holomorphic disc that is not entirely contained in $e_{\alpha\beta}$ as described in Case 2.   Same as in Fig. \ref{fig: ideal0}, the dashed curves represent $C_1$.}
	\label{fig: ideal}
\end{figure}
\end{proof}

\begin{corollary}
There is a restriction map \[\rho_{\alpha\beta}^{\gamma}: \bigoplus_{L_i,L_j} CW^*_{P_{\alpha\beta\gamma}}(L_i, L_j)\to \bigoplus_{L_i, L_j} CW^*_{\mathcal C_{\alpha\beta}}(L_i, L_j)\] which is a quotient by $CW^*_{\mathcal J_{\alpha\beta}^\gamma}(L_i, L_j)$, and it is compatible with $A_\infty$-structures with no higher order terms.
\end{corollary}

\begin{proof}
It follows from Lemma \ref{lemma: ideal}.
\end{proof}

\begin{lemma} \label{lemma: global} 
Given inputs $x_1 \in \mathcal X(L_{i_0}(k_0), L_{i_1}(k_1)),\ldots,x_d\in \mathcal X(L_{i_{d-1}}(k_{d-1}), L_{i_d}(k_d))$ and any generator $y\in \mathcal X(L_{i_0}(k_0), L_{i_d}(k_d))$ appearing in the output $\mu^d(x_1,\ldots, x_d)$, there is a single pair-of-pants $P_0\subset H$ and a sufficiently large $N$, dependent on $x_1,\ldots, x_d, y$, such that for all $n\geq N$, there are representatives $x_j\in \mathcal X_{P_0}(L_{i_{j-1}}(k_{j-1}), L_{i_{j}}(k_{j}); H_n), y\in  \mathcal X_{P_0}(L_{i_0}(k_0), L_{i_d}(k_d); dH_n)$ for all $j=1,\ldots, d$.  Hence, $\mu^d(x_1,\ldots,x_d)=\mu^d_{P_0}(x_1,\ldots,x_d)$.

\end{lemma}
\begin{proof}

 From the assumptions, there is a sufficiently large $N_0$, dependent on $x_1,\ldots, x_d, y$, such that for all $n\geq N_0$, $x_j$ has a representative $x_j\in \mathcal X(L_{i_{j-1}}(k_{j-1}), L_{i_{j}}(k_{j}); H_n)$, for all $j=1,\ldots, d$, and $y\in \mathcal X_{P_0}(L_{i_0}(k_0), L_{i_d}(k_d); dH_n)$.  Suppose the holomorphic disc has at least one of its vertices is contained in a pair of pants $P_{\alpha\beta\gamma}$, but not in an adjacent pair-of-pants $P_{\alpha\beta\eta}$, and at least one of its vertices is contained in  $P_{\alpha\beta\eta}$ but is not contained in $P_{\alpha\beta\gamma}$.  We will show by contradiction that such a holomorphic disc cannot exist.  The same arguments also exclude discs in which an edge goes into $P_{\alpha\beta\eta}$ to reach an vertex lying in another pair-of-pants even further away from $P_{\alpha\beta\gamma}$. Hence, all intersections points are actually in a single pair-of-pants.  Note that $P_{\alpha\beta\gamma}$ and $P_{\alpha\beta\eta}$ overlap on the cylinder $e_{\alpha\beta}$.  We focus on the portion of the disc that lies in $e_{\alpha\beta}$, and the manner in which it escapes into both ends of the cylinder.  Again, we view perturbed pseudoholomorphic discs as ordinary pseudoholomorphic discs with perturbed boundaries as explained in Section \ref{subsec: background}. 

None of the inputs can be in $\mathcal C_{\alpha\beta}(H_n)$, which is in the positively wrapped region shared by both pairs of pants.  This is due to what we noticed before (in the proof of Lemma \ref{lemma: ideal}) that the angle at any input in the positively wrapped region is less than $\pi/2$.  Then by convexity, the holomorphic disc will not go beyond that input point, i.e. it does not escape into both $P_{\alpha\beta\gamma}\backslash e_{\alpha\beta}$ and $P_{\alpha\beta\eta}\backslash e_{\alpha\beta}$.  Hence, in the region $[1,3]\times S^1\subset e_{\alpha\beta}$, either there's no vertex at all, or there is one output generator.

First, assume the output generator $y$ is not in the positively wrapped region, then there's no intersection point at all in this region.  Consider the universal cover of the cylinder $e_{\alpha\beta}$ and lifts of all Lagrangians in $e_{\alpha\beta}$ to this universal cover.  We can find two boundary arcs in Lagrangians $\phi_n^r(L_{i_{d-r}}(k_{d-r}))$ and $\phi_n^s(L_{i_{d-s}}(k_{d-s}))$ that are non-parallel lines and that they intersect the fiber of the universal cover over $\tau=2$ at $\psi_1(2)$ and $\psi_2(2)$, respectively.  The Hamiltonian perturbations we use actually keep $\psi_1(2)$ and $\psi_1(2)$ constant as $n$ varies, and the input marked points move along by continuation.  So, these two lines will cross each other for all $n$ that is large enough (e.g. for all $n>|\psi_1(2)-\psi_2(2)|+2\pi |k_{d-r}-k_{d-s}|$ will do).  So there is no holomorphic discs with the assumed inputs and output for sufficiently large $n$.  

 Now, suppose the output $y$ is in the positively wrapped region.  In the universal cover of $e_{\alpha\beta}$, the disc will look like one of the configurations in Figure \ref{fig: global}. The output point is bounded by Lagrangians $L_{i_d}(k_d)$ and $\phi_n^d(L_{i_0}(k_0))$, which have the biggest and smallest slopes, respectively.  When we increase $n$ to be large enough, $\psi_1(2)$ and $\psi_2(2)$ remain constant as before, and the boundary arcs will cross each other again inside the positively wrapped region, which rules out the existence of such holomorphic discs.
\begin{figure}[h]
\centering
	\scalebox{0.6}{\includegraphics{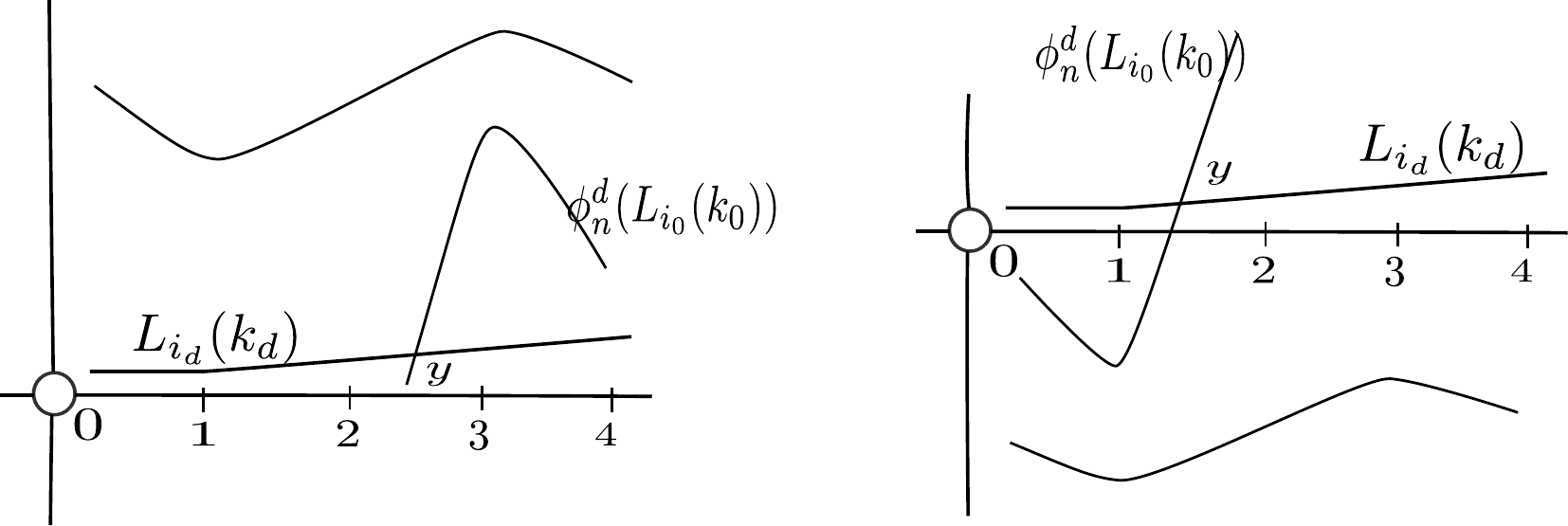}}
	\caption{Two (very similar) types of configurations, both with the output $y$ in the positively wrapped region.}
	\label{fig: global}
\end{figure}
\end{proof}

\begin{theorem} \label{thm: decomposition}
 The wrapped Fukaya category $\mathcal W(H)$ is split-generated by the Lagrangian objects $L_{\alpha}(k)$ where $\alpha\in A$ and $k\in \mathbb Z$.   In a suitable model for $\mathcal W(H)$, the morphism complex between any two objects, $L_{\alpha}(k)$ and $L_{\beta}(l)$, is generated by 
\vspace{-0.02in}
\begin{equation}
\mathcal X (L_{\alpha}(k), L_{\beta}(l)) =\left(\bigcup \mathcal X_{P_{\alpha\beta\gamma}}(L_{\alpha}(k), L_\beta(l))\right)/\sim
\end{equation}
where $\mathcal X_{P_{\alpha\beta\gamma}}(L_{\alpha}(k), L_\beta(l))$ is the set of generators of the morphism complexes in $\mathcal W(P_{\alpha\beta\gamma})$, and the equivalence relation identifies $x\in \mathcal X_{P_{\alpha\beta\gamma}\setminus \mathcal J_{\alpha\beta}^{\gamma}}$ with $\tilde x\in \mathcal X_{P_{\alpha\beta\eta}\setminus\mathcal J_{\alpha\beta}^{\eta}}$ whenever $\rho_{\alpha\beta}^{\gamma}(x)=\rho_{\alpha\beta}^\eta(\tilde x)$, where $\rho_{\alpha\beta}^{\gamma}$, $\rho_{\alpha\beta}^{\eta}$ are restriction maps  to the cylinder $\mathcal C_{\alpha\beta}=P_{\alpha\beta\gamma}\cap P_{\alpha\beta\eta}$.  Moreover, in this model, the $A_{\infty}$-products in $\mathcal W(H)$ are given by those in the pairs of pants.
\end{theorem}
\noindent To clarify the last statement in Theorem \ref{thm: decomposition}, any product $\mu^d(x_1,\ldots, x_d)$ vanishes unless $x_1,\ldots, x_d$ all live in a single pair of pants $P_{\alpha\beta\gamma}$ (by Lemma \ref{lemma: global}), in which case we take the product inside $P_{\alpha\beta\gamma}$.  If in fact the generators $x_1,\ldots, x_d$ all live inside the same bounded cylinder $\mathcal C_{\alpha\beta}$, then so do the discs with these inputs (by the preceding arguments in Lemma \ref{lemma: ideal}), and the calculations $A_\infty$-products in $P_{\alpha\beta\gamma}$ and $P_{\alpha\beta\eta}$ give the same answer, i.e. the product is compatible with the equivalence relation. On the other hand, if at least one of the given inputs $x_1,\ldots, x_d$ lives in $\mathcal J_{\alpha\beta}^\gamma$, then so does the generators appearing in the output by Lemma \ref{lemma: ideal}, and so there is no question of compatibility with the equivalence relation.

\begin{proof} 
Follows from Lemma \ref{lemma: global}.
\end{proof}

\begin{figure}[h]
\centering
	\scalebox{1}{\includegraphics{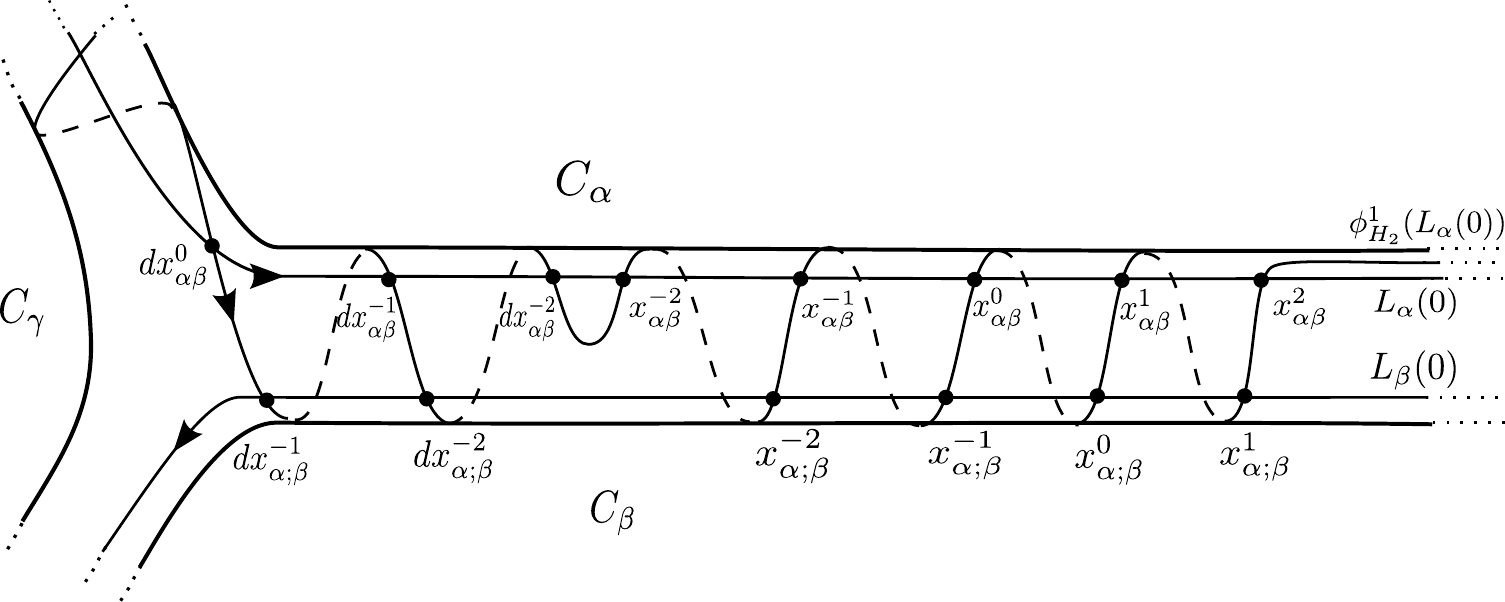}}
	\caption{Generators in $\mathcal C_{\alpha\beta}(L_{\alpha}(0), L_{\alpha}(0); H_2)$ and $\mathcal C_{\alpha\beta}(L_{\alpha}(0), L_{\beta}(0); H_2)$; assuming $d_{\alpha,\beta}=-1$.}
	\label{fig: pantsMorphism}
\end{figure}

We label the $2n+n_{\alpha\beta}(l-k)+1$ generators in $\mathcal C_{\alpha\beta}(L_\alpha(k), L_\alpha(l); H_n)$ by 
\[
x_{\alpha\beta}^{-n},\ldots,x_{\alpha\beta}^{-1}, \ x_{\alpha\beta}^0,\  x_{\alpha\beta}^1,\ldots,x_{\alpha\beta}^{n+n_{\alpha\beta}(l-k)}, 
\]
successively with the minimum of the Hamiltonian labeled as $x_{\alpha\beta}^0$ as illustrated in Figure \ref{fig: pantsMorphism}.   Each Floer complex $CW_{P_{\alpha\beta\gamma}}^*(L_\alpha(k), L_{\beta}(l))$ is a $CW_{P_{\alpha\beta\gamma}}^*(L_{\alpha}(k), L_{\alpha}(k))-CW_{P_{\alpha\beta\gamma}}^*(L_{\beta}(l), L_{\beta}(l))$-bimodule.  We label the $2n+n_{\alpha\beta}(l-k)+d_{\alpha,\beta}+1$ generators in $\mathcal C_{\alpha\beta}(L_\alpha(k), L_\beta(l); H_n)$ by 
\[
x_{\alpha; \beta}^{-n}, \ldots, x_{\alpha;\beta}^{-1}, \ x_{\alpha;\beta}^0, x_{\alpha; \beta}^1, \ldots, x_{\alpha;\beta}^{n+n_{\alpha\beta}(l-k)+d_{\alpha,\beta}}.
\]

The object $L_\alpha(k)$ is equivalent to $L_{\alpha}(l)$ in the category $\mathcal W(P_{\alpha\beta\gamma})$, and similarly in $\mathcal W(P_{\alpha\beta\eta})$.  However, they can be distinguished as the object for which the generator $x_{\alpha\beta}^i \in  \mathcal C_{\alpha\beta}(L_{\alpha}(k), L_\alpha(l))$ in the image of the restriction functor $\rho_{\alpha\beta}^\gamma$ is identified with the generator $\tilde x_{\alpha\beta}^{n_{\alpha\beta}(l-k)-i} \in \mathcal C_{\alpha\beta}(L_{\alpha}(k), L_\alpha(l))$ in the image of $\rho_{\alpha\beta}^\eta$. In addition, $x_{\alpha;\beta}^i$ is identified with $\tilde x_{\alpha;\beta}^{n_{\alpha\beta}(l-k)+d_{\alpha,\beta}-i}$.  

\begin{remark} The techniques presented in this section, especially those in the proof of Lemmas \ref{lemma: ideal} and \ref{lemma: global}, can be applied in more general situations.  Suppose we have a decomposition of a surface $\Sigma=P\cup Q$ into two components such that $P$ and $Q$ overlaps in a cylinder $C$, and let $e\subset \Sigma$ be the largest cylinder containing $C$ and we think of $e$ as a circle fibration over a line segment that is the base.  We can applied the above techniques to compute the Fukaya category of $\Sigma$ (or the wrapped Fukaya category if $\Sigma$ is open) when we have a choice of split-generating Lagrangians $(L_i)_{i\in I}$ such that each $L_i$ satisfies the following conditions.
\begin{itemize}
\item The intersection $L_i\cap e$ is a union of disjoint arcs such that each arc intersects each circle fiber of $e$ just once.

\item  For any two of the above arcs, the portion of $L_i$ in the complement of $e$ that is connected by these two arcs cannot be homotopically trivial. 

\end{itemize}

\end{remark}

\section{The Landau-Ginzburg mirror}

\subsection{Generating objects.} 

We consider the Landau-Ginzburg mirror $(Y,W)$ mentioned at the end of Section \ref{sec: tropical}.  The category of matrix factorizations $MF(Y,W)$ is defined to be the Verdier quotient of $MF^{naive}(Y,W)$ by the subcategory $Ac(Y,W)$ of acyclic elements \cite{AAEKO13, LP11, Or11}.  The objects of $MF^{naive}(Y,W)$ are 
\[
T:=
\xymatrix{
T_1\ar@<.25pc>[r]^-{t_1} & T_0\ar@<.25pc>[l]^-{t_0}
},
\]
where $T_1, T_0$ are locally free sheaves of finite rank on $Y$, and $t_1, t_0$ are morphisms satisfying $t_{i+1}\circ t_i=W\cdot id_{T_i}$.  The morphism complex
\[
\sHom({S}, {T})=\bigoplus_{i,j} \sHom(S_i,T_j)
\] 
is graded by $(i+j)\text{mod } 2$,  that is
\[\sHom^{even}({S}, {T})=\sHom(S_0,T_0) \oplus \sHom(S_1,T_1),\]
\[ \sHom^{odd}({S}, {T})=\sHom(S_0,T_1) \oplus \sHom(S_1,T_0).\]
The differential on this complex is $d:f\mapsto t\circ f-(-1)^{|f|} f\circ s$.  The equivalence $MF(Y,W)\stackrel{\sim}{\rightarrow} D^b_{sg}(D)$ is given by $T\mapsto \Coker(t_1)$, which is a sheaf on $D=W^{-1}(0)$ because it is annihilated by $W$.  

For each irreducible divisor $D_{\alpha}=\{t_{\alpha;1} =0 \}$, $t_{\alpha;1}:\Os(-D_\alpha) \to \Os$ is an injective sheaf homomorphism.  It also induces a map $t_{\alpha;1}:\Os(-D_{\alpha})(k)\to \Os(k)$ for any $k\in \mathbb Z$ (with the notation $\Os(k):=\Os(1)^{\otimes k}$, where $\Os(1)$ is the polarization (ample line bundle) on $Y$ determined by the polytope $\Delta_Y$).   For each $\alpha \in A, k\in \mathbb Z$, we consider
\begin{equation}
{T_{\alpha}(k)}:=
\xymatrix{
{\Os(-D_\alpha)(k)}\ar@<.25pc>[r]^-{t_{\alpha; 1}} & {\Os(k) }\ar@<.25pc>[l]^-{t_{\alpha; 0}}
}.
\end{equation}
in $MF(Y, W)$.  The object ${T_{\alpha}(k)}$ in $D^b_{sg}(D)$ is  $\Coker (t_{\alpha;1})=\Os_{D_\alpha}(k)$. An argument similar to that in Section 6 of \cite{AAEKO13} implies that:

\begin{lemma} $D^b_{sg}(D)$ is split-generated by objects $\Os_{D_\alpha}(k), k\in \mathbb Z, \alpha\in A$.
\end{lemma}

\begin{proof}
We imitate the proof of Lemma 6.3 in \cite{AAEKO13}.  Because $D^b_{sg}(D)$ is a quotient of $D^b\Coh(D)$, it is sufficient to show that $\Os_{D_\alpha}(k)$'s generate the category $D^b\Coh(D)$.  Denote by  $\cT \subset D^b\Coh(D)$ the full triangulated subcategory split-generated by these objects.  

Since the divisors $D_\alpha$ are precisely the irreducible components of $D$, it is sufficient to prove that that $D^b_{D_\alpha}\Coh(D)\subset \cT$ for all $\alpha\in A$, where $D^b_{D_\alpha}\Coh(D)$ is the full subcategory consisting of complexes with cohomology supported on $D_\alpha$.  According to Theorem 4 or \cite{Or09}, the $\Os_{D_\alpha}(k)$'s, i.e. powers of an ample line bundle, split-generate the subcategory of the perfect complexes $\mathfrak{Perf}(D_\alpha)$, which is equal to $D^b\Coh(D_\alpha)$ since $D_\alpha$ is smooth.   The restriction map from a complex whose cohomology is supported on $D_\alpha$ to the complex that is actually supported on $D_\alpha$ is a quasi-isomorphism, hence it's an isomorphism in the derived category.  So, $\Os_{D_\alpha}(k)$'s split-generate $D^b\Coh(D_\alpha)$ implies that they also split-generate $D^b_{D_\alpha}\Coh(D)$.
\end{proof}

\subsection{\v{C}ech model and  homotopic restriction functors}\label{sec: cech}

Let 
\[\mathfrak U=\{U_{\alpha\beta\gamma}=\mathbb C[x_{\alpha\beta}, x_{\alpha\gamma}, x_{\beta\gamma}], W=x_{\alpha\beta}x_{\alpha\gamma}x_{\beta\gamma}\}_{\alpha,\beta,\gamma\in A \text{ adjacent }}\]
be a finite covering of $(Y,W)$ by affine toric subsets. The moment polytope of each $U_{\alpha\beta\gamma}$ corresponds to a corner of $\Delta_Y$ coming from $C_\alpha, C_\beta, C_\gamma$.  The divisor $D$ restricted to $U_{\alpha\beta\gamma}$ is equal to the restriction $(D_{\alpha}\cup D_{\beta}\cup D_{\gamma})|_{U_{\alpha\beta\gamma}}$, with $D_{\alpha}|_{U_{\alpha\beta\gamma}}=\{x_{\beta\gamma}=0\}$, $D_{\beta}|_{U_{\alpha\beta\gamma}}=\{x_{\alpha\gamma}=0\}$, $D_{\gamma}|_{U_{\alpha\beta\gamma}}=\{x_{\alpha\beta}=0\}$, and $W$ is locally given by the product of these affine coordinates.  The object ${T_{\alpha}(k)}$ restricted to this local chart is 
\[{T_{\alpha}(k)}(U_{\alpha\beta\gamma})=\xymatrix{
{\Os(-D_\alpha)(k)(U_{\alpha\beta\gamma})}\ar@<.25pc>[r]^-{x_{\beta\gamma}} & 
{\Os(k) (U_{\alpha\beta\gamma})}\ar@<.25pc>[l]^-{x_{\alpha\beta}x_{\alpha\gamma}}} ,\]
and similarly for ${T_{\beta}(k)}(U_{\alpha\beta\gamma})$ and ${T_{\gamma}(k)}(U_{\alpha\beta\gamma})$.  Note that $T_{\alpha}(k)(U_{\alpha\beta\gamma})$ is actually equivalent to $T_{\alpha}(0)(U_{\alpha\beta\gamma})$ in the category of matrix factorizations on the local chart $U_{\alpha\beta\gamma}$.  

For each coordinate, there is a restriction functor, simply to the complement of the coordinate plane $\{x_{\alpha\beta}=0\}$,
\[
\sigma_{\alpha\beta}^\gamma: MF \left(U_{\alpha\beta\gamma}, x_{\alpha\beta}x_{\alpha\gamma}x_{\beta\gamma}\right) \to MF\left( \mathbb C^*[x_{\alpha\beta}]\times \mathbb C[x_{\alpha\gamma},x_{\beta\gamma}], x_{\alpha\beta}x_{\alpha\gamma}x_{\beta\gamma}\right)\cong D^b(\Coh(\mathbb C^*[x_{\alpha\beta}])),
\]
where $\mathbb C^*[x_{\alpha\beta}]$ is the $\mathbb C^*$ with ring of functions $\mathbb C[x_{\alpha\beta}^{\pm}]$.
The equivalence   
\[
MF\left( \mathbb C^*[x_{\alpha\beta}]\times \mathbb C[x_{\alpha\gamma},x_{\beta\gamma}], x_{\alpha\beta}x_{\alpha\gamma}x_{\beta\gamma}\right)\cong D^b(\Coh(\mathbb C^*[x_{\alpha\beta}]))
\]
is due to the Kn\"{o}rrer periodicity theorem \cite{Or04, Kn87}.  In this section, we will check the commutativity of the following diagram up to the first order

\begin{equation}\label{diagram}
\xymatrix@R2pc@C3pc{
\mathcal A:=\mathcal W \left(\coprod P_{\alpha\beta\gamma}\right) \ar[r]^\rho \ar[d]_{\cQ_a} & \mathcal B:=\mathcal W \left(\coprod \mathcal C_{\alpha\beta}\right) \ar[d]_{\cong}\\
\mathcal A':=MF \left(\coprod (U_{\alpha\beta\gamma}, x_{\alpha\beta}x_{\alpha\gamma}x_{\beta\gamma})\right) \ar[r]^{\ \ \sigma} & \mathcal B':=D^b(\Coh(\coprod\mathbb C^*[x_{\alpha\beta}])).
}
\end{equation}

In this diagram, $\rho$ is the restriction functor from Section \ref{section: PPD}.  Due to HMS for pairs of pants (as shown in \cite{AAEKO13}), the category $\mathcal A$ is quasi-equivalent to $\mathcal A'$ via the $A_\infty$-functor $\cQ_a$.  The categories $\mathcal B$ and $\mathcal B'$ are isomorphic due to HMS for cylinders---that their underlying morphism spaces are isomorphic and, for both of them, the higher products on the morphism spaces are all identically zero (see e.g. \cite{Au13}).   So,  we can think of $\cB$ and $\cB'$ as the same category for convenience and denote it by $\cB$.  We want to show that the two $A_\infty$-restriction functors 
\begin{equation}
\cF:= \rho,\  \ \ \cG :=\sigma\circ \cQ_a: \ \ \mathcal A \to \mathcal B
\end{equation}
are equal up to the first order.

The following diagram shows the morphisms $(f_1,f_0)\in \sHom^0(T_{\alpha}(k)(U_{\alpha\beta\gamma}), T_{\alpha}(l)(U_{\alpha\beta\gamma}))$ and $(h_0,h_1)\in \sHom^1(T_{\alpha}(k)(U_{\alpha\beta\gamma}), T_{\alpha}(l)(U_{\alpha\beta\gamma}))$,

\[\xymatrix@R3pc@C3pc{
{\Os(-D_\alpha)(k)(U_{\alpha\beta\gamma})}\ar@<.4pc>[r]^-{x_{\beta\gamma}} \ar[d]_{f_1} \ar[rd]^{\ \ h_0}& 
{\Os(k) (U_{\alpha\beta\gamma})}\ar@<-.1pc>[l]^-{x_{\alpha\beta}x_{\alpha\gamma}} \ar[d]_{f_0} \ar[ld]^{\ \ h_1}\\
{\Os(-D_\alpha)(l)(U_{\alpha\beta\gamma})}\ar[r]^-{x_{\beta\gamma}}&
{\Os(l) (U_{\alpha\beta\gamma}) }\ar@<.3pc>[l]^-{x_{\alpha\beta}x_{\alpha\gamma}}.
}\]
The differential maps are
\[(f_0, f_1)\mapsto (x_{\alpha\beta}x_{\alpha\gamma}(f_0-f_1), x_{\beta\gamma}(f_1-f_0)),\] 
\[(h_0, h_1)\mapsto (x_{\alpha\beta}x_{\alpha\gamma}h_1+x_{\beta\gamma}h_0, x_{\alpha\beta}x_{\alpha\gamma}h_1+x_{\beta\gamma}h_0).\]  
Hence in cohomology
\[\Hom^0(T_{\alpha}(k)(U_{\alpha\beta\gamma}), T_{\alpha}(l)(U_{\alpha\beta\gamma}))\cong \Os(l-k)(U_{\alpha\beta\gamma})/(x_{\beta\gamma}, x_{\alpha\beta}x_{\alpha\gamma}),\]
\[\Hom^1(T_{\alpha}(k)(U_{\alpha\beta\gamma}), T_{\alpha}(l)(U_{\alpha\beta\gamma}))=0.\]
Restricting via $\sigma_{\alpha\beta}^\gamma$ gives 
\[\Hom^0(T_{\alpha}(k)(U_{\alpha\beta\gamma}), T_{\alpha}(l)(U_{\alpha\beta\gamma}))\cong \Os(l-k)(U_{\alpha\beta\gamma})/(x_{\beta\gamma}, x_{\alpha\gamma})\cong \Os_{D_{\alpha\beta}|U{\alpha\beta\gamma}}(l-k),\]
where $D_{\alpha\beta}|U_{\alpha\beta\gamma}=(D_{\alpha}\cap D_{\beta})|U_{\alpha\beta\gamma}\cong \mathbb C^*[x_{\alpha\beta}]$.  
Hence $T_\alpha(k)(U_{\alpha\beta\gamma})$, $T_\alpha(l)(U_{\alpha\beta\gamma})$ and $T_\alpha(k)(U_{\alpha\beta\eta})$,  $T_\alpha(l)(U_{\alpha\beta\eta})$ are respective objects in $MF(U_{\alpha\beta\gamma}, x_{\alpha\beta}x_{\gamma\beta}x_{\alpha\gamma})$ and $MF(U_{\alpha\beta\eta}, x_{\alpha\beta}x_{\eta\beta}x_{\alpha\eta})$ for which the generator $x_{\alpha\beta}^i\in \Os_{D_{\alpha\beta}|U{\alpha\beta\gamma}}(l-k)$ in the image of the restriction function $\sigma_{\alpha\beta}^\gamma$ is identified with $\tilde x_{\alpha\beta}^{n_{\alpha\beta}(l-k)-i}\in \Os_{D_{\alpha\beta}|U{\alpha\beta\eta}}(l-k)$ in the image of $\sigma_{\alpha\beta}^\eta$.  
Indeed, the restriction of $\Os(1)$ to $D_{\alpha\beta}$ has degree given by the length of the corresponding edge of $\Delta_Y$, i.e. $n_{\alpha\beta}$, so $\Os_{D_{\alpha\beta}}(l-k)$ has degree $n_{\alpha\beta}(l-k)$.  

A similar calculation can be carried out for $\Hom(T_{\alpha}(k)(U_{\alpha\beta\gamma}), T_\beta(l)(U_{\alpha\beta\gamma}))$.  For morphisms $(f_1,f_0)\in \sHom^0(T_{\alpha}(k)(U_{\alpha\beta\gamma}), T_{\beta}(l)(U_{\alpha\beta\gamma}))$ and $(h_0,h_1)\in \sHom^1(T_{\alpha}(k)(U_{\alpha\beta\gamma}), T_{\beta}(l)(U_{\alpha\beta\gamma}))$,
\[\xymatrix@R3pc@C3pc{
{\Os(-D_\alpha)(k)(U_{\alpha\beta\gamma})}\ar@<.4pc>[r]^-{x_{\beta\gamma}} \ar[d]_{f_1} \ar[rd]^{\ \ h_0}& 
{\Os(k) (U_{\alpha\beta\gamma})}\ar@<-.1pc>[l]^-{x_{\alpha\beta}x_{\alpha\gamma}} \ar[d]_{f_0} \ar[ld]^{\ \ h_1}\\
{\Os(-D_\beta)(l)(U_{\alpha\beta\gamma})}\ar[r]^-{x_{\alpha\gamma}}&
{\Os(l) (U_{\alpha\beta\gamma}) }\ar@<.3pc>[l]^-{x_{\alpha\beta}x_{\beta\gamma}},
}\] 
the differential maps 
\[(f_0, f_1)\mapsto (x_{\alpha\beta}x_{\beta\gamma}f_0- x_{\alpha\beta}x_{\alpha\gamma}f_1,  x_{\alpha\gamma}f_1-x_{\beta\gamma}f_0),\] 
\[(h_0, h_1)\mapsto (x_{\alpha\gamma}h_0+x_{\alpha\beta}x_{\alpha\gamma}h_1, x_{\alpha\beta}x_{\beta\gamma}h_1+x_{\beta\gamma}h_0).\]  
Hence in cohomology
\[\Hom^0(T_{\alpha}(k)(U_{\alpha\beta\gamma}), T_{\beta}(l)(U_{\alpha\beta\gamma}))=0.\]
\[\Hom^1(T_{\alpha}(k)(U_{\alpha\beta\gamma}), T_{\beta}(l)(U_{\alpha\beta\gamma}))\cong \Os(D_{\alpha})(l-k)(U_{\alpha\beta\gamma})/(x_{\alpha\gamma}, x_{\beta\gamma}, x_{\alpha\beta}x_{\alpha\gamma}, x_{\alpha\beta}x_{\beta\gamma}),\]
Restricting via $\sigma_{\alpha\beta}^\gamma$ gives 
\[\Hom^1(T_{\alpha}(k)(U_{\alpha\beta\gamma}), T_{\beta}(l)(U_{\alpha\beta\gamma}))\cong \Os(D_{\alpha})(l-k)(U_{\alpha\beta\gamma})/(x_{\alpha\gamma},x_{\beta\gamma})\cong \Os_{D_{\alpha\beta}|U_{\alpha\beta\gamma}}(D_{\alpha})(l-k),\]
The restriction functors from adjacent affine charts to their common overlap now identify $x_{\alpha;\beta}^i$ and $\tilde x_{\alpha;\beta}^{n_{\alpha\beta}+d_{\alpha,\beta}-i}$, which are generators whose degrees add up to the degree of $\Os(D_{\alpha})(l-k)_{|D_{\alpha\beta}}$, namely $n_{\alpha\beta}(l-k)+d_{\alpha,\beta}$.  

This matches with the behavior described at the end of Section \ref{sec: wFk} for the restriction functors in Floer theory, via the natural identification between cohomology-level morphisms on two sides of mirror symmetry as suggested by our notations.  Hence the functors $\mathcal F$ and $\mathcal G$ agree on cohomology as claimed.

Below we show that $\cF$ and $\cG$ are homotopic by following Sections (1d)-(1h) of \cite{Se08}.  
The $A_\infty$-functors $\cF$ and $\cG$ are themselves objects of an $A_\infty$-category $fun(\cA, \cB)$.  For $T\in \sHom_{fun(\cA,\cB)}^g(\cF,\cG)$, the differential is defined to be
\begin{equation}\label{eqn: fun}
\begin{array}{l}
\mu^1_{\nufun(\cA, \cB)}(T)^r(a_r,\ldots, a_1) \vspace{0.1in} \\

 \hspace{0.95in} ={\sum\limits_{1\leq i\leq j}} \ \ {\sum\limits_{s_1+\cdots +s_j=r}}
\mu^j_\cA  \Big(\cG^{s_j}(a_r,\ldots, a_{d-s_j+1}),\ldots, \cG^{s_{i+1}}(\ldots, a_{s_1+\cdots+s_{i}+1}),\vspace{-0.05in} \\
  \hspace{2.6in} T^{s_i}(a_{s_1+\cdots+s_i},\ldots, a_{s_1+\cdots+s_{i-1}+1}),\\
 \hspace{2.7in}  \cF^{s_{i-1}}(a_{s_1+\cdots+s_{i-1}}, \ldots),\ldots, \cF^{s_1}(a_{s_1},\ldots,a_1) \Big)\\

\hspace{1.2in} -{\sum\limits_{m,n}} \ \ T^{r-m+1}  \Big(a_d,\ldots, a_{n+m+1}, \mu_\cA^m(a_{n+m},\ldots, a_{n+1}),a_n, \ldots, a_1\Big).

\end{array}
\end{equation}   
The difference $D=\cF-\cG\in \sHom^1_{fun(\cA,\cB)}(\cF, \cG)$  is an actual natural transformation,  $\mu^1_{fun(\cA, \cB)}(D)=0$.  Note that our $\cF$ and $\cG$ act in the same way on objects.   By definition, two functors $\cF$ and $\cG$ are homotopic if $D=\mu^1_{fun(\cA, \cB)}(T)$ for  some $T\in \sHom^0_{fun(\cA,\cB)}(\cF,\cG)$ with $T^0=0$.   The pair of functors  $(\cF,\cG)$ makes the category $\mathcal B$ a $A_\infty$-bimodule over $\mathcal A$ by acting through $\cG$ on the left and $\cF$ on the right.  With this bimodule structure, the Hochschild complex $CC^1(\cA,\cB)$ is the collection of all maps
\begin{equation}\label{eq: CC1}
\sHom_\cA(X_{r-1},X_r)\otimes\cdots \otimes \sHom_\cA(X_0, X_1)\to \sHom_\cB(\cF X_0, \cG X_r)[1-r],
\end{equation}
for all $r$, with the differential being $\mu^1_{fun(\cA, \cB)}$.  So the Hochschild cohomology $HH^1(\cA, \cB)$ determines the classification of $A_\infty$-functors from $\mathcal A$ to $\mathcal B$.

Let $A=H(\cA)$ and $B=H(\cB)=\cB$ be the underlying cohomological categories, and let $F=H(\cF)$ and $G=H(\cG)$.   There is a complete decreasing filtration $F^*$ of $\sHom_{fun(\cA, \cB)}(\cF,\cG)$ by the length, i.e. the $F^r$ term consists of $T\in \sHom_{fun(\cA, \cB)}(\cF,\cG)$ such that $T^0=\cdots =T^{r-1}=0$.  This induces a spectral sequence with $E_1^{rs}=CC^{r+s}(A,B)^s$ and its $\partial$-cohomology is 
$HH^{r+s}(A, B)^s=E_2^{r,s}$, which, if converges, goes to  $HH^*(\cA, \cB)$.   For morphisms $a_1,\ldots, a_r$ in $A$, the Hochschild differential $\partial$ is given by 
\begin{equation}
(\partial t)(a_r,\ldots, a_1)=t(a_r,\ldots, a_2)F(a_1)+G(a_r)t(a_{r-1},\ldots, a_1)+\sum_n t(a_r,\ldots, a_{n+1}a_{n},\ldots, a_1).
\end{equation}
A Hochschild cohomology calculation similar to that in Lemma 3.2 of \cite{AAEKO13} shows that:
\begin{lemma}
The Hochschild cohomology $HH^1(A, B)^{1-r}=0$ for $r\geq 2$.
\end{lemma}

To show that $\cF$ and $\cG$ are homotopic, the argument is somewhat similar to Lemma 1.9 of \cite{Se08}. Assume we have built a homotopy $T$ up to a certain order, i.e. suppose we  have constructed $T^0=0, T^1, \ldots, T^{r-1}$ so that  $D=\mu^1_{fun(\cA,\cB)}(T)$ up to order $r-1$.  Note that we can certainly do so for $r=2$ since $\cF^1$ and $\cG^1$ agree in cohomology.  For the induction step, we want to be able to construct $T^r$.  For an arbitrary $T^r$, the order $r$ terms of the difference $D-\mu^1_{fun(\cA,\cB)}(T)$ induce a map $h$ on cohomology, and that $h$ is a cocycle in $CC^1(A,B)^{1-r}$.  Because $HH^1(A,B)^{1-r}$ vanishes for any $r\geq 2$, we have that $h=\partial v$ is exact.  Using a chain level representative of $v$, we can adjust $T^{r-1}$ so that  $D-\mu^1_{fun(\cA,\cB)}(T)$ vanishes at the $r$th order on the cohomology level.  Then we can choose a $r$th order homotopy $T^r$ so that $D=\mu^1_{fun(\cA,\cB)}(T)$ holds on the chain level.   Hence:

\begin{corollary}
Two $A_\infty$-functors from $\mathcal A$ to $\mathcal B$ which agree to first order in cohomology are homotopic, so $\cF$ and $\cG$ are homotopic.
\end{corollary}

\subsection{Limits of restriction functors}

\subsubsection{Limits of $A_\infty$-functors} \label{sec: limit}

Given an $A_\infty$-functor $\cE:  \cC \to \cD$, its limit $\cL=\cL(\cE)$ is defined to be the $A_\infty$-category with the same objects as $\cC$, and with morphism spaces
\[\sHom_{\cL}(X_0, X_1) = \sHom_{\cC}(X_0, X_1)[1] \oplus \sHom_{\cD}(\cE X_0,\cE X_1),\]
for any objects $X_0, X_1$.  The differential on this morphism space is given by 
\begin{equation}\label{eq: limit mu1}
\mu^1_{\cL}(c_1, d_1)=(\mu^1_{\cC}(c_1), \ \cE^1(c_1)+\mu^1_{\cD}(d_1)).
\end{equation}
Note that under this differential, $(c_1, d_1)$ is closed when $c_1$ is closed and $d_1$ is a null-homotopy for $\cE^1(c_1)$.
The composition $\mu^2_{\cL}: \sHom_{\cL}(X_1, X_2) \otimes \sHom_{\cL}(X_0,X_1)\to \sHom_{\cL}(\cE X_0, \cE X_2)$ is given by
\begin{equation}\label{eq: limit mu2}
\mu^2_{\cL}((c_2, d_2), (c_1,d_1))=\left(\mu_{\cC}^2(c_2, c_1), \ \cE^2(c_2,c_1)+\mu_{\cD}^2(\cE^1(c_2), d_1)\right).
\end{equation}
To make sense of this definition of the composition, we see that for closed $(c_1, d_1)$ and $(c_2, d_2)$, we have $\mu^2_{\cC}(c_2, c_1)$ is closed and $\cE^2(c_2,c_1)+\mu^2_\cD(\cE^1(c_1), d_1)$ is a null-homotopy for $\cE^1\left(\mu^2_\cC(c_2,c_1)\right)$ because 
$\cE^1\left(\mu^2_\cC(c_2,c_1)\right)=\mu^1_\cD\left(\cE^2(c_2, c_1)\right)+\mu^2_\cD\left(\cE^1(c_2), \cE^1(c_1)\right)=\mu^1_\cD\left(\cE^2(c_2,c_1)+\mu^2_\cD(\cE^1(c_1), d_1)\right),$
where the last equality is due to Leibniz's rule that $\mu^1_\cD\left(\mu^2_D(\cE^1(c_2), d_1)\right)=\mu^2_D\left(\cE^1(c_2), \mu^1_\cD(d_1)\right)$.    Note that we made an arbitrary choice of replacing $\cE^1(c_1)$ by $\mu^1_\cD(d_1)$, instead of replacing $\cE^1(c_2)$ by $\mu^1_\cD(d_2)$, so that  $\mu^2_\cD(\cE^1(c_2), \cE^1(c_1))=\mu^2_\cD(\cE^1(c_2), \mu^1_\cD(d_1))$.  

Similarly $\mu^k_{\cL}: \sHom_{\cL}(X_{k-1},X_k)\otimes \cdots \otimes \sHom_{\cL}(X_0, X_1)\to \sHom_{\cL}(\cE X_0,\cE X_k)$ is given by
\begin{equation}\label{eq: limit muk}
\mu^k_{\cL}((c_k, d_k),\ldots, (c_1, d_1))= 
\big(\mu_{\cC}^k(c_k,\ldots,c_1), \ \cE^k(c_k,\ldots,c_1) + \Delta^{k}(c_k,\ldots, c_2, d_1)\big),
\end{equation}
where 
\[\Delta^{k}(c_k,\ldots, c_2, d_1)= \sum\limits_{1\leq j\leq k}\ \sum\limits_{s_1+\cdots+s_{j-1}=k-1}\mu^j_{\cD}\left(\cE^{s_{j-1}}(c_k,\ldots, c_{k-s_{j-1}+1}),\ldots, \cE^{s_{1}}(c_{s_1-1},\ldots,c_2), d_1\right).\]
This sum is over all ways to put all the inputs from $\sHom_{\cC}$ into copies of the functor, and then compose in $\cD$ the resulting elements of $\sHom_{\cD}$ together with the unique input from $\sHom_{\cD}$ in the final position.
and zero on all other pieces.  So, $\mu^k_{\cL}$ is identically zero unless all inputs except possibly the last one were in the $\sHom_{\cC}$ piece of $\sHom_{\cL}$.

The $A_\infty$-functor equations for $\cE$ then give the $A_\infty$-associativity equations for $\mu_{\cL}$. 
Indeed, 
\[
\begin{array}{ll}
&\hspace{-0.3in} {\sum\limits_{\substack{1\leq m\leq k \\ 0\leq n\leq k-m}}} \mu_{\cL}^{k-m+1}\left((c_k, d_k), \ldots, \mu_{\cL}^m\left((c_{n+m},d_{n+m}),\ldots, (c_{n+1},d_{n+1})\right), \ldots (c_1,d_1)\right)\\

= &{\sum\limits_{m,n}} \bigg( \mu_{\cC}^{k-m+1}\left(c_k, \ldots, c_{n+m+1}, \mu_{\cC}^m\left(c_{n+m},\ldots, c_{n+1}\right), c_n, \ldots c_1\right), \\
& \cE^{k-m+1}\left(c_k,\ldots,\mu_{\cC}^m\left(c_{n+m},\ldots, c_{n+1}\right),\ldots, c_1\right))+ 
   \Delta^{k-m+1}(c_k,\ldots, \cE^m(c_m,\ldots, c_1))+\\

&  
\Delta^{k-m+1}(c_k, \ldots, \mu_{\cC}^m(c_{n+m},\ldots, c_{n+1}), \ldots, c_2, d_1)+

\Delta^{k-m+1}(c_k,\ldots, \Delta^{m}(c_m,\ldots, c_2, d_1))\bigg)\\

= & (0,\  0+0).
\end{array}
\]
The first $0$ is given by the $A_\infty$-equations for $\mu_{\cC}$.  The second $0$ is given by the $A_\infty$-functor equations for $\cE$.  The third $0$ is given by a combination of the $A_\infty$-functor equations for $\cE$ and the $A_\infty$-equations for $\mu_{\cD}$.

\subsubsection{Matrix factorization category as a limit of the restriction functor} \label{sec: limitMF} For the \v{C}ech covering $\mathfrak U=\{U_{\alpha\beta\gamma}\}$ of the mirror manifold $Y$, where each $U_{\alpha\beta\gamma}$ is a $\mathbb C^3$ piece as in Section \ref{sec: cech}, we can assign each piece a distinct integer $\mathfrak o(\alpha\beta\gamma)\in\{1, 2, \ldots, |\mathfrak U|\}$ to make $\mathfrak U$ an ordered cover.   To do so, we can define a generic height function on the moment polytope $\Delta_Y$ so that no two vertices get the same value.  This gives an ordering of the vertices that corresponds to the $U_{\alpha\beta\gamma}$'s.

The morphism complex $\mathscr{S}=\sHom(T_0, T_1)$ is a sheaf defined by $U_{\alpha\beta\gamma}\mapsto \sHom(T_0(U_{\alpha\beta\gamma}), T_1(U_{\alpha\beta\gamma}))$ for any two objects $T_0, T_1$ in $MF(Y, W)$.  In the \v{C}ech model for $MF(Y, W)$, the morphism is given by global sections of the total complex of the double \v{C}ech complex 
\[\left(C^*(\mathfrak U, \mathscr{S}), d\right)=\left(C^0(\mathfrak U, \mathscr{S}), d_{\cA'}\right)\oplus \left(C^1(\mathfrak U, \mathscr{S}), d_{\cB}\right).\] 
The differential  
\[d=\left(
\begin{array}{ll} 
d_{\mathcal A'}  & 0 \\
\delta & d_{\mathcal B} 
\end{array}\right).\]
consists of the differential  $d_{\cA'}$ on the $\mathbb C^3$ pieces, the differential $d_{\cB'}$ on the overlaps, plus the \v{C}ech differential $\delta$ that is the difference between the restriction maps $\sigma^1: C^0(\mathfrak U, \mathscr{S})\to C^1(\mathfrak U, \mathscr{S})$ coming from two overlapping $\mathbb C^3$ pieces.  That is, 
\[\delta= \bigoplus_{\mathfrak o(\alpha\beta\gamma)<\mathfrak o(\alpha\beta\eta)} \sigma^1\oplus -\sigma^1:
\bigoplus_{\mathfrak o(\alpha\beta\gamma)<\mathfrak o(\alpha\beta\eta)}  \mathscr{S}(U_{\alpha\beta\gamma})\oplus  \mathscr{S}(U_{\alpha\beta\eta}) \longrightarrow \bigoplus_{\alpha\beta} \mathscr{S}(U_{\alpha\beta})\]
Because we are working with $\mathbb Z_2$ grading, we can ignore the minus sign when defining $\delta$ for convenience, so $\delta=\sigma^1$.   Note that this differential coincides with the differential $\mu^1_{\cL}$ for the limit category $\cL=\cL(\sigma)$ defined by Equation (\ref{eq: limit mu1}).

The product $\mu^2_{\cL}$ consists of the product $\mu^2_{\cA'}$ on $\cA'$ as well as the maps
\[\mathscr S(U_{\alpha\beta\gamma})\otimes \mathscr S(U_{\alpha\beta})\to \mathscr S (U_{\alpha\beta})\] 
given by first restricting the $\mathbb C^3$ piece to the overlap, denoted by $U_{\alpha\beta}$, and then taking the product in the overlap.  This is how the product is defined in \v{C}ech theory.

Hence, in the usual dg-model for matrix factorizations on the pieces of the \v{C}ech covering with $\sigma^1:\cA'\to \cB$ is the restriction to the overlaps between $\mathbb C^3$ pieces, the morphisms in the limit category $\cL(\sigma)$ are exactly the \v{C}ech model for $MF(Y, W)$ of the total space.

\subsubsection{Wrapped Fukaya category as a limit of the restriction functor}

In our model for the wrapped Fukaya category, $\rho:\cA \to \cB$ only has a linear term $\rho^1$, and $\rho^1$ is surjective on morphism spaces.   We have $\cA=\mathcal W(\coprod P_{\alpha\beta\gamma})$ and for any adjacent pairs of pants $P_{\alpha\beta\gamma}$ and $P_{\alpha\beta\eta}$, $\rho^1$ maps

\begin{equation}\label{pullback}
\xymatrix@R2pc@C3pc{
  & A_{\alpha\beta\gamma}:=\bigoplus\limits_{L_i,L_j} CW^*_{P_{\alpha\beta\gamma}}(L_i,L_j) \ar[d]_{ \rho^1=\rho_{\alpha\beta}^\gamma }\\
A_{\alpha\beta\eta}:=\bigoplus\limits_{L_i,L_j} CW^*_{P_{\alpha\beta\eta}}(L_i,L_j) \ar[r]^{\ \ \ \rho^1=\rho_{\alpha\beta}^\eta\ \ \ } & B_{\alpha\beta}:=\bigoplus\limits_{L_i,L_j} CW^*_{\mathcal C_{\alpha\beta}}(L_i,L_j).
}
\end{equation}
 For the above diagram, we can defined the pullback category $\cK$ with same objects as those in $\cA$ and whose morphism complexes are given by the kernel  
\[
\ker\left(\rho^1\oplus (-\rho^1): A_{\alpha\beta\gamma}\oplus A_{\alpha\beta\eta}\to B_{\alpha\beta}\right)=\left\{(a,a')\in A_{\alpha\beta\gamma}\oplus A_{\alpha\beta\eta}| \rho^1(a)=\rho^1(a')\right\}.
\]  
For the set of pairs of pants $\{P_{\alpha\beta\eta}\}$, we assign an ordering $\mathfrak o(\alpha\beta\gamma)$ as we did previously in Section \ref{sec: limitMF}.  Then we can build the category $\cK$ for the entire Riemann surface, i.e. over all pairs of pants as the above, by letting
\[\cK=\ker \left( \bigoplus \limits_{\mathfrak o(\alpha\beta\gamma)<\mathfrak o(\alpha\beta\eta)}  
\rho^1\oplus (-\rho^1): \bigoplus \limits_{\mathfrak o(\alpha\beta\gamma)<\mathfrak o(\alpha\beta\eta)} A_{\alpha\beta\gamma}\oplus A_{\alpha\beta\eta} \to \bigoplus \limits_{\alpha\beta} B_{\alpha\beta} \right). \] 
By Theorem \ref{thm: decomposition}, the wrapped Fukaya category $\mathcal W(H)$ is $\cK$.  

Restricting to one pair of adjacent pairs of pants, so $\cA=\mathcal W(P_{\alpha\beta\gamma}\coprod P_{\alpha\beta\eta})$, then for each pair of objects  $L_i, L_j$, $\ker(\rho^1\oplus (-\rho^1))$ is a subset of $\sHom_{\cA}(L_i, L_j)$, and we can think of $\cK$ as a subcategory of the limit category $\cL=\cL(\rho^1\oplus(-\rho^1))$ as defined in Section \ref{sec: limit}.    Indeed, $\mu^k_{\cL}$ maps morphisms in $\ker(\mathcal \rho^1\oplus (-\rho^1))$ just by multiplication using $\mu^k_{\mathcal A}$, and $\mu^k_{\mathcal A}$ maps $\ker(\rho^1\oplus(-\rho^1)) $ to $\ker(\rho^1\oplus (-\rho^1)) $ since 
\[\left(\rho^1\oplus -\rho^1\right)\left(\mu^k_{\cA}\left((a_k,a_k'),\ldots,(a_1, a_1')\right)\right) = \mu^k_{\cB}\left(\rho^1(a_k)-\rho^1(a_k'),\ \ldots,\ \rho^1(a_1)-\rho^1(a_1')\right).\]

We claim that the category $\cK$ is the same as the limit category $\cL$, up to quasi-isomorphism given by the inclusion of $\cK$ into $\cL$.    
 Indeed, 
\[\sHom_{\cL}(L_i, L_j) = \sHom_{\cA}(L_i, L_j) [1] \oplus \sHom_{\cB}\left((\rho^1\oplus -\rho^1) L_i,(\rho\oplus -\rho^1) L_j\right)\] is larger, but by the surjectivity of $\rho^1\oplus (-\rho^1)$, every element of $ \sHom_{\cB}\left((\rho^1\oplus -\rho^1) L_i,(\rho^1\oplus -\rho^1) L_j\right)$ is of the form $ \left( \rho^1\oplus -\rho^1\right)(x,x')$, and is then equal mod $\mu^1_{\cL}(x,x')$ to $-\mu^1_{\cA}(x,x')$. So the part of $\sHom_{\cA}(L_i, L_j)[1]$ that's not in $\ker\left(\rho^1\oplus -\rho^1\right)$ kills the $ \sHom_{\cB}\left((\rho^1\oplus -\rho^1) L_i,(\rho^1\oplus -\rho^1) L_j\right)$ part, and what remains in cohomology is exactly the same as the cohomology of $\ker\left(\rho^1\oplus -\rho^1\right)$, i.e. the inclusion is a quasi-isomorphism.   Again, since we are working with $\mathbb Z_2$ grading, we can ignore the minus signs, so $\mathcal W(H)$ is quasi-isomorphic to  $\cL(\rho)$.

\subsubsection{Quasi-equivalence of the limits}

We have two homotopic restriction functors $\mathcal \rho: \mathcal A\to \mathcal B$ and $\sigma: \mathcal A'\to \mathcal B$. In addition, as we discussed above,  $\mathcal W(H)$ is quasi-isomoporphic to $\cL(\rho)$ and $MF(Y,W)=\cL(\sigma)$.  Specializing $\cE$ and $\tcE$ in Lemma \ref{lemma: qe of limits} below to $\rho$ and $\sigma$ gives us the following theorem.
\begin{theorem} The wrapped Fukaya category $D \mathcal W(H)$ is quasi-equivalent to the category of matrix factorizations $D^b_{sing}(X, W)$.  
\end{theorem}

For the following lemma, we use the same notations as those in Section \ref{sec: limit}.  See Section 1 of \cite{Se08} for background discussions on $A_\infty$-categories, in particular, the discussion on homotopy in Section (1h), and Equations (1.2), (1.6), (1.7), and (1.8), which are the $A_\infty$-associativity equations, $A_\infty$-functor equations, formula for compositions of $A_\infty$-functors, and formula for the differential of pre-natural transformations, respectively.
\begin{lemma}\label{lemma: qe of limits} Suppose there are functors $\cE:\cC\to \cD$ and $\tcE: \tcC \to \cD$, and there is an quasi-isomorphism $\cQ: \cC\to \tcC$ such that  $\cE$ and $\tcE\circ \cQ$ are homotopic functors that act the same way on objects.  Then the limits $\cL(\cE)$ and $\cL(\tcE)$ are quasi-isomorphic.
\end{lemma}

\begin{proof} The functors $\cE, \ \tcE\circ Q$ are themselves objects of an $A_\infty$-category $ \nufun(\cC, \cD)$, and let $D=\tcE\circ \cQ -\cE\in \sHom_{\nufun(\cC,\cD)}^1(\cE,\tcE)$ be the natural transformation defined by $D^k=(\tcE\circ \cQ)^k-\cE^k$.   Then, $\cE$ and $\tcE\circ Q$ being homotopic means that $D=\mu^1_{\nufun(\cC, \cD)}(T)$ for some $T\in \sHom^0_{\nufun(\cC, \cD)}(\cE, \tcE\circ Q)$ with $T^0=0$.   That is
\begin{equation}\label{eqn: Dk}
\begin{array}{l}
(\tcE\circ \cQ)^k - \cE^k =D^k =\mu^1_{\nufun(\cC, \cD)}(T)^k(c_k,\ldots, c_1) \vspace{0.1in} \\

 \hspace{0.95in} ={\sum\limits_{1\leq i\leq j}} \ \ {\sum\limits_{s_1+\cdots +s_j=k}}
\mu^j_\cD  \Big((\tcE\circ Q)^{s_j}(c_k,\ldots, c_{d-s_j+1}),\ldots, (\tcE\circ Q)^{s_{i+1}}(\ldots, c_{s_1+\cdots+s_{i}+1}),\vspace{-0.05in} \\
  \hspace{2.6in} T^{s_i}(c_{s_1+\cdots+s_i},\ldots, c_{s_1+\cdots+s_{i-1}+1}),\\
 \hspace{2.7in}  \cE^{s_{i-1}}(c_{s_1+\cdots+s_{i-1}}, \ldots),\ldots, \cE^{s_1}(c_{s_1},\ldots,c_1) \Big)\\

\hspace{1.2in} -{\sum\limits_{m,n}} \ \ T^{k-m+1}  \Big(c_d,\ldots, c_{n+m+1}, \mu_\cC^m(c_{n+m},\ldots, c_{n+1}),c_n, \ldots, c_1\Big),

\end{array}
\end{equation}
where the terms with $s_i=0$ vanish because  $T^0=0$, so we can just assume $s_i\neq 0$. For the first two orders, we have 
\begin{equation} \label{eqn: D1}
(\tcE\circ \cQ)^1-\cE^1=D^1(c)=\mu^1_\cD(T^1(c))+T^1(\mu^1_\cC(c)),
\end{equation}
\begin{equation} \label{eqn: D2}
\begin{array}{ll}
(\tcE\circ \cQ)^2-\cE^2=D^2(c_2,c_1)=& 
\mu^1_\cD\left(T^2(c_2,c_1)\right) +T^2\left(c_2,\mu^1_\cC(c_1)\right)+T^2\left(\mu^1_\cC(c_2),c_1\right) \\
& \hspace{-0.1in} +\mu^2_\cD\left(T^1(c_2),\cE^1(c_1)\right)+\mu^2_D\left((\tcE\circ Q)^1(c_2), T^1(c_1)\right)+T^1\left(\mu^2_{\cC}(c_2,c_1)\right),
\end{array}
\end{equation}
and recall that  $(\tcE\circ \cQ)^2(c_2, c_1)=\tcE^2\left(\cQ^1(c_2), \cQ^1(c_1)\right) +\tcE^1\left( \cQ^2(c_2, c_1)\right)$.

We now construct an $A_\infty$-functor $\cP: \cL(\cE)\to \cL(\tcE)$, with the map on objects given by $\cQ$.    For the first order, let 
\[\cP^1: \sHom_\cC(X_0, X_1)[1]\oplus \sHom_\cD(\cE X_0, \cE X_1) \to  \sHom_{\tcC}(\cQ X_0, \cQ X_1)[1]\oplus \sHom_\cD((\tcE\circ \cQ) X_0, (\tcE\circ \cQ) X_1) \]
such that 
\begin{equation}\label{eqn: P1}
\cP^1(c,d)=(\cQ^1(c), d+T^1(c))
\end{equation} 
As defined, $\cP^1$ is satisfies the first order $A_\infty$-functor equation, i.e. it is a chain map.  Indeed,
\[
\begin{array}{ll}
\mu^1_{\cL(\tcE)}\left(\cP^1\left(c, d\right)\right)  =\mu^1_{\cL(\tcE)}\left(\cQ^1(c),\  d+T^1(c)\right) 
 &= \left(\mu^1_\cC(\cQ^1(c)), \ (\tcE\circ Q)^1(c)+\mu_\cD^1(T^1(c))+\mu_\cD^1(d) \right),\\
& = \left(\cQ^1(\mu^1_\cC(c)),\  T^1(\mu_\cD^1(c))+\cE^1(c)+\mu^1_\cD(d) \right)\\
\end{array}
\]
is equal to
\[
\cP^1\left(\mu^1_{\cL(\cE)}(c, d)\right)  = \cP^1\left(\mu^1_\cC(c), \ \cE^1(c)+\mu^1_\cD(d)\right)=
 \left(\cQ^1(\mu^1_\cC(c)),\  T^1(\mu_\cD^1(c))+\cE^1(c)+\mu_\cD^1(d) \right) 
  \]

\noindent In general, we defined for $k\geq 2$ that
\begin{equation}\label{eqn: Pk}
\cP^k((c_k, d_k),\ldots, (c_1,d_1))=\left(Q^k(c_k,\ldots, c_1),\ \ T^k(c_k, \ldots, c_1)+\Diamond^k(c_k,\ldots, c_2, d_1) \right),
\end{equation}
where 
\[
\begin{array}{ll}
\Diamond^k(c_k,\ldots, c_2, d_1) &  \\
={\sum\limits_{\substack{1\leq i \leq j-1\\  2\leq j\leq k}}} \ \ {\sum\limits_{\substack{s_1+\cdots +s_{j-1}=k-1 \\ s_i\neq 0}}}
\mu_\cD^j \Big(&  (\tcE\circ \cQ)^{s_{j-1}}(c_k,\ldots, c_{k-s_{j-1}+1}),\ldots, 
(\tcE\circ \cQ)^{s_i+1}(\ldots, a_{s_1+\cdots+s_i+1}),\vspace{-0.18in}\\

 & T^{s_i}(c_{s_1+\cdots+s_i},\ldots, , c_{s_1+\cdots+s_{i-1}+1}),\\

 & \cE^{s_i-1}(c_{s_1+\cdots+s_{i-1}}, \ldots),\ldots,  \cE^{s_1}(c_{s_1-1},\ldots, c_2)    ,d_1\Big).
\end{array}\]
We will show below that $\cP$ is an $A_\infty$-functor.  Before we do that, using the fact that $\cQ$ is a quasi-isomorphism, we see that $\cP$ is a quasi-isomorphism because the $\cP^1$ has a inverse  
\[(\cP^1)^{-1}(\tilde c,\tilde d)= \left ((\cQ^1)^{-1}(\tilde c),\ \  \tilde d-T^1\left((\cQ^1)^{-1}(\tilde c)\right)\right).\]

Below we show that $\cP^2\left((c_2,d_2),(c_1,d_1)\right)=\left(\cQ^2(c_2,c_1),T^2(c_2,c_1)+\mu_\cD^2(T^1(c_2),d_1 )\right)$ satisfies the 2nd order $A_\infty$-functor equation, and it illustrates features of the messy calculations involved in showing that all $\cP^k$'s satisfy the $A_\infty$-functor equations.
\[\begin{array}{l}
\mu^1_{\cL(\tcE)}\left(\cP^2\left((c_2,d_2),(c_1,d_1)\right)\right)  +\cP^2\left(\mu^1_{\cL(\cE)}(c_2,d_2),(c_1,d_1)\right)+ \cP^2\left((c_2,d_2),\mu^1_{\cL(\cE)}(c_1,d_1)\right)\\ 

\hspace{2in}  + \cP^1\left(\mu^2_{\cL(\cE)} \left( (c_2,d_2),(c_1,d_1)  \right) \right)+\mu^2_{\cL(\tcE)}\left( \cP^1\left(c_2, d_2 \right), \cP^1\left(c_1, d_1\right)  \right)
 \\

=\bigg( \mu^1_{\cL(\tcE)}\Big(  \cQ^2(c_2,c_1),T^2(c_2,c_1)+\mu_\cD^2(T^1(c_2),d_1 )\Big) \bigg)
 + \bigg( \cP^2\Big(\left(\mu^1_\cC(c_2), \cE^1(c_2)+\mu^1_\cD(d_2))\right) ,(c_1,d_1)   \Big) \bigg)\\
 
\ \ +\bigg( \cP^2\Big((c_2,d_2), \left(\mu^1_\cC(c_1), \cE^1(c_1)+\mu^1_\cD(d_1))\right)    \Big) \bigg)
 + \bigg( \cP^1\Big( \mu_\cC^2(c_2,c_1), \cE^2(c_2,c_1)+ \mu^2_{\cD}(\cE^1(c_2),d_1)\Big)   \bigg)\\

\ \ +\bigg(\mu^2_{\cL(\tcE)}\Big( \left(\cQ^1(c_2), d_2 + T^1(c_2)\right),  \left(\cQ^1(c_1), d_1 + T^1(c_1)\right) \Big) \bigg)
\end{array}
\]
\[
\begin{array}{l}
= \Bigg(\mu^1_\cC\left(\cQ^2(c_2,c_1)\right)+\cQ^2\left(\mu^1_\cC(c_2), c_1\right)+\cQ^2\left(c_2,\mu^1_\cC(c_1)\right)+ \cQ^1\left(\mu^2_\cC(c_2,c_1)\right)+\mu^2_\cC\left(\cQ^1(c_2),\cQ^1(c_1)\right),\\

\hspace{1in} \Big(\boxed{\tcE^1\left( \cQ^2(c_2,c_1)\right)}+ \boxed{\mu^1_\cD\left(T^2(c_2,c_1)\right)}+ \mu^1_\cD\left(\mu^2_\cD\left(T^1(c_2),d_1\right)\right)\Big)\\

\hspace{1in} +\Big(\boxed{T^2\left(\mu^1_\cC(c_2),c_1\right)}+\mu^2_\cD\left(T^1\left(\mu^1_\cC(c_2)\right),d_1\right)\Big)\\
\hspace{1in} +\Big(\boxed{T^2\left(c_2, \mu^1_\cC(c_1)\right)}+\boxed{\mu^2_\cD\left(T^1(c_2),\cE^1(c_1)\right)}+\mu^2_\cD\left(T^1(c_2),\mu^1_\cD(d_1)\right)\Big)\\

\hspace{1in} + \Big(\boxed{ T^1\left(\mu^2_\cC(c_2,c_1)\right)}+\boxed{\cE^2(c_2,c_1)}+\mu^2_\cD\left(\cE^1(c_2),d_1\right)  \Big)\\

\hspace{1in}+ \left( \boxed{\tcE^2\left(\cQ^1(c_2),\cQ^1(c_1)\right)}+\mu^2_\cD\left((\tcE\circ\cQ)^1(c_2),d_1\right)+
\boxed{\mu^2_\cD\left((\tcE\circ\cQ)^1(c_2), T^1(c_1)\right)}  \right)

\Bigg)\\

=\big(0, 0\big).

\end{array}
\]
In the above calculation, the first component is $0$ because $\cQ$ is an $A_\infty$-functor.   For the second component, all terms that are boxed add up to $0$ because of Equation (\ref{eqn: D2}).  The leftover terms also add up to $0$ due to the Leibniz's rule for $\mu_\cD$ and Equation (\ref{eqn: D1}).  

Just like the above calculation, in general all the terms in the $k$-th order $A_\infty$-functor equation for $\cP$ add up to zero.  Indeed, the first component is $0$ because it is just a sum of the terms in the $k$-th order $A_\infty$-functor equation for $\cQ$.  The the second component, a subset of the terms  (like the boxed terms above) gives $(\tcE\circ \cQ)^k+\cE^k+D^k=0$.  The leftover terms in the second component add up to $\mu^2_\cD\left((\tcE\circ \cQ)^{k-1}+\cE^{k-1}+D^{k-1}, \ d_1\right)=\mu^2_\cD(0,d_1)=0$ after manipulations using the $k$-th order $A_\infty$-associativity equations for $\mu_D$, the lower order $A_\infty$-functor equations for $\tcE\circ Q$, and the lower order equations in (4.7).  
\end{proof}

\appendix
\section{Proof of split-generation for the wrapped Fukaya category}\label{app: generation}

\begin{lemma}
The wrapped Fukaya category $\mathcal W(H)$ is split-generated by objects $L_{\alpha}(k)$, $\alpha\in A$, $k\in \{0,1\}$.
\end{lemma}

\begin{proof}
For each bounded edge $e_{\alpha\beta}\cong (0,4)\times S^1$, denote by $S_{\alpha\beta}$ the circle $\{\tau=2\}\times S^1$.  The Dehn twist of $L_\alpha(0)$ successively along all of the $S_{\alpha\beta}$'s for the bounded edges along the boundary of a component $C_\alpha$ is isomorphic to $L_{\alpha}(1)$.   So by Seidel (cf. Section 17 of \cite{Se08}), there is an exact triangle 
\[
\xymatrix{
\bigoplus S_{\alpha\beta}\ar[r]  &    L_{\alpha}(0)\ar[d]\\
 & L_{\alpha}(1)\ar[lu]^{[1]}    \  .}
\]
Hence, the $S_{\alpha\beta}$'s are split-generated by $L_{\alpha}(0)$ and $L_{\alpha}(1)$.

Let $P_{\alpha\beta\gamma}$ be any pair-of-pants in $H$, and say it is at the joint of components $C_{\alpha}, C_\gamma, C_\beta$, in that counterclockwise order  in Fig. 10.  We  consider two polygons that cover $P_{\alpha\beta\gamma}$ (the ``front'' and ``back'' of the pants) with boundaries in Lagrangians $L_{\alpha}(0)$,  $S_{\alpha\gamma}$ (only if the edge $e_{\alpha\gamma}$ is bounded), $L_{\gamma}(0)$, $S_{\gamma\beta}$ (only if $e_{\gamma\beta}$ is bounded), $L_\beta(0)$, and $S_{\beta\alpha}$ (only if $e_{\beta\alpha}$ is bounded).

On a bounded edge $e_{\alpha\gamma}$, the front polygon has two corners---one corner at a generator labeled by $p_{\alpha}^{\gamma}\in CW(L_{\alpha}(0), S_{\alpha\gamma})$ and another at a generator of the same degree labeled by $q^{\gamma}_{\alpha}\in CW(S_{\alpha\gamma}, L_{\gamma}(0))$.  On an unbounded edge $e_{\beta\alpha}$,  the front polygon has one corner at the first non-trivial trajectory of $H_n$, from $L_{\beta}(0)$ to $L_{\alpha}(0)$, in the infinite cylindrical end (i.e. the trajectory that goes ``halfway around the cylinder''); it is a degree 1 generator labeled by $x^0_{\beta; \alpha}$ as in Section \ref{section: PPD}.  So, this is a hexagon for a pair-of-pants that only has bounded legs, a triangle for a pair-of-pants that only has unbounded legs, and a quadrilateral or pentagon for the intermediate cases; see Fig. \ref{fig: generation}.  The polygon at the back of the pants also has two corners on a bounded edge $e_{\alpha\gamma}$ at the generators $(p_{\alpha}^{\gamma})^*\in CW(S_{\alpha\gamma}, L_{\alpha}(0))$ and ($q^{\gamma}_{\alpha})^*\in CW(L_\gamma(0),S_{\alpha\gamma})$.  Note that $\deg((p_{\alpha}^{\gamma})^*)=\deg(p_{\alpha}^{\gamma})+1$ and $\deg((q_{\alpha}^{\gamma})^*)=\deg(q_{\alpha}^{\gamma})+1$.  Similarly, the back polygon has one corner on an unbounded edge $e_{\alpha\beta}$ at the degree 1 generator $x^0_{\alpha;\beta}$.  
\begin{figure}[h]
\centering
	\scalebox{0.65}{\includegraphics{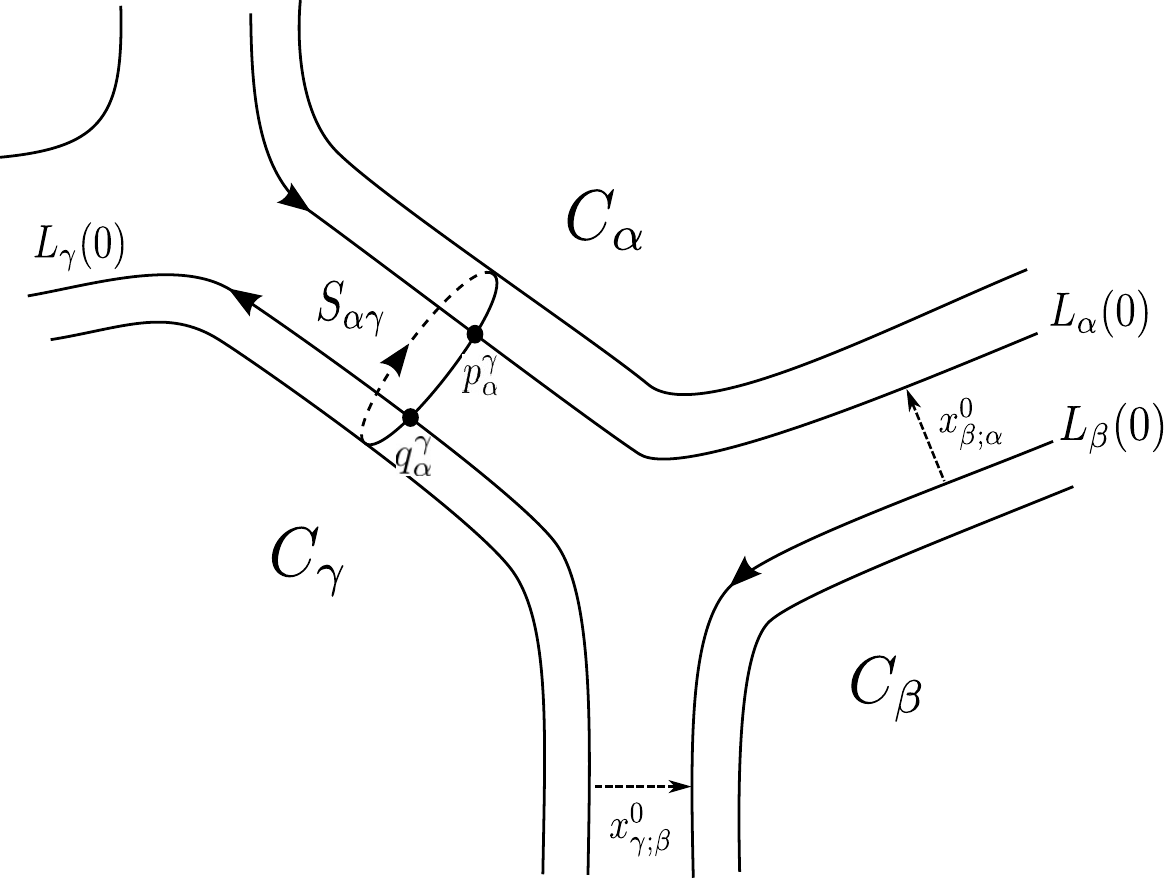}}
	\caption{Lagrangians $L_\alpha(0),\ L_\beta(0), \ L_\gamma(0)$, and $S_{\alpha\gamma}$. The pair-of-pants $P_{\alpha\beta\gamma}$ in this figure is covered by quadrilaterals, one in the front and one in the back.} 
	\label{fig: generation}
\end{figure}

A Hochschild chain complex 
\[CC_*(\mathcal W, \mathcal W)=\bigoplus CW^*(L_{d-1}, L_0)\otimes CW^*(L_{d-2}, L_{d-1}) \otimes \cdots \otimes CW^*(L_0, L_1)\] is generated by cyclically composable words $\underline{a_d}\otimes a_{d-1}\cdots \otimes a_1$ where each $a_i\in CW^*(L_{i-1}, L_{i})$.  We underline the distinguished generator  $\underline{a_d}$, and we think of $CW^*(L_{d-1}, L_0)$ as an element of a $\mathcal W$-bimodule.   The degree of $\underline{a_d}\otimes a_{d-1}\otimes\cdots\otimes a_1$ is given by $\deg(a_d)+\sum_{i=1}^{d-1}(\deg(a_i)+1)$, i.e. $\underline {a_d}$ contributes its usual degree whereas each of the other letters contributes its reduced degree of $\deg(a_i)+1$.  

 For each of the above polygons, we can write the product of the generators at the corners to get a Hochschild chain.  In the case we have a hexagon, a Hochschild chain is 
\[
\underline{q_{\beta}^{\alpha}}\otimes p_{\beta}^{\alpha} \otimes q_{\gamma}^{\beta}\otimes p_{\gamma}^{\beta}\otimes q_{\alpha}^{\gamma}\otimes p_{\alpha}^{\gamma}.
\] 
There are 6 choices of such Hochschild chains for each hexagon made by cyclically permuting the generators, and we just need to pick one.  Later on in the proof, we will give more criteria regarding which generator we need to pick to be the distinguished one at the front.  In the case we have a pentagon with the edge $e_{\alpha\beta}$ being unbounded, the Hochschild chain is 
\[
\underline{x^0_{\beta;\alpha}}\otimes q_{\gamma}^{\beta}\otimes p_{\gamma}^{\beta}\otimes q_{\alpha}^{\gamma}\otimes p_{\alpha}^{\gamma},
\]
and similarly for other types of polygons.   

Let $\Phi$ be the Hochschild chain obtained by summing all tensor products, one term each for the front and back of each pair-of-pants in $H$.   We wish to show that the Hochschild differential of $\Phi$ is zero, i.e. it is a Hochschild cycle.   Then, for any generic point in the bounded part of $H$, there is a unique polygon that covers it just once.  So, the open-closed map defined by Abouzaid in his generation criterion takes this Hochschild cycle to the identity element of $SH^0(H)$, which is the generator of the symplectic cohomology given by the sum of all minima of the chosen Hamiltonian in the interior of $H$.  So these objects split-generate by the generation criterion \cite{Ab10}.

We devote the rest of this section to show that $\Phi$ is a Hochschild cycle.  The Hochschild differential of each term $\underline{a_d}\otimes a_{d-1}\otimes\cdots\otimes a_1$ is given by 
\[
\begin{array} {ll}
b(\underline{a_d}\otimes\cdots \otimes a_1)=& \sum_{1\leq i+j < d} \  \underline{a_d}\otimes \cdots\otimes a_{i+j+1}\otimes \mu^{j}(a_{i+j},\ldots, a_{i+1})\otimes a_{i}\otimes\cdots\otimes a_1\\
 &+ \sum_{0\leq i+j< d}\  \mu^{d-j}(a_{i},\ldots, a_1, \underline{a_d},a_{d-1},\ldots,a_{i+j+1})\otimes a_{i+j}\otimes\cdots \otimes a_{i+1}.
\end{array}
\] 
To find the Hochschild differential of $\Phi$, we need to compute products involving sub-expressions among the set of corners described above.  In this section, when we compute the wrapped Floer complexes, we can simplify the computation by using the Hamiltonian in Section \ref{sec: hamiltonian} but without the extra wrappings on the bounded edge, i.e. we let the Hamiltonians $H_n$ be $H_0$, for all $n$, on bounded edges.
 
   For a front hexagon,
\[\begin{array} {ll}
&b(\underline{q_{\beta}^{\alpha}}\otimes p_{\beta}^{\alpha} \otimes q_{\gamma}^{\beta}\otimes p_{\gamma}^{\beta}\otimes q_{\alpha}^{\gamma}\otimes p_{\alpha}^{\gamma})\\

= & \underline{q_{\beta}^{\alpha}}\otimes \mu^5(p_{\beta}^{\alpha}, q_{\gamma}^{\beta}, p_{\gamma}^{\beta},q_{\alpha}^{\gamma},p_{\alpha}^{\gamma} ) \\
  
  & + \mu^5(\underline{q_{\beta}^{\alpha}},p_{\beta}^{\alpha}, q_{\gamma}^{\beta}, p_{\gamma}^{\beta},q_{\alpha}^{\gamma})\otimes p_{\alpha}^{\gamma} 
 + \mu^5(p_{\alpha}^{\gamma}, \underline{q_{\beta}^{\alpha}},p_{\beta}^{\alpha}, q_{\gamma}^{\beta}, p_{\gamma}^{\beta})\otimes  q_{\alpha}^{\gamma} + \cdots (\text{5 of these})\\
 
  &+ \mu^6(\underline{q_{\beta}^{\alpha}}, p_{\beta}^{\alpha}, q_{\gamma}^{\beta}, p_{\gamma}^{\beta},q_{\alpha}^{\gamma},p_{\alpha}^{\gamma} )
         + \mu^6(p_{\alpha}^{\gamma}, \underline{q_{\beta}^{\alpha}}, p_{\beta}^{\alpha}, q_{\gamma}^{\beta}, p_{\gamma}^{\beta},q_{\alpha}^{\gamma}) + \cdots (\text{6 of these}) \\
        
 = & \underline{q_{\beta}^{\alpha}}\otimes (q_{\beta}^{\alpha})^*
   + \left( (\underline{p_{\alpha}^{\gamma}})^*\otimes p_{\alpha}^\gamma + (\underline{q^\gamma_\alpha})^*\otimes q_{\alpha}^{\gamma}
 +(\underline{p_\gamma^\beta})^*\otimes p_\gamma^\beta + (\underline{q_{\gamma}^{\beta}})^*\otimes q_{\gamma}^{\beta}
 +(\underline{p_\beta^\alpha})^*\otimes p_\beta^\alpha\right)\\
 & + e_{L_\alpha}^{\alpha\beta\gamma} + e^{front}_{S_{\alpha\gamma}}+ e_{L_\gamma}^{\alpha\beta\gamma}+ e^{front}_{S_{\gamma\beta}}+ e_{L_\beta}^{\alpha\beta\gamma}+ e^{front}_{S_{\beta\alpha}}.    \end{array}\]
In the above calculation, there are no terms involving $\mu^1, \mu^2, \mu^3,$ and $\mu^4$ because there is simply no holomorphic polygon with less than 6 sides with vertices coming from the above generators.   For each corner (6 total) of a hexagon, we have
\[\mu^5(p_{\beta}^{\alpha}, q_{\gamma}^{\beta}, p_{\gamma}^{\beta},q_{\alpha}^{\gamma},p_{\alpha}^{\gamma} ) =(q_{\beta}^{\alpha})^*\in CF(L_\alpha(0), S_{\beta\alpha}), \]
and so on.   The computation for  $\mu^6$ needs some explanation.   For each side (6 total) of a hexagon, we have 
\[\mu^6(q_{\beta}^{\alpha}, p_{\beta}^{\alpha}, q_{\gamma}^{\beta}, p_{\gamma}^{\beta},q_{\alpha}^{\gamma},p_{\alpha}^{\gamma} )=e_{L_\alpha}^{\alpha\beta\gamma}=dx^0_{\alpha\gamma} \in CF(L_{\alpha}(0), L_{\alpha}(0))\] or
 \[\mu^6(p_{\alpha}^{\gamma}, q_{\beta}^{\alpha}, p_{\beta}^{\alpha}, q_{\gamma}^{\beta}, p_{\gamma}^{\beta},q_{\alpha}^{\gamma})=e_{S_{\alpha\gamma}}^{front}\in CF^0(S_{\alpha\gamma}, S_{\alpha\gamma})\]
 and so on for the cyclic permutations. Generators in $CF^0(L_\alpha(0), L_\alpha(0))$ correspond to the maxima of the existing Hamiltonian.  To define $CF^0(S_{\alpha\gamma},S_{\alpha\gamma})$, we need to pick an additional auxiliary Hamiltonian to perturb $S_{\alpha\gamma}$.  We can pick this Hamiltonian to have exactly one minimum in each of the front and back of the pair of pants, though it doesn't really matter as we can just sum up the minima of this Hamiltonian in the region.  
   
 Similarly, for the corresponding hexagon on the back of the same pair-of-pants,
\[\begin{array} {ll}
&b\left((\underline{q_{\beta}^{\alpha}})^*\otimes (p_{\alpha}^{\gamma})^* \otimes (q^{\gamma}_{\alpha})^*\otimes (p_{\gamma}^{\beta})^*\otimes (q_{\gamma}^{\beta})^*\otimes (p_{\beta}^{\alpha})^*\right)\\

 = & (\underline{q_{\beta}^{\alpha}})^*\otimes q_{\beta}^{\alpha}+\left(
 \underline{p_{\beta}^{\alpha}}\otimes (p_{\beta}^\alpha)^* +\underline{q_\gamma^\beta}\otimes (q_{\gamma}^{\beta})^* + \underline{p_\gamma^\beta}\otimes (p_{\gamma}^{\beta})^*+ \underline{q^\gamma_\alpha}\otimes (q^{\gamma}_\alpha)^*+ \underline{p^\gamma_\alpha}\otimes (p^{\gamma}_\alpha)^*\right)\\
 &  +e^{back}_{S_{\beta\alpha}}+ e_{L_\beta}^{\alpha\beta\gamma} + e^{back}_{S_{\gamma\beta}}  + e_{L_\gamma}^{\alpha\beta\gamma} +e^{back}_{S_{\alpha\gamma}} + e_{L_\alpha}^{\alpha\beta\gamma} \end{array}\]
    
 \noindent For the adjacent hexagon in the front, if we choose the same corner as the above to be the distinguished generator, then
 \[\begin{array}{ll}
 & b\left(\underline{q^{\beta}_{\alpha}}\otimes p_{\alpha}^{\beta} \otimes q_{\eta}^{\alpha}\otimes p_{\eta}^{\alpha}\otimes q_{\beta}^{\eta}\otimes p_{\beta}^{\eta}\right) 
 = b\left((\underline{q_{\beta}^{\alpha}})^*\otimes (p_{\beta}^{\alpha} )^*\otimes q_{\eta}^{\alpha}\otimes p_{\eta}^{\alpha}\otimes q_{\beta}^{\eta}\otimes p_{\beta}^{\eta}\right)\\
 = & (\underline{q_{\beta}^{\alpha}})^*\otimes q_{\beta}^{\alpha} +\left( \underline{p_{\beta}^{\alpha}}\otimes (p_{\beta}^\alpha)^* +\cdots\right)+ e_{L_\alpha}^{\alpha\beta\eta}+ e_{\beta\alpha}^{front} + e_{L_\beta}^{\alpha\beta\eta}+\cdots.
 \end{array} \]  If we change the choice of the distinguished generator to a different corner for the adjacent hexagon in the front, we get, for example
 \[\begin{array}{ll}
 & b\left(\underline{q_{\eta}^{\alpha}}\otimes p_{\eta}^{\alpha}\otimes q_{\beta}^{\eta}\otimes p_{\beta}^{\eta}\otimes q^{\beta}_{\alpha}\otimes p_{\alpha}^{\beta} \right) 
 = b\left(\underline{q_{\eta}^{\alpha}}\otimes p_{\eta}^{\alpha}\otimes q_{\beta}^{\eta}\otimes p_{\beta}^{\eta}\otimes (q_{\beta}^{\alpha})^*\otimes (p_{\beta}^{\alpha} )^*\right)\\
 = & \underline{q_{\eta}^{\alpha}}\otimes (q_{\eta}^{\alpha})^* +\left(\underline{p_{\beta}^{\alpha}}\otimes (p_{\beta}^\alpha)^*+ \underline{q_{\beta}^{\alpha}}\otimes (q_{\beta}^\alpha)^*+\cdots\right)+ e_{L_\alpha}^{\alpha\beta\eta}+ e_{\beta\alpha}^{front} + e_{L_\beta}^{\alpha\beta\eta}+\cdots.
 \end{array} \]  
 
We need to make sure to let the same corner be the distinguished generator for both the front and the back of each pair-of-pants.  For example, for the Hochschild chain $\underline{q_{\beta}^{\alpha}}\otimes p_{\beta}^{\alpha} \otimes q_{\gamma}^{\beta}\otimes p_{\gamma}^{\beta}\otimes q_{\alpha}^{\gamma}\otimes p_{\alpha}^{\gamma}$, the corner of the hexagon corresponding to $\underline{q_\beta^\alpha}$ is the distinguished generator.  So for the hexagon in the back, we need to let the same corner $(\underline{q_\beta^\alpha})^*$ be the distinguished generator as we did in the calculation above.  After making this choice, all the terms will cancel in the sum and thus $b\Phi=0$.  The cancellations happen in the following ways (note that the Floer complexes are $\mathbb Z_2$ graded):
\begin{itemize}
\item The $e_{L_\alpha}^{\alpha\beta\gamma}$ term from the front hexagon cancels  the same term from the back hexagon, and so on for other terms of this type.   

\item For each $S_{\alpha\beta}$, the term $e_{S_{\alpha\beta}}^{front}$ from the front hexagon cancels the same term from the adjacent hexagon in the front.  

\item    For the front hexagon, consider the term $\underline{q_\beta^\alpha}\otimes (q_\beta^\alpha)^*$, which involves the distinguished generator. If the same corner is further chosen as the distinguished one for the polygons coming from the adjacent pair-of-pants, then $\underline{q_\beta^\alpha}\otimes (q_\beta^\alpha)^*$ cancels  with the same term coming from the adjacent polygon in the front; otherwise, $\underline{q_\beta^\alpha}\otimes (q_\beta^\alpha)^*$ cancels the same term coming from the adjacent polygon in the back.

\item For the front hexagon,  the term $(\underline{p_\beta^\alpha})^* \otimes p_\beta^\alpha$, which does not involve the distinguished generator, cancels the same term coming from adjacent polygon in the back.    The calculation above shows the same situation when the term $\underline{p_\beta}^\alpha\otimes(p_\beta^\alpha)^*$ from the back hexagon cancels the same term coming from the adjacent polygon in the front, and so on for the other terms of this type.
\end{itemize}

For those pair-of-pants that has at least one unbounded edge, such as when we have a pentagon, 
\[
\begin{array}{ll}
& b(\underline{x^0_{\beta;\alpha}}\otimes q_{\gamma}^{\beta}\otimes p_{\gamma}^{\beta}\otimes q_{\alpha}^{\gamma}\otimes p_{\alpha}^{\gamma})\\
= &\cancel{ \underline{x^0_{\beta;\alpha}}\otimes \mu^4(q_{\gamma}^{\beta}\otimes p_{\gamma}^{\beta}\otimes q_{\alpha}^{\gamma}\otimes p_{\alpha}^{\gamma})}  +\left(\mu^4(\underline{x^0_{\beta;\alpha}}\otimes q_{\gamma}^{\beta}\otimes p_{\gamma}^{\beta}\otimes q_{\alpha}^{\gamma}) \otimes p_{\alpha}^\gamma+\cdots (\text{4 of these})\right)+ \\
& \left( \mu^5(\underline{x^0_{\beta;\alpha}}\otimes q_{\gamma}^{\beta}\otimes p_{\gamma}^{\beta}\otimes q_{\alpha}^{\gamma}\otimes p_{\alpha}^{\gamma})+\cdots (\text{5 of these})\right)\\
= & \left( (\underline{p_{\alpha}^\gamma})^*\otimes p_\alpha^\gamma+\cdots\right) +x^0_{\beta;\alpha} + e_{L_\alpha}^{\alpha\beta\gamma} + e_{S_{\alpha\gamma}}^{front}+\cdots, \end{array}
\]
where the term $\mu^4(q_{\gamma}^{\beta}\otimes p_{\gamma}^{\beta}\otimes q_{\alpha}^{\gamma}\otimes p_{\alpha}^{\gamma})$ is zero because the Reeb cord $x^0_{\beta,\alpha}$ that lies in this pentagon is in the wrong direction to be an output.    So, for each pair-of-pants that has at least one unbounded leg, it is most convenient to assign the distinguished generators for that pair-of-pants to be the Reeb cords in the cylindrical end.
\end{proof}

\section{A global cohomology level computation} \label{app: pedestrian computation}
This appendix will not be part of the version to be submitted to a journal.  We compute the wrapped Fukaya category $\mathcal W(H)$ and the category of singularity $D^b_{sg}(W^{-1}(0))$ at the level of cohomology, using methods very similar to those in \cite{AAEKO13}.  Aside from using a few earlier notations, this appendix is self contained, independent from the sheaf theoretic computation in the rest of the paper.  We show it because it is a straightforward demonstration of HMS, though we do not know how to extend this to compute the higher $A_\infty$-structures.  We will list the generators of the morphism complexes for both categories, but we will be very brief in discussing the product structures on the morphism complexes because this computation is not the main point of this paper.  
  
\subsection {The wrapped Fukaya category.}  We use Abouzaid's model of the wrapped Fukaya category \cite{Ab10}, with only a Hamiltonian perturbation $H:M\to \mathbb R$ that is quadratic in the cylindrical ends.  Then the wrapped Floer complex $CW^*(L_\alpha(k),L_\beta(l))$ is generated by the Reeb chords that are the time-1 trajectories of the Hamiltonian flow from $L_\alpha(k)$ to $L_\beta(l)$.   Equivalently, up to a change of almost-complex structure, $CW^*(L_\alpha(k),L_\beta(l))=\langle\phi^1(L_\alpha(k))\cap L_\beta(l)\rangle$.  The product $\mu^2: CW^*(L_\beta(j),L_\gamma(l))\otimes CW^*(L_\alpha(j), L_\beta(k)) \to CW^*(L_\alpha(j), L_\gamma(l))$ is equivalent to the usual Floer product
$CF^*(\phi^1(L_\beta(j)), L_\gamma(l))\otimes CF^*(\phi^2(L_\alpha(j)), \phi^1(L_\beta(k))\to CF^*(\phi^2(L_\alpha(j)), L_\gamma(l)),$
which counts holomorphic triangles with boundaries on $\phi^2(L_\alpha(j))$, $\phi^1(L_\alpha(k))$, and $L_\alpha(l)$.  The inputs are $\phi^1(p_1)\in \phi^2(L_\alpha(j))\cap \phi^1(L_\beta(k))$ and $p_2\in \phi^1(L_\beta(k))\cap L_\alpha(l)$, and the output is $\tilde q\in \phi^2(L_\alpha(j))\cap L_\alpha(l)$, which corresponds to $q\in \phi^1(L_\alpha(j))\cap L_\alpha(l)$ by the rescaling method explained in \cite{Ab10}.  

\subsubsection{Lagrangians from the same component.}\ \\
 First, we list the generators of $CW^*(L_\alpha(k), L_\alpha(l))$; see Fig. \ref{C1} and \ref{C2}.  There are two cases. 

 \emph{Case 1: on a bounded edge $e_{\alpha\beta}$}.  For $k<l$, $L_\alpha(k)$ intersects $L_\alpha(l)$ at $n_{\alpha\beta}(l-k)+1$ points on $e_{\alpha\beta}$, all of which have degree 0.  We label each intersection point sequentially by $x_{\alpha\beta}^{n_{\alpha\beta}(l-k)-j}y_{\alpha\beta}^j$, for $j=0,\ldots, n_{\alpha\beta}(l-k)$.  If $e_{\alpha\beta}$ is adjacent to another bounded edge $e_{\alpha\eta}$, then there is a generator that is an interior intersection point at the joint of these two edges labeled by $x_{\alpha\beta}^{n_{\alpha\beta}(l-k)}$ on one side and $y_{\alpha\eta}^{n_{\alpha\eta}(l-k)}$ on the other side; we identify them $x_{\alpha\beta}^{n_{\alpha\beta}(l-k)}=y_{\alpha\eta}^{n_{\alpha\eta}(l-k)}$.   For $k> l$, $L_\alpha(k)$ intersects $L_\alpha(l)$ at $n_{\alpha\beta}(k-l)-1$ points of degree 1, and we label them sequentially by $\left(x_{\alpha\beta}^{n_{\alpha\beta}(k-l)-2-j}y_{\alpha\beta}^j\right)^*$, for $j=0,\ldots, n_{\alpha\beta}(k-l)-2$.  When $e_{\alpha\beta}$ is adjacent to another bounded edge $e_{\alpha\eta}$, we get an extra degree 1 generator that is an interior intersection point at the joint of these two edges.  Again, we identify the two labels coming from both sides.     

 \emph{Case 2: on an unbounded edge $e_{\alpha\gamma}$}.  Let $e_{\alpha\eta}$ be the adjacent edge. If $k<l$, then  $\phi^1(L_\alpha(k))$ and $L_\alpha(l)$ have one interior intersection point at the joint of $e_{\alpha\gamma}$ and $e_{\alpha\eta}$, which we label by $x_{\alpha\gamma}^0$, and we label the infinitely many purturbed intersections points on $e_{\alpha\gamma}$ successively by $x_{\alpha\gamma}^j$ for $j=1,2,\ldots$.  Like before, we need to set $x_{\alpha\gamma}^0$ to equal the label for the same generator coming from the edge $e_{\alpha\eta}$, i.e.   $x_{\alpha\gamma}^0=x_{\alpha\eta}^0$ if $e_{\alpha\eta}$ is unbounded, and  $x_{\alpha\gamma}^0=y_{\alpha\eta}^{n_{\alpha\eta}(l-k)}$ (or $=x_{\alpha\eta}^{n_{\alpha\beta}(l-k)}$ depending on which side $e_{\alpha\gamma}$ is attached to $e_{\alpha\beta}$) if $e_{\alpha\eta}$ is bounded. If $k>l$, then there is no interior intersection point, only infinitely many perturbed intersection points labeled by $x_{\alpha\gamma}^j$ for $j=1,2,\ldots$.  

\begin{figure}[h]
\centering
	\scalebox{0.8}{\includegraphics{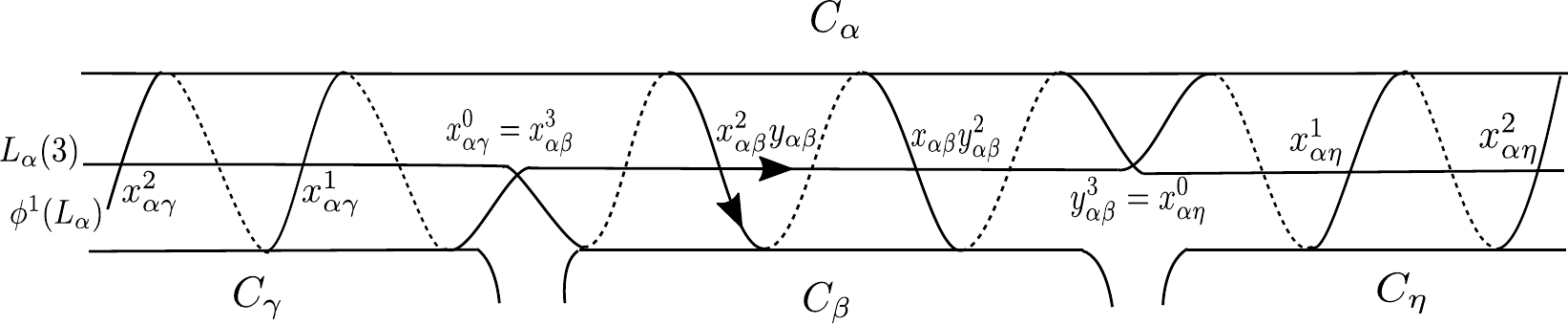}}
	\caption{Generators of $CW^*(L_\alpha, L_\alpha(3))$ assuming $n_{\alpha\beta}=1$.}
	\label{C1}
\end{figure}

\begin{figure}[h]
\centering
	\scalebox{0.8}{\includegraphics{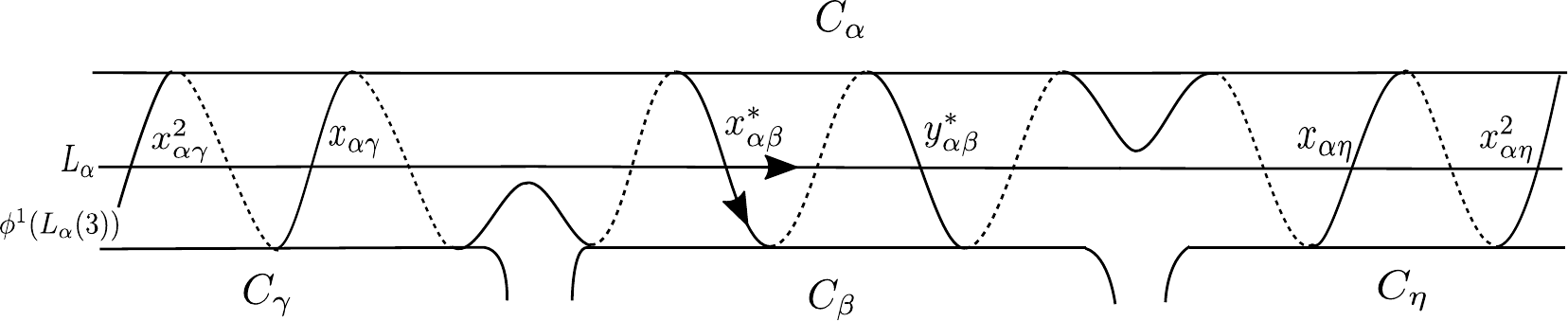}}
	
	\caption{Generators of $CW^*(L_\alpha(3), L_\alpha)$ assuming $n_{\alpha\beta}=1$. The generators $x_{\alpha\beta}^*$ and $y_{\alpha\beta}^*$ are of degree 1.  Intersection points on unbounded edges are always of degree 0. }
	\label{C2}
\end{figure}

As for the products $CW^*(L_\alpha(k),L_\alpha(l))\otimes CW^*(L_\alpha(j), L_\alpha(k))\to CW^*(L_\alpha(j),L_\alpha(l))$, there are no holomorphic triangles with vertices in more than one edge.   The product $p_2\cdot p_1=q$ needs to satisfy $\deg (p_1)+\deg(p_2)=\deg (q) (\text{mod }2)$ where the grading is by $\mathbb Z_2$, so there are just the three possibilities listed below.  The formulas for the products are obtained by counting holomorphic triangles.

\noindent $\bullet$  $CW^0(L_\alpha(k),L_\alpha(l))\otimes CW^0(L_\alpha(j), L_\alpha(k))\to CW^0(L_\alpha(j),L_\alpha(l))$, with $x_{\alpha\beta}^{p'}y_{\alpha\beta}^{q'}\cdot x_{\alpha\beta}^py_{\alpha\beta}^q=x_{\alpha\beta}^{p+p'}y_{\alpha\beta}^{q+q'}$.  For generators on a bounded edge, such a product is only possible when $j\leq k \leq l$.  

\noindent $\bullet$ $CW^0(L_\alpha(k),L_\alpha(l))\otimes CW^1(L_\alpha(j), L_\alpha(k))\to CW^1(L_\alpha(j),L_\alpha(l))$, $k<l<j$, with $x_{\alpha\beta}^{p'}y_{\alpha\beta}^{q'}\cdot (x_{\alpha\beta}^py_{\alpha\beta}^q)^*=(x_{\alpha\beta}^{p-p'}y_{\alpha\beta}^{p-p'})^*$.

\noindent $\bullet$ $CW^1(L_\alpha(k),L_\alpha(l))\otimes CW^0(L_\alpha(j), L_\alpha(k))\to CW^1(L_\alpha(j),L_\alpha(l))$, $l<j<k$, with $(x_{\alpha\beta}^{p'}y_{\alpha\beta}^{q'})^*\cdot x_{\alpha\beta}^py_{\alpha\beta}^q=(x_{\alpha\beta}^{p'-p}y_{\alpha\beta}^{p'-p})^*$.

\subsubsection {Lagrangians from two adjacent components.}  In this section, we focus on two adjacent components $C_\alpha$ and $C_\beta$. Recall that on the edge $e_{\alpha\beta}$, $L_{\alpha}(k)$ and $L_\beta(l)$ have the opposite orientations.  When the edge $e_{\alpha\beta}$ is unbounded, we do the same as in \cite{AAEKO13}; we label the infinite sequence of odd degree intersection points by $x^j_{\alpha;\beta}$, for $j=0,1,2\ldots$. 

On a bounded edge $e_{\alpha\beta}$, if $k<l$, then $L_\alpha(k)$ and $L_\beta(l)$ intersect at $n_{\alpha\beta}(l-k)+d_{\alpha,\beta}+1$ points of degree 1, and we label them by $x_{\alpha;\beta}^{n_{\alpha\beta}(l-k)+d_{\alpha,\beta}-j}y_{\alpha;\beta}^j$ for $j=0,\ldots n_{\alpha\beta}(l-k)+d_{\alpha,\beta}$.  For $k>l$, $L_\alpha(k)$ and $L_\beta(l)$ intersect at $n_{\alpha\beta}(k-l)+d_{\alpha,\beta}+1$ points of degree 0, and we label them by $\left(x_{\alpha;\beta}^{n_{\alpha\beta}(k-l)+2+d_{\alpha,\beta}-j}y_{\alpha;\beta}^j\right)^*$ for $j=0,\ldots, n_{\alpha\beta}(k-l)+d_{\alpha,\beta}$.   See Figure \ref{C3}.

\begin{figure}[h]
\centering
	\scalebox{0.6}{\includegraphics{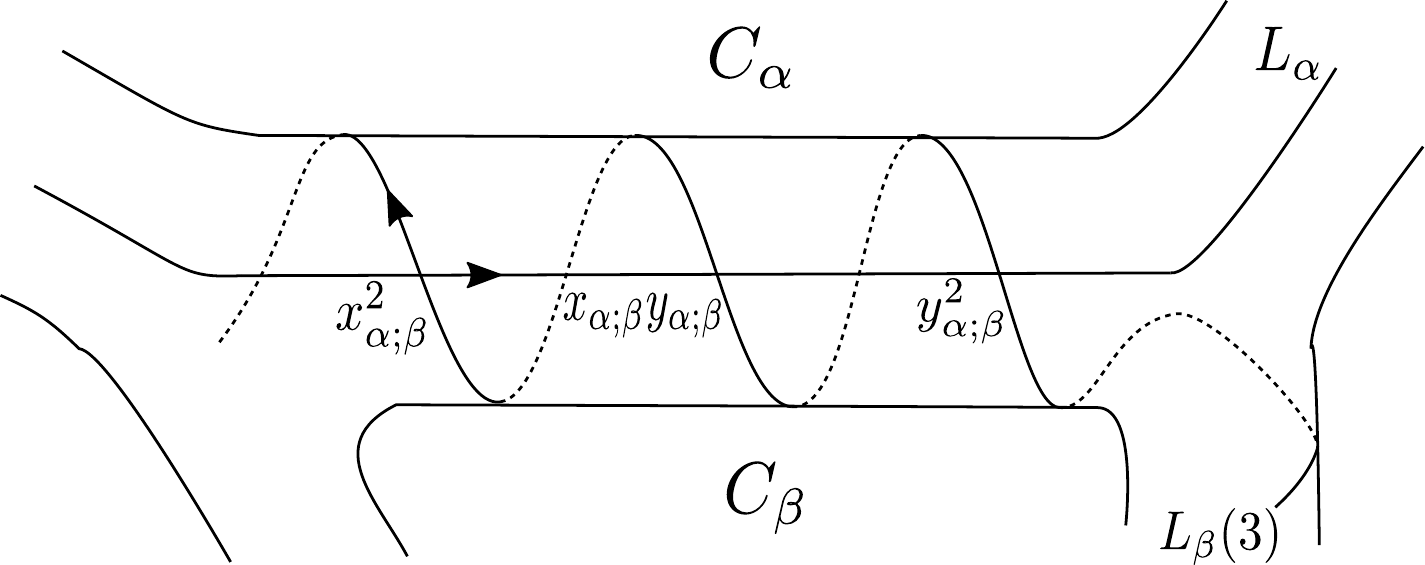}}
	\caption{Intersections of $L_\alpha$ and $L_\beta(3)$ on a bounded edge $e_{\alpha\beta}$; assuming $n_{\alpha\beta}=1$, $d_{\alpha,\beta}=-1$.}
	\label{C3}
\end{figure}

 Assuming $j<k<l$, we have the following products: 

\noindent $\bullet$  $CW^1(L_\alpha(k),L_\beta(l))\otimes CW^0(L_{\alpha}(j),L_\alpha(k))\to CW^1(L_\alpha(j), L_\beta(l))$, with  $x_{\alpha;\beta}^{p'}y_{\alpha;\beta}^{q'}\cdot x_{\alpha\beta}^{p}y_{\alpha\beta}^{q} = x_{\alpha;\beta}^{p+p'}y_{\alpha;\beta}^{q+q'}$.

\noindent $\bullet$ $CW^0(L_{\beta}(k),L_\beta(l))\otimes CW^1(L_\alpha(j),L_\beta(k)) \to CW^1(L_\alpha(j), L_\beta(l))$, with $x_{\alpha\beta}^{p'}y_{\alpha\beta}^{q'}\cdot x_{\alpha;\beta}^{p}y_{\alpha\beta}^{q} = x_{\alpha;\beta}^{p+p'}y_{\alpha;\beta}^{q+q'}$.  

\noindent $\bullet$ $CW^1(L_\beta(k), L_\alpha(l)) \otimes CW^1(L_\alpha(j),L_\beta(k))\to CW^0(L_\alpha(j),L_\alpha(l))$, with $x_{\beta;\alpha}^{p'} y_{\beta;\alpha}^{q'}\circ x_{\alpha;\beta}^py_{\alpha;\beta}^q=x_{\alpha\beta}^{p+p'+1} y_{\alpha\beta}^{q+q'+1}$.

\noindent There are three other products by interchanging $\alpha$ and $\beta$ above.  There are additional products of a similar nature when we change the assumption of $j<k<l$ to other orders.

\subsubsection {Lagrangians from three adjacent components.}
\noindent Suppose a holomorphic triangle is bounded by Lagrangians from three different components $C_\alpha, C_\beta, C_\gamma$.  The boundary of this triangle can be decomposed into three arcs, each arc only winding around one leg of the pair of pants.  Pick a point in the center of the triangle, then joining these three arcs to this center point form three loops $\gamma_1$, $\gamma_2$, and $\gamma_3$.  The fundamental group of the pair of pants is a free group on two generators $a$ and $b$, where $a$ is a loop around one leg of the pair of pants, $b$ is a loop around another leg, and $ab$ is a loop around the third leg.  Hence $\gamma_1=a^p$, $\gamma_2=b^q$, and $\gamma_3=(ab)^r$.  However $\gamma_1\cdot \gamma_2=\gamma_3$, hence we must have $p=q=r=1$ or $p=q=r=0$.   Hence, the vertices of such a holomorphic triangle in a pair-of-pants cannot be an intersection point that is further down any edge than the first one.

\subsection{Category of singularities of the Landau-Ginzburg mirror}  We describe the triangulated category of singularities $D_{sg}(D)$ of $D=W^{-1}(0)$.   
As in \cite{Or04}, 
\[\Hom_{D_{sg}}(\Os_{D_\alpha}(k), \Os_{D_\beta}(l)[n])\cong \Ext_D^n(\Os_{D_\alpha}(k), \Os_{D_\beta}(l)),\] for $n>\dim D=2$.  
The morphisms only depend on whether the degree $n$ is even or odd, so calculating Ext's for $n>2$ is enough to determine the morphisms.

Assuming $D_{\alpha\beta}:=D_\alpha\cap D_\beta\neq \emptyset$, we get a 2-periodic resolution of $\Os_{D_\alpha}(k)$ on $D$ by locally free sheaves, 
\[\{\cdots \to \Os_D(k) \stackrel{}{\rightarrow} \Os_D(-D_\alpha)(k) \stackrel{}{\rightarrow} \Os_D(k)\}\to\Os_{D_\alpha}(k)\to 0.\]
We can replace $\Os_{D_\alpha}(k)$ with the above 2-periodic complex, and $\sHom(-, \Os_{D_\beta}(l))$ maps this complex to
\[0\to \Os_{D_\beta}(l-k)\stackrel{\phi_{\alpha\beta}}{\rightarrow}\Os_{D_\beta} (D_\alpha)(l-k)\stackrel{\psi_{\alpha\beta}}{\rightarrow}\Os_{D_\beta}(l-k)\to \cdots.\]  We get $\Ext_D^n(\Os_{D_\alpha}(k), \Os_{D_\beta}(l))$ as the hypercohomology groups of this complex.

\subsubsection  {Morphisms $\Hom_{D_{sg}}(\Os_{D_\alpha}(k), \Os_{D_\alpha}(l)[n])$.} 
In this case $\phi_{\alpha\alpha}=0$, and $\psi_{\alpha\alpha}$ is isomorphic to the injective map $\Os_{D_\alpha}(-H_{\alpha})(l-k)\to \Os_{D_{\alpha}}(l-k)$ with its cokernel being $\Os_{H_{\alpha}}(l-k)$, where $H_\alpha=\bigcup_\beta D_{\alpha\beta}$.  
Hence for $n>2$, $\Ext_D^{even}(\Os_{D_\alpha}(k), \Os_{D_\alpha}(l))=H^0(\Os_{H_\alpha}(l-k))$ and $\Ext_D^{odd}(\Os_{D_\alpha}(k), \Os_{D_\alpha}(l))=H^1(\Os_{H_\alpha}(l-k))$. So \[\Hom_{D_{sg}}(\Os_{D_\alpha}(k),\Os_{D_\alpha}(l)[*])\cong H^*(\Os_{H_\alpha}(l-k)).\]

First, suppose $D_\alpha$ is unbounded.  If $H_\alpha=\mathbb C\cup \mathbb C$, then $H^0(\Os_{H_\alpha}(l-k))=\mathbb C[x]\oplus \mathbb C[y]/(x^0=y^0)$ and $H^1(\Os_{H_\alpha}(l-k))=0$.  Otherwise, $H_\alpha$ consists of two $\mathbb C^1$'s connected by one or more $\mathbb P^1$'s, as illustrated in Fig. \ref{divisor}.  We denote the two $\mathbb C^1$'s by $C'$ and $C''$, and the connecting sequence of $\mathbb P^1$'s by $C_1,...,C_m$, and $C=\bigcup_{j=1}^m C_j$.  Then $H_\alpha=C\cup C'\cup C''$ and there is an exact sequence 
\[0\to H^0(\Os_{C\cup C'\cup C''}(l-k))\stackrel{\cong}{\longrightarrow} H^{0}(\Os_C(l-k))\oplus H^0(\Os_{C'\cup C''}) \twoheadrightarrow H^0(\Os_{p\cup p''})\hspace{1in}\]
\[\hspace{2.2in}\stackrel{0}{\longrightarrow} H^1(\Os_{C\cup C'\cup C''}(l-k))\stackrel{\cong}{\longrightarrow} H^1(\Os_C(l-k))\to 0.\]

\begin{figure}[h]
\centering
	\scalebox{0.7}{\includegraphics{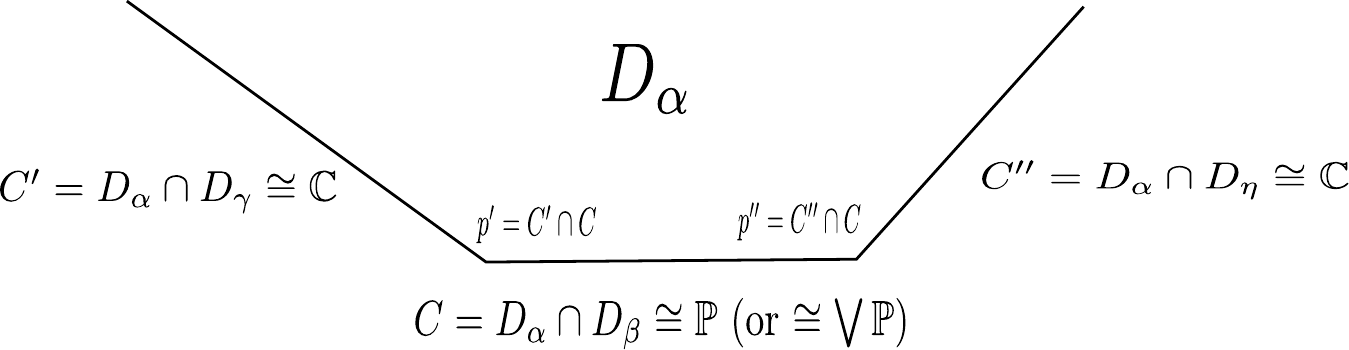}}
	\caption{$H_\alpha=C\cup C'\cup C''$}
	\label{divisor}	
\end{figure} 

\noindent When $l-k\geq 0$, \[H^0(\Os_C(l-k))\cong \bigoplus_{j=1}^m H^0(\Os_{C_j}(l-k))/\sim\  \cong \bigoplus_{j=1}^m \mathbb C[x_j: y_j]^{n_j(l-k)}/\sim,\]
\[H^1(\Os_C(l-k))=0,\] 
where $(x_j:y_j)$ are homogenous coordinates.  The equivalence relation above identifies $y_j^{n_j(l-k)}\in H^0(\Os_{C_j}(l-k))$ and $x_{j+1}^{n_{j+1}(l-k)}\in H^0(\Os_{C_{j+1}}(l-k))$ with each other and with the section of $\Os_C(l-k)$ obtained by considering these two polynomials together on $C_j\cup C_{j+1}$. In this notation, if $C_j=D_\alpha\cap D_{\beta_j}$, then we are identifying $x_j\sim x_{\alpha\beta_j}$, $y_j \sim y_{\alpha\beta_j}$, and $n_j\sim n_{\alpha\beta_j}$.  Suppose $C'=D_\alpha\cap D_\gamma$ and $C''=D_\alpha\cap D_\eta$, then $H^0(\Os_{C'\cup C''})=\mathbb C[x_{\alpha\gamma}]\oplus \mathbb C[x_{\alpha\eta}].$

\noindent When $l-k<0$, $H^0(\Os_C(l-k))=0$, 
$H^0(\Os_{C'\cup C''})=\mathbb C[x_{\alpha\gamma}]\oplus \mathbb C[x_{\alpha\eta}]$,  and
\[H^1(\Os_{C_j}(l-k))=\left(\mathbb C[x_j: y_j]^{n_j(k-l)-2}\right)^*.\]
(When $n_j(l-k)=-1$, $H^1(\Os_{C_j}(l-k))=0$.) Then $H^1(\Os_{C}(l-k))$ is a direct sum of all the $H^1(\Os_{C_j}(l-k))$ for $j=1,\ldots, m$ plus $m-1$ extra generators associated to each vertex $C_j\cap C_{j+1}$. 
Note by Serre duality, $H^1(\mathbb P^1,\Os(n(l-k)))\cong H^0(\mathbb P^1,\Os(n(k-l)-2))^*$, with $(x_j^py_j^q)^*=\frac{dx}{x_j^{p+1}y_j^q}=-\frac{dy}{x^py^{q+1}}$, where $p+q=n(k-l)-2$.

 Comparing this with the wrapped Floer cohomology, we get \[\Hom^*(L_\alpha(k),L_\alpha(l))\cong \Hom_{D_{sg}}(\Os_{D_\alpha}(k), \Os_{D_\alpha}(l)[*]).\]

\subsubsection{Morphisms $\Hom_{D_{sg}}(\Os_{D_\alpha}(k), \Os_{D_\beta}(l)[n])$, $\alpha\neq \beta$.}  In this case $\psi_{\alpha\beta}=0$, and $\phi_{\alpha\beta}$ is isomorphic to the injective map $\Os_{D_{\beta}}(l-k)\to\Os_{D_\beta}(D_\alpha)(l-k)$ with its cokernel being $\Os_{D_{\alpha\beta}}(D_\alpha)(l-k)$.   Hence for $n>2$, $\Ext_D^{even}(\Os_{D_\alpha}(k), \Os_{D_\beta}(l))=H^1(\Os_{D_{\alpha\beta}}(D_\alpha)(l-k))$ and $\Ext_D^{odd}(\Os_{D_\alpha}(k), \Os_{D_\beta}(l))=H^0(\Os_{D_{\alpha\beta}}(D_\alpha)(l-k))$.  So \[\Hom_{D_{sg}}(\Os_{D_\alpha}(k),\Os_{D_\beta}(l)[*+1])\cong H^*(\Os_{D_{\alpha\beta}}(D_\alpha)(l-k)).\]  
\indent If $D_{\alpha\beta}$ is unbounded, then $D_{\alpha\beta}=\mathbb C[x]$ and $H^*(\Os_{D_{\alpha\beta}}(D_\alpha)(l-k))\cong H^0(\Os_{\mathbb C})\cong \mathbb C$.  If $D_{\alpha\beta}$ is bounded and  $D_{\alpha\beta} \cong \mathbb P^1$, then $\Os_{D_{\alpha\beta}}(D_\alpha)\cong \Os_{\mathbb P^1}(d_{\alpha,\beta})$, where $d_{\alpha,\beta}=\deg\Os(D_\alpha)|_{D_{\alpha\beta}}$,  and \[H^*(\Os_{D_{\alpha\beta}}(D_\alpha)(l-k))=H^*(\Os_{\mathbb P^1}(n_{\alpha\beta}(l-k)+d_{\alpha,\beta})).\]

  \noindent For $n_{\alpha\beta}(l-k)\geq -d_{\alpha,\beta}$, 
\[H^0(\Os_{\mathbb P^1}(n_{\alpha\beta}(l-k)+d_{\alpha,\beta}))\cong \mathbb C[x_{\alpha;\beta}:y_{\alpha;\beta}]^{n_{\alpha\beta}(l-k)+d_{\alpha,\beta}},\]
\[H^1(\Os_{\mathbb P^1}(n_{\alpha\beta}(l-k)+d_{\alpha,\beta}))=0.\]
For $n_{\alpha\beta}(l-k)=-d_{\alpha,\beta}-1$, \[H^0(\Os_{\mathbb P^1}(n_{\alpha\beta}(l-k)+d_{\alpha,\beta}))=H^1(\Os_{\mathbb P^1}(n_{\alpha\beta}(l-k)+d_{\alpha,\beta}))=0.\]  
For $n_{\alpha\beta}(l-k)<-d_{\alpha,\beta}-1$, 
\[H^0(\Os_{\mathbb P^1}(n_{\alpha\beta}(l-k)+d_{\alpha,\beta}))=0,\]
\[H^1(\Os_{\mathbb P^1}(n_{\alpha\beta}(l-k)+d_{\alpha,\beta}))=\left(\mathbb C[x_{\alpha;\beta}:y_{\alpha;\beta}]^{n_{\alpha\beta}(k-l)+d_{\alpha,\beta}+2}\right)^*.\]

Comparing with wrapped Floer cohomology, we get 
\[\Hom^*(L_\alpha(k),L_\beta(l))\cong \Hom_{D_{sg}}(\Os_{D_\alpha}(k), \Os_{D_\beta}(l)[*]).\]

\end{document}